\documentclass[11pt]{article}

\usepackage[utf8]{inputenc}
\usepackage[T1]{fontenc}
\usepackage{textcomp}
\usepackage{amsmath,amssymb,amsthm}
\usepackage{lmodern}
\usepackage[a4paper, margin=0.8in]{geometry}
\usepackage{xcolor, pict2e}
\usepackage{microtype}
\usepackage{listings}
\usepackage{moreverb}
\usepackage{hyperref}
\usepackage[sans]{dsfont}
\usepackage[font={sf,small}, labelfont={sf,bf}, margin=1cm]{caption}
\usepackage[makeroom]{cancel}
\usepackage{centernot}
\usepackage{tikz}
\usepackage{tikz-3dplot}
\usetikzlibrary{calc}

\usepackage{import}
\usepackage{xifthen}
\usepackage{pdfpages}
\usepackage{transparent}
\usepackage{chngcntr}
\counterwithin{figure}{section}
\usepackage{subcaption}

\definecolor{auburn}{rgb}{0.43, 0.21, 0.1}

\usepackage{hyperref}
\hypersetup{
    colorlinks=true,
    linkcolor=blue,
    citecolor=auburn,
    urlcolor=blue,
    pdfborder={0 0 0}
}

\newtheorem*{theorem*}{Theorem}
\newtheorem{theorem}{Theorem}[section]
\newtheorem{proposition}[theorem]{Proposition}
\newtheorem{lemma}[theorem]{Lemma}
\newtheorem{corollary}[theorem]{Corollary}
\newtheorem{notation}[theorem]{Notation}
\newtheorem{definition}[theorem]{Definition}
\newtheorem{claim}[theorem]{Claim}

\newtheorem{conjecture}[theorem]{Conjecture}
\theoremstyle{remark}
\newtheorem{remark}[theorem]{Remark}

\newcommand{\R}{\mathbb{R}}
\newcommand{\N}{\mathbb{N}}
\newcommand{\Z}{\mathbb{Z}}

\newcommand{\D}{\mathbb{D}}
\renewcommand{\S}{\mathbb{S}}

\newcommand{\Cc}{\mathcal{C}}

\newcommand{\Lc}{\mathcal{L}}
\newcommand{\Mc}{\mathcal{M}}

\newcommand{\Pc}{\mathcal{P}}
\newcommand{\Fs}{\mathsf{F}}
\newcommand{\Ir}{\mathrm{I}}

\newcommand{\Lf}{\mathfrak{L}}
\newcommand{\Pf}{\mathfrak{P}}

\newcommand{\Expect}[1]{\mathbb{E} \left[ #1 \right] }

\newcommand{\Prob}[1]{\mathbb{P} \left( #1 \right) }

\renewcommand{\P}{\mathbb{P}}
\newcommand{\E}{\mathbb{E}}
\newcommand{\Pds}{\mathds{P}}

\newcommand{\floor}[1]{\left\lfloor #1 \right\rfloor}

\newcommand{\indic}[1]{ \mathbf{1}_{ \left\{ #1 \right\} } }
\newcommand{\eps}{\varepsilon}

\renewcommand{\d}{\mathrm{d}}

\newcommand{\wt}{\widetilde}
\newcommand{\wh}{\widehat}

\def \loopmeasure{\mu^{\rm loop}}

\def \Bes{{\rm Bes}}
\def \rad{{\rm rad}}
\def \ang{{\rm ang}}
\def \bub{{\rm bub}}
\def \cross{{\rm Cross}}
\def \scale{{\rm scale}}
\def \diam{{\rm diam}}
\def \Seed{{\rm Seed}}
\def \loc{{\rm LocUniq}}
\def \b2b{{\rm B2B}}
\def \tr{{\rm tr}}

\newcommand{\drawcube}[4][]{
\def\h{.5} 

\coordinate (A) at (#2- \h, #3- \h, #4- \h);
\coordinate (B) at (#2+ \h, #3- \h, #4- \h);
\coordinate (C) at (#2+ \h, #3+ \h, #4- \h);
\coordinate (D) at (#2- \h, #3+ \h, #4- \h);

\coordinate (E) at (#2- \h, #3- \h, #4+ \h);
\coordinate (F) at (#2+ \h, #3- \h, #4+ \h);
\coordinate (G) at (#2+ \h, #3+ \h, #4+ \h);
\coordinate (H) at (#2- \h, #3+ \h, #4+ \h);

\fill[fill=gray!40,fill opacity=0.6] (A)--(B)--(C)--(D)--cycle;
\fill[fill=gray!40,fill opacity=0.6] (A)--(B)--(F)--(E)--cycle;
\fill[fill=gray!40,fill opacity=0.6] (A)--(D)--(H)--(E)--cycle;

\fill[fill=gray!40,fill opacity=0.6] (E)--(F)--(G)--(H)--cycle;
\fill[fill=gray!40,fill opacity=0.6] (B)--(C)--(G)--(F)--cycle;
\fill[fill=gray!40,fill opacity=0.6] (D)--(C)--(G)--(H)--cycle;

\draw[black] (A)--(B)--(F)--(E)--cycle;

\draw[black] (E)--(F)--(G)--(H)--cycle;
\draw[black] (B)--(C)--(G)--(F)--cycle;
}

\title{Three-dimensional Brownian loop soup clusters}
\author{Antoine Jego\thanks{CEREMADE, CNRS, Université Paris-Dauphine, PSL Research University, Place de Lattre de Tassigny, 75016 Paris, France; antoinejego@hotmail.fr} \and Titus Lupu\thanks{Sorbonne Université and Université Paris Cité, CNRS, Laboratoire de Probabilités, Statistique et Modélisation, 75005 Paris, France; titus.lupu@sorbonne-universite.fr}}
\date {\today}

\numberwithin{equation}{section}

\begin{document}

\maketitle

\abstract{
We study Brownian loop soup clusters in $\R^3$ for an arbitrary intensity $\alpha>0$. We show the existence of a phase transition for the presence of unbounded clusters and study its basic properties. In particular, we show that, when $\alpha$ is sufficiently large, almost surely all the loops are connected into a single cluster. Such a phenomenon is not observed in discrete percolation-type models. In addition, we prove the existence of a one-arm exponent and compare the clusters with the finite-range system obtained by imposing lower and upper bounds on the diameter of the loops.

Finally, we provide a toolbox concerning the Brownian loop measure in $\R^d$, $d \ge 3$. In particular, we derive decomposition formulas by rerooting the loops in specific ways and show that the loop measure is conformally invariant, generalising results of \cite{Lupu18} in dimension 1 and \cite{Lawler04} in dimension~2.
}

\setcounter{tocdepth}{1}
\tableofcontents


\section{Introduction}

As introduced by Lawler and Werner \cite{Lawler04}, the Brownian loop measure in the $d$-dimensional Euclidean space $\R^d$ is defined by
\begin{equation}
\label{E:loopmeasure}
\loopmeasure = \int_{\R^d} \d x \int_0^\infty \frac{\d t}{t} \frac{1}{(2\pi t)^{d/2}} \P^{x,x;t},
\end{equation}
where $\P^{x,x;t}$ is the law of a Brownian bridge of duration $t$, from $x$ to $x$. In our convention, Brownian motion has generator $\frac12 \Delta$. The factor $t^{-1-d/2}$ in \eqref{E:loopmeasure} makes the measure scale invariant.
For a domain $D \subset \R^d$, the loop measure in $D$ is defined by restricting the whole-space measure to loops which remain in $D$:
\begin{equation}
    \loopmeasure_D(\d \wp) = \mathbf{1}_{\wp \subset D} \loopmeasure(\d \wp).
\end{equation}
The Brownian loop soup $\Lc_D^\alpha$ with intensity parameter $\alpha>0$ in $D$ is a random collection of loops distributed as a Poisson point process with intensity measure $\alpha \loopmeasure_D$.
Due to the divergence of the integral in \eqref{E:loopmeasure} at $t=0$, $\Lc_D^\alpha$ contains almost surely infinitely many small loops.

In this article we will be primally interested in the \emph{clusters} of the Brownian loop soup which we define below, specifically in $\R^3$. Indeed, two independent Brownian trajectories in $\R^3$ intersect each other with positive probability \cite[Theorem 9.22]{morters2010brownian} and therefore the Brownian loops in the 3D loop soup form nontrivial clusters.
In dimension 2, the clusters of the Brownian loop soup have been already intensively studied, in particular in relation to conformal invariance and Schramm--Loewner Evolution processes (SLE) \cite{SheffieldWernerCLE}.
In contrast, two independent Brownian trajectories in dimensions $d\ge4$ with different starting points almost surely do not intersect. Therefore the clusters are trivial in these dimensions, each containing one single Brownian loop.
In Section \ref{SS:literature}, more references to the literature are given, in particular in connection with discrete models.

\begin{definition}[Clusters]
Let $\Lc$ be a collection of (Brownian-type) trajectories and $\wp,\wp' \in \Lc$. By definition, $\wp ,\wp'$ intersect each other, which we denote by $\wp\cap\wp'\neq\varnothing$, if the sets of points visited by $\wp$ and $\wp'$ intersect each other.
The trajectories $\wp$ and $\wp'$ will be said to belong to the same \emph{cluster} of $\Lc$ if there exists a finite chain $\wp_1, \dots, \wp_n\in \Lc$ with $\wp_1=\wp$ and $\wp_n=\wp'$ and such that for all $i=1, \dots, n-1$, $\wp_i \cap \wp_{i+1} \neq \varnothing$.

If $A,B$ are two subsets of $\R^d$, we will denote by $\{ A \overset{\Lc}{\longleftrightarrow} B\}$ the event that there exists a cluster of $\Lc$ which intersects both $A$ and $B$.
\end{definition}

As already alluded to, a major motivation for the study of Brownian loop soup clusters in dimension 2 was the conformal invariance of the model and its relation to SLE curves, which appear as scaling limits of interfaces in 2D models of statistical physics at criticality \cite{zbMATH05204617}. As we will prove in this paper, the conformal invariance of Brownian loop soups is not specific to the 2D case and holds in any dimension. However, the group of conformal mappings is not as rich when $d\ge 3$: it is generated by translations, rotations, scalings and polar inversions \cite{monge}. Moreover, it is believed by physicists that the scaling limits of models of statistical physics (such as percolation and Ising) at criticality are conformal invariant in all dimensions \cite{POLCHINSKI1988226,zbMATH06398901}.
While the clusters of the 3D Brownian loop soup do not necessarily correspond to the scaling limits of 3D critical statistical models, they provide a model of random fractal ``sponges'' which exhibit conformal invariance by construction. This symmetry constitutes a motivation for the study of this model.


\smallskip
\noindent\textbf{Main difficulties, phenomenology and contribution.}
The first difficulty comes from the possibility that clusters which realise long crossings become more and more one-dimensional.
Contrary to the planar case, 3D Brownian motion avoids one-dimensional curves and so attaching new loops to such clusters would then become extremely costly.
Moreover, unlike in the discrete setting where the mesh size plays the role of the smallest possible scale, in the continuum, loops can be arbitrarily small. It is therefore possible in principle that infinitely many clusters intersect a given ball.
For the same reason, there is no scale where desired estimates must clearly hold and can then be propagated by induction on scales and, in fact, Brownian loop soup is scale invariant.

In this continuum long-range percolation model, some phenomenology differs from what is observed in the discrete setting. For instance, we show that there is a sufficiently supercritical phase where all the Brownian loops form one single (necessarily unbounded) cluster: every loop is connected to every other. An analogous phenomenon has already been observed in dimension 2 \cite{SheffieldWernerCLE}. In 3D, it is still however an open problem to show that this uniqueness phase coincides with the whole supercritical regime.

In addition to unravelling the subtleties of the model, we need to combine tools coming from two different areas: percolation theory and the study of Brownian motion/Brownian loop measure. In particular, we provide a robust toolbox which produces precise estimates on the Brownian loop measure which are key inputs in geometrical percolation-type constructions.

\medskip

We describe our main results in Section \ref{SS:main_results} and compare them with the existing literature in Section~\ref{SS:literature}.

\subsection{Main results}\label{SS:main_results}

For $R > 0$, a domain $D\subset\R^3$ and $\alpha>0$, let $\Lc_{D,\ge R}^\alpha$ be the subset of $\Lc_D^\alpha$ consisting of loops of diameter at least $R$. For $D = \R^3$, if there is no ambiguity, we will simply write $\Lc^\alpha=\Lc^\alpha_{\R^3}$ and $\Lc_{\ge R}^\alpha = \Lc^\alpha_{\R^3,\ge R}$.
We introduce the following parameters corresponding to potentially different thresholds:
\begin{align}
\label{E:def_alpha_c}
\text{critical:} \hspace{49pt} \quad
\alpha_c & = \inf \{ \alpha > 0: \P(\Lc^\alpha \text{ contains an unbounded cluster}) > 0 \}, \\
\label{E:def_alpha_u}
\text{uniqueness:} \hspace{32pt} \quad
\alpha_u & = \inf \{ \alpha > 0: \P(\Lc^\alpha \text{ contains a unique unbounded cluster}) > 0 \}, \\
\label{E:def_alpha_b2b}
\text{ball-to-ball:} \hspace{20pt} \quad
\alpha_\b2b & = \inf \{ \alpha >0: \P(\forall R>1 : B(0,1) \overset{\Lc^\alpha}{\longleftrightarrow} \partial B(0,R)) >0 \}, \\
\label{E:def_alpha_tr}
\text{truncated:} \hspace{34pt} \quad
\alpha_\tr & = \{ \alpha > 0: \P(\Lc_{\ge1}^\alpha \text{ contains an unbounded cluster}) > 0 \}.
\end{align}
We will prove in Lemma \ref{L:I3} below that the uniqueness property in \eqref{E:def_alpha_u} is monotone in $\alpha$. For the other three properties above, this clearly follows from the monotonicity of $\Lc^\alpha$ and $\Lc^\alpha_{\ge 1}$. We will also show in Lemma \ref{L:zero_one} that the above events have probability either zero or one. For $\alpha_c$, $\alpha_u$ and $\alpha_\tr$, this follows by ergodicity (Lemma \ref{L:ergodicity}), but for $\alpha_\b2b$ this requires an argument.

\begin{theorem}[Phase transition]\label{T:phase_transition}
We have $0<\alpha_\b2b\le\alpha_c\le\alpha_u\le\alpha_\tr<\infty$.
\end{theorem}

\begin{remark}
    In fact we show a stronger statement concerning the inclusion of the intervals appearing in \eqref{E:def_alpha_c}-\eqref{E:def_alpha_tr}. For instance, if $\alpha >0$ is such that $\Lc_{\ge1}^\alpha$ contains an unbounded cluster with positive probability, then $\Lc^\alpha$ contains a unique unbounded cluster with positive probability.
\end{remark}

The inequalities $\alpha_\b2b \le \alpha_c\le\alpha_u$ are trivial, whereas the inequality $\alpha_u \le \alpha_\tr$ requires a justification.
The existence of the phase transition, i.e. that $\alpha_\b2b>0$ and $\alpha_\tr < \infty$, follows similar arguments as in the planar case \cite{SheffieldWernerCLE}.
To show that $\alpha_\tr < \infty$, we build concrete connections and compare the model with a one-dependent percolation model whose success probability is close to 1. We then conclude by \cite{MR1428500}. To show that $\alpha_\b2b >0$, we stochastically dominate Brownian-loop percolation by a variant of Mandelbrot's fractal percolation and then use \cite{Chayes91}.
The proof of Theorem~\ref{T:phase_transition} is given in Section~\ref{S:first_perco}. This section also contains basic properties of the structure of clusters depending on the value of $\alpha$. For instance, we prove in Lemma \ref{L:I2_no_bdd} that, when $\alpha >\alpha_c$, a.s. each cluster of $\Lc^\alpha$ is unbounded and dense in $\R^3$ and, either $\Lc^\alpha$ contains a unique cluster a.s., or $\Lc^\alpha$ contains infinitely many clusters a.s.

\medskip

Our second main result concerns the existence of one-arm exponents:

\begin{theorem}[One-arm exponent]\label{T:crossing}
For all $\alpha >0$, there exists $\xi = \xi(\alpha) \ge 0$ such that
\begin{equation}
    \label{E:T_one-arm}
\P \Big( \partial B(1) \overset{\Lc_{\R^3}^\alpha}{\longleftrightarrow} \partial B(r) \Big) = r^{\xi + o(1)},
\qquad \text{as } r \to 0.
\end{equation}
Moreover, the following bounds are satisfied:
\begin{itemize}
    \item for all $\alpha >0$, $\xi < 1$;
    \item $\xi > 0$ if, and only if, $\alpha$ is such that the left hand side of \eqref{E:T_one-arm} vanishes as $r \to 0$ (in particular when $\alpha < \alpha_\b2b$).
\end{itemize}
\end{theorem}

As we will see in \eqref{E:L_estimate_crossing2}, the probability that a loop intersects both $\partial B(1)$ and $\partial B(r)$ behaves asymptotically like $\alpha r$. The fact that $\xi < 1$ shows that realising this crossing with a cluster is much easier.
Note that, by scale invariance of $\Lc^\alpha_{\R^3}$, the probability on the left hand side of \eqref{E:T_one-arm} is invariant under $r \to 1/r$. It thus also describes the decay of the probability of large crossings.
The proof of Theorem~\ref{T:crossing} is contained in Section \ref{S:one-arm}.

\medskip

Finally, we introduce two-sided restrictions on the diameter of loops and compare the structure of clusters of $\Lc^\alpha$ with those restricted loops. The restricted loops are much closer to Bernoulli percolation on the lattice.
For $R>1$, let $\Lc_{1,R}^\alpha$ be the subset of $\Lc^\alpha$ consisting of all loops with diameter in $[1,R]$ and let
\begin{equation}\label{E:def_alpha_R}
R\text{-truncated:} \hspace{15pt} \quad
\alpha_\tr^R = \inf \{ \alpha> 0: \P( \Lc_{1,R}^\alpha \text{ contains an unbounded cluster})>0 \}.
\end{equation}

\begin{theorem}[Truncated model and local uniqueness]\label{T:JR}
For all $R>1$, $0<\alpha_\tr^R < \infty$ and
\begin{equation}
    \label{E:TJR}
    \text{for all } R'>R>1, \qquad \alpha_\tr^R > \alpha_\tr^{R'}.
\end{equation}
Moreover, when $\Lc_{\ge 1}^\alpha$ percolates, there must exist $R >1$ such that $\Lc^\alpha_{1,R}$ percolates. Also, the supercritical phase of $\Lc_{\ge 1}^\alpha$ coincides with the local uniqueness phase of $\Lc^\alpha$:
\begin{equation}
    \label{E:T_locuniq}
\alpha_\tr = \lim_{R \to \infty} \alpha_\tr^R = \inf \{ \alpha >0: \forall~D\subset\R^3 \text{ connected and open}, \P(\Lc^\alpha_D \text{ contains a unique cluster}) =1 \}.
\end{equation}
Finally, $\Lc_{\ge 1}^\alpha$ does not percolate when $\alpha=\alpha_\tr$: almost surely it does not contain an unbounded cluster.
\end{theorem}

It should be possible to prove that $\alpha_\tr$ is universal in the sense that it does not depend on the choice of cutoff. For instance, one could consider the loops with duration at least one instead of restricting the diameter.
The proof of Theorem \ref{T:JR}, \eqref{E:TJR} is based on an Aizenman--Grimmett-type argument \cite{MR1116036} and is written in Section \ref{S:AG}. The proof of Theorem \ref{T:JR}, \eqref{E:T_locuniq} is based on a Grimmett--Marstrand-type argument \cite{zbMATH04169828} and is written in Section \ref{S:GM}. The identification of these regimes can be done all the way up to the critical point thanks to the long-range interactions: sprinkling can be done on the diameter of loops, instead of sprinkling the intensity $\alpha$. This eventually yields a proof of absence of percolation at the critical point: \eqref{E:TJR} demonstrates that $\bigcup_{R>1} [\alpha_\tr^R,\infty) = \bigcup_{R>1} (\alpha_\tr^R,\infty) $ is open.

Finally, we believe that all the thresholds defined in \eqref{E:def_alpha_c}-\eqref{E:def_alpha_tr} coincide but this is still an open problem. For instance, showing that $\alpha_c=\alpha_u$ amounts to showing the absence of a phase where there are infinitely many unbounded clusters, each cluster being dense in $\R^3$. A naive Burton--Keane argument cannot be used for this purpose since, almost surely, there are infinitely many small loops in $\Lc^\alpha$ intersecting the boundary of a given ball.

\begin{conjecture}
    $\alpha_c=\alpha_u=\alpha_\tr=\alpha_\b2b$.
\end{conjecture}

\noindent\textbf{Brownian excursions, loop measure and conformal invariance.}
Many of the proofs of this paper use geometric strategies that are standard in percolation theory. However, these proofs have to be coupled with precise estimates on the Brownian loop measure of specific events which can be rather tricky. Using the definition \eqref{E:loopmeasure} of $\loopmeasure$ to obtain the desired estimates is often hopeless, or at least inefficient. The common strategy is to decompose the loop measure in a specific way which depends on each case at hand.
For instance, one can wish to reroot the loop measure at the point of the loop which is closest to a given point (Proposition~\ref{P:decompo_loopmeasure}), or at a double point of the loop (Lemma~\ref{L:double}).
Based on the two-dimensional results of \cite{Lawler04}, this strategy has been used many times in the 2D case but, to the best of our knowledge, is used for the first time in 3D in this article. For this reason, we dedicate Part \ref{PartA} of this article to properties of the Brownian loop measure $\loopmeasure$: we construct excursion measures, bubbles measures and link them to the Brownian loop measure.
We also prove conformal invariance of $\loopmeasure$ in any dimension $d \ge 3$ (Lemma~\ref{L:inversion_loop}).
This generalises properties which are well known in 1D \cite{Lupu18} and 2D \cite{Lawler04}.

\subsection{Literature review}\label{SS:literature}

We now describe some closely related works on the structure of loop soup clusters. As already mentioned, the two-dimensional picture is pretty well understood; see \cite{JLQb,arXiv:2304.03150} and references therein for recent papers on this subject.
The discussion below is restricted to dimensions $d \ge 3$.

\medskip

\noindent\textbf{Random walk loop soup.}
Brownian loop soup has a natural analogue in the discrete setting, called random walk loop soup \cite{LawlerFerreras07RWLoopSoup}.
Contrary to Brownian loop soup clusters, random walk loop soup clusters have a rich structure in dimensions $d \ge 4$.
In the discrete, the mesh size can be thought of as a minimal scale for the loops. So random walk loop soup clusters can be heuristically compared to the clusters of the truncated Brownian loop soup $\Lc_{\ge 1}^\alpha$ where only loops with a diameter at least 1 are kept.
By contrast, the behaviour of clusters of $\Lc^\alpha$ (without truncation) is both phenomenologically different and more difficult to study.

The article \cite{chang2016phase} studies random walk clusters on $\Z^d$ for $d \ge 3$.
They prove existence of a phase transition which corresponds to the fact that $0<\alpha_c<\infty$ using our notation. In 3D, they do not prove existence of a one-arm exponent $\xi$, though they obtain bounds on the probability of the one-arm event. Assuming existence of the exponent, these bounds read as $\xi < 1$ for all $\alpha$, $\xi > 0$ for $\alpha$ small enough and $\xi \to 1$ as $\alpha \to 0$.
We also mention the recent article \cite{arXiv:2508.09291} which obtains precise asymptotics on the one-arm probability when $d \ge 5$.

The article \cite{chang2017supercritical} also studies random walk clusters on $\Z^d$ for $d \ge 3$. Among other things, Chang considers truncations on the number of steps of the walk, closely related to our truncation \eqref{E:Jr} on the diameter. He proves that the critical intensity of the truncated model is strictly decreasing function of the truncation parameter, which can be compared with Theorem \ref{T:JR}, \eqref{E:TJR}. He also proves a statement analogous to the equality $\alpha_\tr = \lim_{R \to \infty} \alpha_\tr^R$ stated in Theorem \ref{T:JR}, \eqref{E:T_locuniq}; see \cite[Section 4]{chang2017supercritical}.

\medskip
\noindent\textbf{Metric graph loop soups and the Gaussian free field.}
At the special value $\alpha =1/2$, one can couple a metric graph loop soup and a metric graph Gaussian free field (GFF) $\varphi$ in such a way that clusters of the loop soup agree with the excursion sets $\{ \varphi \neq 0 \}$ \cite{zbMATH06603570}. This leads to a form of exact solvability and a deeper understanding of the model at $\alpha=1/2$.
The value $\alpha=1/2$ also corresponds the critical point $\alpha_c$ for clusters of the metric graph loop soup \cite{zbMATH07898720}.

The exact value of the one-arm exponent has been predicted in any dimension $d\ge 3$ in \cite{werner2021clusters}, based on the exact expression of the two-point connectivity function \cite{zbMATH06603570}. Since then, this conjecture has been confirmed in a series of works; see \cite{zbMATH07662556, arXiv:2405.17417, zbMATH08035669, arXiv:2406.02397} and the references therein.

\medskip

\textit{
During the final stages of preparation of this manuscript, the article \cite{arXiv:2510.20526} was released on ArXiv. In particular, the authors prove the existence of a one-arm exponent $\xi$ for the Brownian loop soup clusters in $\R^3$ using a different method. Although their proof is written for $\alpha = 1/2$, it should generalise to any value of $\alpha$ yielding an alternative derivation of \eqref{E:T_one-arm}. Moreover, they show that $\xi(1/2) > 1/2$, i.e. that $\xi(1/2)$ is strictly larger than the analogous exponent in the metric graph. The article \cite{arXiv:2510.20526} thus provides a strong evidence that Brownian loop soup clusters do not correspond to the scaling limit of the metric graph clusters when $\alpha =1/2$.
Combining the fact that $\xi(1/2)>0$ and Theorem~\ref{T:crossing} above, one obtains that $\alpha_\b2b \ge 1/2$.
}

\medskip

\noindent\textbf{Organisation.}
This article is divided into two main parts.
Part \ref{PartA} is concerned with general properties of Brownian excursions and the Brownian loop measure in dimensions $d \ge 3$.
Part~\ref{PartB} deals with the Brownian loop soup clusters in 3D and contains in particular the proofs of the main results announced in the introduction.
Before diving into Part \ref{PartA}, we start with a short preliminary section used throughout the article.

\section{Preliminaries}

\subsection{Notations}

Here are a few notations regularly used in this paper:
\begin{itemize}
\item $B(x,R)$: the Euclidean ball $\{y \in \R^d: \|y-x\| < R\}$, where $R>0$ and $x \in \R^d$. We simply write $B(R)$ when $x=0$.
\item $\S^{d-1}$: the unit sphere $\{ \omega \in \R^d: \|\omega\| = 1 \}$ of $\R^d$.
\item $m_{R\S^{d-1}}$: the uniform measure on $R\S^{d-1}$ whose total mass equals $2\pi^{d/2}\Gamma(d/2)^{-1}R^{d-1}$, where $R>0$.
\item $G_D$, $H_D$: Green's function and Poisson kernel in $D$; see Section \ref{SS:harmonic}.
\end{itemize}

Mirroring the definitions \eqref{E:def_alpha_c}-\eqref{E:def_alpha_tr} and \eqref{E:def_alpha_R} of the critical points, we introduce the following supercritical phases:
\begin{align}
\mathrm{I}_\infty & = \{ \alpha > 0: \P(\Lc^\alpha \text{ contains an unbounded cluster}) > 0 \}, \\
\mathrm{I}_u & = \{ \alpha > 0: \P(\Lc^\alpha \text{ contains a unique unbounded cluster}) > 0 \}, \\
\mathrm{I}_\b2b & = \{ \alpha >0: \P(\forall R>1 : B(0,1) \overset{\Lc^\alpha}{\longleftrightarrow} \partial B(0,R)) >0 \}, \\
\mathrm{I}_\tr & = \{ \alpha > 0: \P(\Lc_{\ge1}^\alpha \text{ contains an unbounded cluster}) > 0 \},\\
\label{E:Jr}
\Ir_\tr^R & = \{ \alpha> 0: \P( \Lc_{1,R}^\alpha \text{ contains an unbounded cluster})>0 \}.
\end{align}
Statements about the supercritical phases above will be more precise than statements about the associated critical points: for instance, claiming that $\Ir_\infty \supset \Ir_u$ is stronger than claiming than $\alpha_c \le \alpha_u$ since it also handles the lower-end of the interval.

\subsection{Setup}
For $T>0$, let $\S^1_T$ be the circle of length $T$ and consider the space of rooted loops
\begin{equation}\label{E:space_rooted_loops}
\Pf := \{ \wp : \S^1_{T(\wp)} \to \R^3 \text{ continuous, with } T(\wp) \in (0,\infty) \}.
\end{equation}
We define an equivalence relation $\sim$ on $\Pf$ by saying that
two rooted loops $(\wp_t)_{0 \le t \le T(\wp)}$ and $(\wp_t')_{0 \le t \le T(\wp')}$ are equivalent if, and only if, $T(\wp) = T(\wp')$ and there exists $t_0 \in \S^1_{T(\wp)}$ s.t. $\wp_{t+t_0} = \wp'_t$ for all $t \in \S^1_{T(\wp)}$. We will view the equivalence class $[\wp]$ of a rooted loop $\wp$ as an unrooted loop and we will denote by
\begin{equation}\label{E:space_unrooted_loops}
[\Pf] := \Pf / \sim
\end{equation}
the space of unrooted loops. To ease notations, we will often simply write $\wp$ instead of $[\wp]$.
Finally, we consider the space $\Lf$ of ``locally finite'' collections $\Lc$ of unrooted loops:
\begin{equation}\label{E:space_collection_loops}
\Lf := \{ \Lc \subset [\Pf]: \forall R>0, \# \{ \wp \in \Lc, \wp(\S^1_{T(\wp)}) \subset B(0,R), \mathrm{diam}(\wp) > R^{-1} \} < \infty \}.
\end{equation}
We can equip each of these spaces with natural metrics that turn them into Polish spaces; see \cite[Section 2.1]{aidekon2023multiplicative}.
We will view the measure $\loopmeasure$ as a Borel measure on $[\Pf]$ and, for all $\alpha>0$, we will view a Brownian loop soup $\Lc^\alpha$ with intensity $\alpha$ as a random variable taking values in $\Lf$.

\subsection{Classical lemmas}

For ease of future reference, we record standard percolation and Poisson point process results. 
The following result is often referred to as Mecke equation, or Palm's formula and can be found in \cite[Theorems 4.1 and 4.4]{zbMATH06796875}.

\begin{lemma}[Palm's formula]\label{L:Palm}
Let $D\subset\R^3$ be an open set and $\alpha>0$.
Any measurable functions $F : [\Pf] \times \Lf \to [0,\infty]$ and $G : [\Pf] \times [\Pf] \times \Lf \to [0,\infty]$ satisfy
\begin{equation}\label{E:Palm}
\hspace{50pt}
    \E \Big[ \sum_{\wp \in \Lc_D^\alpha} F(\wp,\Lc_D^\alpha)] = \alpha \int_{[\Pf]} \E[F(\wp,\Lc_D^\alpha \cup \{\wp\})] \loopmeasure_D(\wp) \hspace{40pt} \text{and}
\end{equation}
\begin{equation}\label{E:Palm2}
    \E \Big[ \sum_{\wp \neq \wp' \in \Lc_D^\alpha} G(\wp,\wp',\Lc_D^\alpha)] = \alpha^2 \int_{[\Pf]\times[\Pf]} \E[G(\wp,\wp',\Lc_D^\alpha \cup \{\wp,\wp'\})] \loopmeasure_D(\wp) \loopmeasure_D(\wp').
\end{equation}    
\end{lemma}

Let $P(\wp,\Lc_D^\alpha)$ be a given property depending on $\Lc_D^\alpha$ and $\wp \in \Lc_D^\alpha$; for instance $P(\wp,\Lc_D^\alpha)$ could be the property that $\wp$ belongs to an unbounded cluster of $\Lc_D^\alpha$.
By \eqref{E:Palm}, if we want to show that $P(\wp,\Lc_D^\alpha)$ is satisfied almost surely for every loop $\wp \in \Lc_D^\alpha$, it is enough to show that $P(\wp,\Lc_D^\alpha \cup \{\wp\})$ holds almost surely where $\wp$ is sampled from $\loopmeasure_D$, independently of $\Lc_D^\alpha$.
If the property depends on two distinct loops $\wp,\wp' \in \Lc_D^\alpha$ instead, we proceed similarly using \eqref{E:Palm2}.
We will repeatedly use this reasoning and refer to it by saying ``by Palm's formula''.

For the next lemma which can be found in \cite[Lemma 2.1]{Janson84}, recall that a function $F: \Lf \to \R$ is said to be increasing if for all $\Lc, \Lc' \in \Lf$, $\Lc \subset \Lc' \implies F(\Lc) \le F(\Lc')$. An event $E$ is increasing if the indicator function of $E$ is an increasing function.

\begin{lemma}[FKG inequality]
    For all increasing measurable bounded functions $F, G : \Lf \to \R$,
    \[ 
    \E[F(\Lc_D^\alpha)G(\Lc_D^\alpha)] \ge \E[F(\Lc_D^\alpha)] \E[G(\Lc_D^\alpha)].
    \]
\end{lemma}

The final result is a direct consequence of FKG inequality and can be found in \cite[Exercise~11]{duminil}.

\begin{lemma}[Square-root trick]\label{L:square-root_trick}
For any increasing events $A_1, \dots, A_n$,
\[
\max_{i=1, \dots, n} \P(A_i) \ge 1 - (1 - \P(A_1 \cup \dots \cup A_n))^{1/n}.
\]
\end{lemma}

\part{Brownian excursions and loop measure}\label{PartA}

\section{Excursions and bubble measures}\label{S:Excursions}

This section introduces natural Brownian-type measures on trajectories and studies some of their properties in any dimension $d \ge 3$.


\subsection{Heat kernel, Green's function and Poisson kernel}\label{SS:harmonic}

We recall basic notions of harmonic analysis. We refer to \cite{MR3410783} for a detailed exposition.
We will denote the heat kernel in $\R^d$ by
\begin{equation}\label{E:heat_kernel}
p^{\R^d}_t(x,y) = \frac{1}{(2\pi t)^{d/2}} \exp \Big(-\frac{\|x-y\|^2}{2t} \Big), \qquad x, y \in \R^d, \quad t >0,
\end{equation}
and the Green's function in $\R^d$ by
\begin{equation}\label{E:Green}
G_{\R^d}(x,y) = \int_0^\infty p^{\R^d}_t(x,y) \d t = c_d \|x-y\|^{2-d},
\qquad x,y \in \R^d,
\qquad \text{where} \quad c_d = \frac{1}{2\pi^{d/2}} \Gamma(d/2-1).
\end{equation}
Now let $D \subset \R^d$ be some open set.
For $x,y \in D$ and $t>0$, let $\pi_t^D(x,y)$ be the probability that a Brownian bridge of duration $t$ from $x$ to $y$ remains in $D$. The Green's function in $D$ is then defined by
\begin{equation}\label{E:Green_D}
G_D(x,y) = \int_0^\infty p^{\R^d}_t(x,y) \pi_t^D(x,y) \d t.
\end{equation}
In the case of the unit ball, the Green's function is explicit: for all distinct points $x, y \in B(0,1)$,
\begin{equation}\label{E:Green_ball}
G_{B(0,1)}(x,y) = G_{\R^d}(x,y) - G_{\R^d}\Big(\frac{x}{\|x\|},\|x\|y \Big).
\end{equation}

For the remaining of this section, we need to impose some regularity assumption on $D$:
\begin{equation}\label{E:regularity_D}
\text{Suppose that }
D \text{ has a smooth boundary and that } G_D \text{ is } C^1 \text{ up to } \partial D.
\end{equation}
See \cite[Chapter 3]{MR3410783} for much more on the relations between regularity of $D$ and harmonic functions.
Let $x \in D$. The first hitting point of $\partial D$ by a Brownian motion starting from $x$ has a density with respect to the $(d-1)$-dimensional Hausdorff measure on $\partial D$: the Poisson kernel $H_D(x,\cdot)$. We use the convention that the integral of $H_D(x,\cdot)$ against the $(d-1)$-dimensional Hausdorff measure on $\partial D$ is one.
The Poisson kernel can also be obtained by differentiating the Green's function \cite[(3.1.47)]{MR3410783}: for any $x \in D$ and $y \in \partial D$, letting $n_y$ being the inward unit normal vector at $y$, we have
\begin{equation}
\label{E:Poisson_Green}
H_D(x,y) = \lim_{\eps \to 0} \frac{1}{2\eps} G_D(x,y+\eps n_y).
\end{equation}
If $D$ is the unit ball or the complement of the unit ball, then the Poisson kernel is explicit \cite[Theorem 3.44]{morters2010brownian}: for all $x \in D$ and $\omega \in \S^{d-1}$,
\begin{equation}
\label{E:Poisson_sphere}
H_D(x,\omega) = \frac{\Gamma(d/2)}{2\pi^{d/2}} \frac{|1-\|x\|^2|}{\|x-\omega\|^d}.
\end{equation}
Finally, the boundary Poisson kernel is defined by differentiating the Poisson kernel and the Green's function: for any distinct $x, y \in \partial D$,
\begin{equation}
\label{E:Poisson_boundary}
H_D(x,y) = \lim_{\eps \to 0} \frac{1}{\eps} H_D(x+\eps n_x,y) = \lim_{\eps \to 0} \frac{1}{2\eps^2} G_D(x+\eps n_x,y + \eps n_y).
\end{equation}

We state the following elementary result for ease of future reference. Recall that $m_{R\S^{d-1}}$ stands for the $(d-1)$-dimensional Hausdorff measure of $R\S^{d-1}$.

\begin{lemma}\label{L:Poisson}
Let $d \ge 3$.
For all $R>r>0$, the following integrals can be computed:
\begin{align}
\label{E:L_Poisson1}
& \int_{R\S^{d-1}} m_{R\S^{d-1}}(\d x) H_{B(R)}(y,x) = 1, \quad y \in r\S^{d-1}; \\
\label{E:L_Poisson2}
& \int_{r\S^{d-1}} m_{r\S^{d-1}}(\d x) H_{\R^d \setminus B(r)}(x,y) = (r/R)^{d-2}, \quad x \in R\S^{d-1};\\
& \int_{R\S^{d-1}} m_{R\S^{d-1}}(\d x) H_{B(R)\setminus \overline{B(r)}}(y,x) = (d-2) \frac{ r^{1-d} }{r^{2-d}-R^{2-d}}, \quad y \in r\S^{d-1}.
\label{E:L_Poisson3}
\end{align}
\end{lemma}

\begin{proof}
By definition of the Poisson kernel, the first (resp. second) integral equals the probability that a Brownian motion starting from $y$ (resp. $x$) hits $\S^{d-1}$ (resp. $r \S^{d-1}$) in finite time. These hitting probabilities can be computed from the fact that $x \mapsto \|x\|^{2-d}$ is harmonic. By \eqref{E:Poisson_boundary}, the third integral is the limit as $\eps \to 0$ of $1/\eps$ times the probability that a Brownian path starting at $y+\eps n_y$ hits $R\S^{d-1}$ before $r \S^{d-1}$, which is equal to:
\[
\frac{1}{\eps} \frac{r^{2-d}-(r+\eps)^{2-d}}{r^{2-d}-R^{2-d}} \sim (d-2) \frac{ r^{1-d} }{r^{2-d}-R^{2-d}}.
\]
This concludes the proof.
\end{proof}

\subsection{Whole-space excursion measures}\label{SS:excursions}

We start by defining and studying some properties of Brownian excursion measures in $\R^d$.
For distinct points $x,y \in \R^d$, consider the following stochastic differential equation
\[
\d W^{x,y}(t) = \d W(t) - \frac{d-2}{\|W^{x,y}(t)-y\|^2} (W^{x,y}(t)-y) \d t, \qquad \text{with} \quad W_0^{x,y} = x,
\]
where $W$ is a $d$-dimensional standard Brownian motion. The process $W^{x,y}$ is well defined up to $\tau_{y,r}$ for all $r>0$ where
\[
\tau_{y,r} = \inf \{ t > 0: \| W^{x,y}(t) - y \| = r \}.
\]
A small calculation involving Itô's formula shows that
\[
\d \|W^{x,y}(t) - y\|^2 = 2 \|W^{x,y}(t) - y\| \d Z_t + (4-d) \d t
\qquad \text{where} \quad \d Z_t = \frac{\sum_{j=1}^d (W^{x,y}_j(t) - y_j) \d W_j(t)}{\|W^{x,y}(t) - y\|}.
\]
The process $(Z_t)_{t \ge0}$ is a continuous local martingale and $\langle Z \rangle_t = t$. By Lévy's characterisation of Brownian motion, $Z$ is a standard one-dimensional Brownian motion and thus $\|W^{x,y} - y\|^2$ is a squared Bessel process of dimension $4-d$. Since $4-d < 2$, this implies in particular that $\|W^{x,y} - y\|^2$ reaches 0 in finite time almost surely. Thus $W^{x,y}$ is a random process from $x$ to $y$, well defined up to the hitting time $\tau_y = \inf \{ t>0 : W^{x,y}(t) = y\}$ of $y$.
We will denote by $\mu^\#_{x,y}$ the law of $(W^{x,y}(t))_{0 \le t \le \tau_y}$ and by $\E^\#_{x,y}$ the associated expectation.
We will identify the law of a standard $d$-dimensional Brownian motion run up to time infinity with $\mu^\#_{x,\infty}$, an excursion from $x$ to $\infty$.

Given a standard $d$-dimensional Brownian motion $W$, $\|W-y\|^{2-d}$ is a local martingale (recall \eqref{E:Green}).
By Girsanov's theorem, $\mu^\#_{x,y}$ agrees with the Doob's transform associated to this local martingale: for any $T>0$ and any test function $f$,
\begin{equation}
\label{E:RN_excursion}
\E^\#_{x,y}[f(W^{x,y}(t))_{0\le t \le T} \indic{\tau_y > T}]
= \|x-y\|^{d-2} \E_x [ f(W(t))_{0 \le t \le T} \|W(T)-y\|^{2-d} ].
\end{equation}
The following lemma provides a third description of $\mu_{x,y}^\#$.

\begin{lemma}
Let $x,y \in \R^d$ be two distinct points. Let $(W_t)_{t \ge 0}$ be a Brownian trajectory starting at $x$ and let $\tau_{y,r} = \inf\{t > 0: W_t \in \partial B(y,r) \}$. Then the law of $(W_t)_{t \le \tau_{y,r} \wedge \tau_{y,R}}$ conditioned on $\tau_{y,r} < \tau_{y,R}$ converges weakly as $r \to 0$, $R \to \infty$ to $\mu^\#_{x,y}$.
\end{lemma}

\begin{proof}
The proof follows from standard arguments that we omit.
\end{proof}

Recall that we denote by $p^{\R^d}_t(x,y)$ the heat kernel in $\R^d$ \eqref{E:heat_kernel}.

\begin{lemma}[Density and time reversal]\label{L:density_excursion}
Let $x, y \in \R^d$ be two distinct points and $(W^{x,y}(t))_{0\le t \le \tau_y}$ be sampled according to $\mu^\#_{x,y}$. 
For all $k\ge0$, $t_0=0 < t_1 <  \dots < t_k < t=t_{k+1}$, the joint law of $(W^{x,y}(t_1), \dots, W^{x,y}(t_k), \tau_y)\indic{\tau_y\ge t}$ has a density with respect to Lebesgue measure on $(\R^d)^k \times (0,\infty)$ given by: 
for all
pairwise distinct points $z_1, \dots, z_k \in \R^d \setminus \{x,y\}$,
\begin{equation}\label{E:L_density_excursion}
\P(W^{x,y}_{t_j} \in \d z_j, j=1 \dots k, \tau_y \in \d t) = \frac{2\pi^{d/2}}{\Gamma(d/2-1)} \|x-y\|^{d-2} \prod_{j=0}^k p^{\R^d}_{t_{j+1}-t_j}(z_j,z_{j+1}),
\end{equation}
where by convention $z_0=x$ and $z_{k+1}=y$. In particular, if $(W^{x,y}(t))_{0\le t \le \tau_y} \sim \mu^\#_{x,y}$, then the time reversal $(W^{x,y}(\tau_y-t))_{0\le t \le \tau_y}$ is distributed according to $\mu^\#_{y,x}$.
\end{lemma}

\begin{proof}[Proof of Lemma \ref{L:density_excursion}]
By the Doob's transform expression \eqref{E:RN_excursion} and dominated convergence theorem, $\tau_y$ has a density with respect to Lebesgue measure on $(0,\infty)$ given by
\begin{align*}
\P(\tau_y \in \d t) = \indic{t>0}
\|x-y\|^{d-2} \int_{\R^d} (\partial_t p^{\R^d}_t(x,z)) \|z-y\|^{2-d} \d z.
\end{align*}
The heat kernel satisfies the heat equation $\partial_t p^{\R^d}_t(x,z) = \frac{1}{2}\Delta_z p^{\R^d}_t(x,z)$.  By an integration by parts, we can make the Laplacian act on $z\mapsto \|z-y\|^{2-d}$ and, recalling that $\Delta_z \|z-y\|^{2-d} = \frac{4\pi^{d/2}}{\Gamma(d/2-1)} \delta_y(z)$, we obtain that
\begin{equation}
\label{E:pf_density_excursion}
\P(\tau_y \in \d t) = \frac{2\pi^{d/2}}{\Gamma(d/2-1)} \|x-y\|^{d-2} p^{\R^d}_t(x,y) \indic{t >0}.
\end{equation}
Alternatively, $\tau_y$ has the law of the first hitting time of the origin of a $(4-d)$-dimensional Bessel process starting at $\|x-y\|$ which has an explicit density (see e.g. \cite[Proposition 2.9]{LawlerBessel}).
\eqref{E:pf_density_excursion} is the $k=0$ case of \eqref{E:L_density_excursion}.
The proof of \eqref{E:L_density_excursion} in the general case then follows from \eqref{E:RN_excursion} and Markov's property.
The invariance under time inversion then follows directly by comparing the finite dimensional marginals of $(W^{x,y}(\tau_y-t))_{0\le t \le \tau_y}$ and $(W^{y,x}(\tau_x))_{0\le t \le \tau_x} \sim \mu^\#_{y,x}$.
\end{proof}

For distinct points $x, y \in \R^d$, we define the measure $\mu_{x,y} = G(x,y) \mu^\#_{x,y}$ with total mass $G(x,y)$. 
We can then define the excursion measure from $y$ to $y$ by
\begin{equation}
\label{E:construction_excursion_yy}
\mu_{y,y} = \lim_{\substack{x \to y\\x \ne y}} \mu_{x,y}.
\end{equation}
The total mass of $\mu_{y,y}$ is infinite and the above limit means that for any $r>0$, if we restrict $\mu_{x,y}$ to the space of trajectories intersecting $\partial B(y,r)$, then the limit exists and agrees with the finite measure $\mu_{y,y}$ restricted to that space.
This convergence is a direct consequence of Lemma \ref{L:density_excursion}. Using this lemma, one can also show that
\begin{equation}\label{E:excursion_bridge}
\mu_{y,y} = \int_0^\infty \frac{1}{(2\pi t)^{d/2}} \P^{y,y,t}.
\end{equation}

We now move to the invariance properties of $\mu_{x,y}$ under the action of conformal maps, focusing on the non trivial case of invariance under inversion.

\begin{notation}\label{N:inversion}
For $R>0$, define the $R$-scaling map
\begin{equation}
\label{E:map_scaling}
\scale_R : \wp \in \Pf \mapsto (R \wp_{t/R^2})_{0 \le t \le T(\wp)} \in \Pf.
\end{equation}
We will denote by $\iota : x \in \R^d \cup \{ \infty \} \mapsto -x / \|x\|^2 \in \R^d \cup \{ \infty \}$ the inversion of $\R^d$, with the convention that $\iota(\infty)=0$ and $\iota(0) = \infty$.
For a path $\wp = (\wp(t))_{0 \le t \le T(\wp)}$, let us denote by $\iota \circ \wp$ the path $(\iota(\wp(\sigma(t)))_{0 \le t \le T(\iota \circ \wp)}$ where
\begin{equation}
\label{E:iota_circ_wp}
\sigma(t) = \inf \Big\{ s>0: \int_0^s \|\wp(u)\|^{-4} \d u \ge t \Big\}
\quad \text{and} \quad T(\iota \circ \wp) = \int_0^{T(\wp)} \|\wp(u)\|^{-4} \d u.
\end{equation}
For any measure $\mu$ on paths, we will denote by $\iota \circ \mu$ the pushforward of $\mu$.
\end{notation}

One can directly check the following scaling property: for any points $x,y \in \R^d$ (distinct or not),
\begin{equation}
\label{E:scaling_excursion_whole}
\scale_R \circ \mu_{x,y} = R^{d-2} \mu_{Rx,Ry}.
\end{equation}

\begin{lemma}[Invariance under inversion]\label{L:inversion_excursion}
For any $x \in \R^d \setminus \{0\}$ and $y \in \R^d\cup\{\infty\} \setminus \{x\}$, $\iota \circ \mu_{x,y}^\# = \mu_{\iota(x), \iota(y)}^\#$
and $\iota \circ \mu_{x,x} = \| \iota(x) \|^{2d-4} \mu_{\iota(x),\iota(x)}.$
\end{lemma}

\begin{proof}[Proof of Lemma \ref{L:inversion_excursion}]
Let $x \in \R^d \setminus \{0\}$ and $y \in \R^d \setminus \{0\}$ (the cases $y=0$ and $y=\infty$ are similar). Let $(W^{x,y}(t))_{0 \le t \le \tau_y} \sim \mu_{x,y}^\#$ and define for all $t \in [0,\tau_y]$, $\tilde{W}^{x,y}(t) = - W^{x,y}(t)/\|W^{x,y}(t)\|^2$ and the time change
\[
\sigma(t) = \inf \Big\{ s>0: \int_0^s \|W^{x,y}(u)\|^{-4} \d u \ge t \Big\}. 
\]
An elementary computation and an application of Itô's formula show that $\tilde{W}^{x,y}$ satisfies
\begin{align*}
&\d \tilde{W}^{x,y}(t) = - \frac{1}{\|W^{x,y}(t)\|^2} \d W^{x,y}(t) + 2 \frac{W^{x,y}(t)}{\|W^{x,y}(t)\|^4} \sum_{j=1}^d W^{x,y}_j(t) \d W^{x,y}_j(t) + (d-2)\frac{W^{x,y}(t)}{\|W^{x,y}(t)\|^4} \d t \\
& = \hspace{-.45pt} \d Z_t \hspace{-.45pt} + \hspace{-.45pt} \Big( \frac{\|W^{x,y}(t)\|^2(W^{x,y}(t)-y)}{\|W^{x,y}_t-y\|^2} - 2 W^{x,y}(t) \sum_{j=1}^d \hspace{-.45pt} \frac{W^{x,y}_j(t)(W^{x,y}_j(t)-y_j)}{\|W^{x,y}_t-y\|^2} + W^{x,y}(t) \Big) \frac{(d-2)\d t}{\|W^{x,y}(t)\|^4}
\end{align*}
\[
\text{where} \hspace{60pt}
\d Z(t) = - \frac{1}{\|W^{x,y}(t)\|^2} \d W(t) + 2 \frac{W^{x,y}(t)}{\|W^{x,y}(t)\|^4} \sum_{j=1}^d W^{x,y}_j(t) \d W_j(t). \hspace{65pt}
\]
A lengthy calculation shows that the $\d t$-term can be rewritten as
\[
-\frac{(d-2)}{\|\tilde{W}^{x,y}(t) - \iota(y)\|^2} (\tilde{W}^{x,y}(t) - \iota(y)) \|\tilde{W}^{x,y}(t)\|^4 \d t.
\]
On the other hand,
$Z$ is a continuous local martingale and its martingale bracket equals
\[
\langle Z_j,Z_k \rangle_t = \delta_{j,k} \int_0^t \|W^{x,y}(s)\|^{-4} \d s = \delta_{j,k} \sigma^{-1}(t).
\]
By Lévy's characterisation, there exists a Brownian motion $\tilde{W}$ such that $Z_t = \tilde{W}_{\sigma^{-1}(t)}$. Altogether, we have
\[
\d \tilde{W}^{x,y}(t) = \d \tilde{W}(\sigma^{-1}(t)) -\frac{(d-2)}{\|\tilde{W}^{x,y}(t) - \iota(y)\|^2} (\tilde{W}^{x,y}(t) - \iota(y)) \|\tilde{W}^{x,y}(t)\|^4 \d t
\]
and thus
\[
\d \tilde{W}^{x,y}(\sigma(t)) = \d \tilde{W}(t) -\frac{(d-2)}{\|\tilde{W}^{x,y}(\sigma(t)) - \iota(y)\|^2} (\tilde{W}^{x,y}(\sigma(t)) - \iota(y)) \d t,
\]
i.e. $(\tilde{W}^{x,y}(\sigma(t)))_t$ follows the law $\mu_{\iota(x), \iota(y)}^\#$.
This concludes the proof of the invariance of $\mu^\#_{x,y}$ under inversion.

The proof that  $\iota \circ \mu_{x,x} = \| \iota(x) \|^{2d-4} \mu_{\iota(x),\iota(x)}$ then follows directly from \eqref{E:construction_excursion_yy} and from the fact that $
\lim_{y \to x} G(x,y) / G(\iota(x),\iota(y)) = \| \iota(x) \|^{2d-4}.
$
\end{proof}

\begin{remark}
Invariance of Brownian motion under inversion for $d \ge 3$ is not well known, but is not new; see \cite{zbMATH03886845}.
We can also mention a related well-known result on harmonic functions (see e.g. \cite[Section~3.1, Problem~15]{MR3410783}): if $\Omega$ is an open set of $\R^d \setminus \{0\}$, $\widetilde{\Omega} = \{ x/\|x\|^2: x \in \Omega \}$ and $u$ is harmonic in $\Omega$, then
\begin{equation}
\label{E:Kelvin}
x \in \widetilde{\Omega} \mapsto \|x\|^{2-d} u(x/\|x\|^2)
\end{equation}
is harmonic in $\widetilde{\Omega}$. In fact, the map \eqref{E:Kelvin} is called the Kelvin transform of $u$.
\end{remark}

\subsection{Boundary excursion and bubble measures}\label{SS:excursion_domain}

We have defined Brownian excursion measures $\mu_{x,y}$ \eqref{E:construction_excursion_yy}, for distinct or same starting and ending points $x$ and $y$, in the whole space $\R^d$. We will now define excursion measures in some open domain $D$ of $\R^d$, bulk and boundary versions.

\smallskip
\noindent\textbf{Bulk-to-bulk.}
The bulk-to-bulk measure is simply defined by restricting the whole-space measure: for all $x,y \in D$, define
\[
\mu_{x,y}^D = \indic{\wp \cap D^c = \varnothing} \mu_{x,y}(\d \wp).
\]
Using Lemma \ref{L:density_excursion}, one can show that the total mass of $\mu_{x,y}^D$ is given by $G_D(x,y)$.

To define boundary versions, we now need to assume that $\partial D$ is smooth in the same sense as in \eqref{E:regularity_D}.

\smallskip
\noindent\textbf{Bulk-to-boundary.}
Let $x \in D$ and $y \in \partial D$. Denote by $n_y$ the inward normal vector at $y$. 
Let $T>0$ and $F:\Pf \to [0,\infty)$ be a nonnegative measurable function. By \eqref{E:RN_excursion}, for any $\eps >0$,
\begin{align*}
\frac{1}{2\eps} \int \mu_{x,y+\eps n_y}^D(\d \wp) F(\wp) \indic{T(\wp) > T}
= \frac{1}{2\eps} \E_x [ F((\wp(t))_{0 \le t \le T}) \indic{\forall t \in [0,T], \wp(t) \in D} G_D(\wp(T),y+\eps n_y) ].
\end{align*}
By \eqref{E:Poisson_Green}, $G_D(\wp(T),y+\eps n_y)/(2\eps) \to H_D(\wp(T),y)$ as $\eps \to 0$ and we deduce that the following weak limit exists:
\begin{equation}
\label{E:excursion_bulk_boundary}
\mu_{x,y}^D := \lim_{\eps \to 0} \frac{1}{2\eps} \mu_{x,y+\eps n_y}^D,
\end{equation}
where the measure $\mu_{x,y}^D$ is characterised by:
for any $T>0$ and nonnegative measurable function $F:\Pf \to [0,\infty)$,
\begin{align*}
\int \mu_{x,y}^D(\d \wp) F((\wp(t))_{0 \le t \le T}) \indic{T(\wp) > T} = \E_x \Big[ F((\wp(t))_{0 \le t \le T}) \indic{\forall t \in [0,T], \wp(t) \in D} H_D(\wp(T),y) \Big].
\end{align*}
The total mass of $\mu_{x,y}^D$ is given by the Poisson kernel $H_D(x,y)$.

\smallskip
\noindent\textbf{Boundary-to-boundary.}
Let $x, y \in \partial D$ be distinct points. The boundary-to-boundary excursion measure is defined by
\begin{equation}
\mu_{x,y}^D
= \lim_{\eps \to 0} \frac{1}{\eps} \mu_{x + \eps n_x,y}
= \lim_{\eps \to 0} \frac{1}{2\eps^2} \mu_{x+\eps n_x, y+\eps n_y}.
\end{equation}
The existence of this limit can be justified in a similar way as for \eqref{E:excursion_bulk_boundary}.
By \eqref{E:Poisson_boundary}, the total mass of $\mu_{x,y}^D$ is given by the boundary Poisson kernel $H_D(x,y)$.

\smallskip
\noindent\textbf{Bubble measure.}
Finally, for $x \in \partial D$, we can define the bubble measure rooted at $x$ by
\begin{equation}\label{E:def_bubmeasure}
\mu^{\bub,D}_x = \lim_{\substack{y \to x\\y \in \partial D}} \mu_{x,y}^D.
\end{equation}
The total mass of $\mu^{\bub,D}_x$ is infinite and the above limit means that for any $r>0$, if we restrict $\mu_{x,y}^D$ to the space of trajectories intersecting $\partial B(x,r)$, then the limit exists and agrees with the finite measure $\mu^{\bub,D}_x$ restricted to that space.

The measures defined in this section satisfy many properties inherited from the properties of the whole-space excursions described in Section \ref{SS:excursions}. For future reference, we mention the following scaling property: recalling the definition \eqref{E:map_scaling} of the $R$-scaling map $\scale_R$ and denoting $\mathbb{H}^d = \{ (t,x) \in \R^d : t>0, x \in \R^{d-1} \}$, one has  
\begin{equation}
\label{E:scaling_excursion_bubble}
\scale_R *\mu^{\bub,\mathbb{H}^d}_0 = R^d \mu^{\bub,\mathbb{H}^d}_0
\quad \text{and} \quad
\scale_R *\mu^{\bub,\R^d \setminus B(0,1)}_{\omega} = R^d \mu^{\bub,\R^d \setminus B(0,R)}_{R\omega}, \quad \omega \in \S^{d-1}.
\end{equation}

\section{Brownian loop measure}\label{S:BLM}

This section studies basic properties of the Brownian loop measure in any dimension $d \ge 3$. Section~\ref{SS:BLM1} studies the invariance under inversion of $\loopmeasure$. Sections \ref{SS:BLM2} and \ref{SS:BLM3} provide two decompositions of $\loopmeasure$ by rerooting the loop in two different specific ways. Finally, Section \ref{SS:BLM4} describes $\loopmeasure$ restricting to loops which cross a given spherical shell.

\subsection{Invariance under inversion}\label{SS:BLM1}

Recall that in Notation \ref{N:inversion} we introduced the inversion map $\iota$. In this section, we show that the Brownian loop measure is invariant under $\iota$:

\begin{lemma}\label{L:inversion_loop}
As measures on the space $[\Pf]$ of unrooted loops, we have
$\iota \circ \loopmeasure = \loopmeasure$.
\end{lemma}

We will be able to use the results derived in Section \ref{SS:excursions} since, by combining \eqref{E:loopmeasure} and \eqref{E:excursion_bridge}, we see that the Brownian loop measure $\loopmeasure$ is related to the excursion measures $\mu_{x,x}$ by
\begin{equation}\label{E:loop_measure_excursion}
\loopmeasure= \int_{\R^3} \d x \, \frac{1}{T(\wp)} \mu_{x,x}(\d \wp).
\end{equation}

\begin{proof}[Proof of Lemma \ref{L:inversion_loop}]
We follow the same approach as in \cite{Lawler04} and say that a measurable function $U : \Pf \to (0,\infty)$ is a \textit{unit weight} if for all $\wp \in \Pf$,
\[
\int_0^{T(\wp)} U(\theta_t \circ \wp) \d t = 1
\]
where $\theta_t \circ \wp$ denotes the time shift of $\wp$ by $t$. A canonical example of unit weight is given by $U_0(\wp) = 1/T(\wp)$ and we can rephrase \eqref{E:loop_measure_excursion} as
\[
\loopmeasure = \int_{\R^d} \d x \, U_0(\wp) \mu_{x,x}(\d \wp).
\]
Since we view $\loopmeasure$ as a measure on the space $[\Pf]$ of unrooted loops, the above equality also holds for any unit weight $U$ instead of $U_0$. Here, we will consider 
$U(\wp) = \| \wp(0) \|^{-4} / T(\iota\circ \wp)$ which is a unit weight by definition \eqref{E:iota_circ_wp} of $\iota \circ \wp$. We thus have
\[
\loopmeasure = \int_{\R^d} \d x \, \|x\|^{-4} \frac{1}{T (\iota \circ \wp)} \mu_{x,x}(\d \wp)
\]
and, by Lemma \ref{L:inversion_excursion}, we deduce that
\begin{align*}
\iota \circ \loopmeasure = \int_{\R^d} \d x \, \|x\|^{-4} \frac{1}{T (\iota \circ \wp)} \iota \circ \mu_{x,x}(\d \wp)
= \int_{\R^d} \d x \, \|x\|^{-4} \|x\|^{4-2d} \frac{1}{T (\wp)} \mu_{\iota(x),\iota(x)}(\d \wp).
\end{align*}
With a change of variable $\tilde{x} = \iota(x)$ (the Jacobian cancels out the $\|x\|^{-2d}$ term), we get
\[
\iota \circ \loopmeasure = \int_{\R^d} \d \tilde x \, \frac{1}{T (\wp)} \mu_{\tilde x, \tilde x}(\d \wp) = \loopmeasure.
\]
\end{proof}

\subsection{Decomposition of Brownian loop measure}\label{SS:BLM2}

The following result is the main result of this section. It gives a decomposition of the $\loopmeasure$-measure by re-rooting the loop at its point whose distance to the origin (or any other fixed point by translation invariance) is minimal. This will be an effective tool for subsequent computations.

\begin{proposition}\label{P:decompo_loopmeasure}
As measures on the space $[\Pf]$ of unrooted loops, we have
\begin{equation}\label{E:P_decompo_loopmeasure}
\loopmeasure = \int_0^\infty \d a ~a^{d-1} \int_{\S^{d-1}} m_{\S^{d-1}}(\d \omega) \mu^{\bub,\R^d \setminus B(0,a)}_{a\omega}.
\end{equation}
\end{proposition}

As a sanity check, one can verify using \eqref{E:scaling_excursion_bubble} that the measure on the right hand side of \eqref{E:P_decompo_loopmeasure} is scale invariant as the Brownian loop measure should be.

Such a decomposition was already known in dimension 2 \cite[Proposition 8]{Lawler04} and in dimension 1 \cite[Corollary 3.20]{Lupu18}.
In fact, \cite{Lupu18} considers loop measures associated to general one-dimensional diffusions. 
To prove Proposition \ref{P:decompo_loopmeasure}, we will establish a link between the Brownian loop measure in $\R^d$ and the loop measure associated to the $d$-dimensional Bessel process reminiscent of the skew-decomposition of Brownian motion. We will then be able to use \cite{Lupu18}'s general result to derive Proposition \ref{P:decompo_loopmeasure}. This will provide a unifying framework for any dimension $d \ge 2$.
We emphasise that the two-dimensional strategy of \cite{Lawler04} cannot be replicated in other dimensions since it eventually relies on the conformal equivalence of the upper half plane and the unit disc.

\paragraph*{Bessel process and Brownian motion on the sphere.}
For any $r>0$, we will denote by $\P^r_{\Bes_d}$ the $d$-dimensional Bessel probability measure on paths that start at $r$.
For any $r,r'>0$, $t>0$, we will denote by $\P^{r,r';t}_{\Bes_d}$ the probability measure on $d$-dimensional Bessel bridges of duration $t$ starting at $r$ and ending at $r'$. We will also denote by $p^{\Bes_d}_t(r,\cdot)$ the density of the law of $\wp_t$ under $\P^r_{\Bes_d}$ with respect to Lebesgue measure on $(0,\infty)$. This density can actually be written explicitly in terms of some special function, see e.g. \cite[Section 2.2]{LawlerBessel}. For any bounded measurable function $F$ on paths, one has
\[
\E^r_{\Bes_d}[F((\wp_s)_{0 \le s \le t})] = \int_0^\infty p^{\Bes_d}_t(r,r') \E^{r,r';t}_{\Bes_d}[F((\wp_s)_{0 \le s \le t})] \d r.
\]
The loop measure associated to the $d$-dimensional Bessel process as defined in \cite[Definition~3.8]{Lupu18} is
\begin{equation}
\label{E:loopmeasure_Bessel}
\loopmeasure_{\Bes_d} = \int_0^\infty \d r \int_0^\infty \frac{\d t}{t} p^{\Bes_d}_t(r,r) \P^{r,r;t}_{\Bes_d}.
\end{equation}
The convention used in this article differs slightly from the one in \cite{Lupu18}: our term $p^{\Bes_d}_t(r,r)$ corresponds to $p_t(r,r)m(r)$ in \cite{Lupu18} where Lebesgue measure was not the reference measure used to define the densities.

Similarly to the Bessel case,
for any $\omega, \omega' \in \S^{d-1}, t>0$,
we will write $\P^\omega_{\S^{d-1}}$, $\P^{\omega,\omega';t}_{\S^{d-1}}$ for the law of Brownian motion and Brownian bridges in the sphere $\S^{d-1}$. We will also denote by $p^{\S^{d-1}}_t(\omega,\omega')$ the density with respect to the measure $m_{\S^{d-1}}$.

\paragraph*{Skew-product representation of Brownian motion.}
We recall the following classical result that describes the law of the radial and angular parts of Brownian motion in $\R^d$, see e.g. \cite[Chapter IV, (35.19)]{MR1780932}. Let $(W(t))_{t \ge 0}$ be a Brownian motion in $\R^d$ starting away from 0. Then $(\|W(t)\|)_{t\ge 0}$ is a $d$-dimensional Bessel process and
\begin{equation}
\label{E:skew_prod}
W(t)/\|W(t)\| = X\Big( \int_0^t \|W(s)\|^{-2} \d s \Big), \quad t \ge 0,
\end{equation}
where $X$ is a Brownian motion on the sphere $\S^{d-1}$ independent of $\|W\|$.

\begin{lemma}\label{L:skew_loopmeasure}
For any bounded measurable function $F: \Pf \to \R$,
\[
\int F(\wp) \loopmeasure_{\R^d}(\d \wp)
= \int \loopmeasure_{\Bes_d}(\d \wp^\rad) p^{\S^{d-1}}_{\tau^\rad}(\omega_0,\omega_0)
\int_{\S^{d-1}} m_{\S^{d-1}}(\d \omega) \E^{\omega,\omega;\tau^\rad}_{\S^{d-1}} [F(\wp^\rad \wp^\ang \circ U^\rad)]
\]
where $\omega_0 \in \S^{d-1}$ is any point and the time change $U^\rad$ is given by
\begin{equation}
\label{E:L_skew}
U^\rad(t) = \int_0^t \wp^\rad(s)^{-2} \d s, 0\le t\le T(\wp^\rad), \quad \text{and} \quad \tau^\rad = U^\rad(T(\wp^\rad)).
\end{equation}
\end{lemma}

\begin{proof}
This proof follows directly from the definitions \eqref{E:loopmeasure} and \eqref{E:loopmeasure_Bessel} of the Brownian loop measure in $\R^d$ and the Bessel loop measure in $(0,\infty)$ and from the skew-product representation \eqref{E:skew_prod} of Brownian motion. We write the details for completeness.
It is enough to prove the lemma for nonnegative continuous functions $F: \Pf \to [0,\infty)$. Let us fix such a function $F$.
We are first going to show that for all 
$r>0$, $\omega \in \S^2$ and $t>0$,
\begin{align}\label{E:pf_skew3}
r^{d-1} p_t^{\R^d}(r\omega,r\omega) \E_{\R^d}^{r\omega,r \omega;t}[F(\wp)] = p_t^{\Bes_d}(r,r) \E_{\Bes_d}^{r,r;t} [ p^{\S^{d-1}}_{\tau^\rad}(\omega,\omega) \E_{\S^{d-1}}^{\omega,\omega;\tau^\rad}[F(\wp^\rad \wp^\ang \circ U^\rad))] ].
\end{align}
Let $t>0$ and $x = r_0 \omega_0 \in \R^d$ with $r_0>0$ and $\omega_0 \in \S^{d-1}$. Let $f:\R^d \to [0,\infty)$ be nonnegative measurable function. To derive \eqref{E:pf_skew3}, we will compute $\E^x_{\R^d}[f(\wp(t))F((\wp(s))_{0\le s\le t})]$ in two different ways. Our first computation is direct and simply uses a change of variables:
\begin{align}
\notag
\E^x_{\R^d}[f(\wp(t))F((\wp(s))_{0\le s\le t})]
& =
\int_{\R^d} \d y~ p_t^{\R^d}(x,y) f(y) \E_{\R^d}^{x,y;t}[F(\wp)] \\
& = \int_0^\infty \d r ~r^{d-1} \int_{\S^{d-1}} m_{\S^{d-1}}(\d \omega) p_t^{\R^d}(x,r\omega) f(r\omega) \E_{\R^d}^{x,r \omega;t}[F(\wp)].
\label{E:pf_skew1}
\end{align}
On the other hand, using the skew-product representation \eqref{E:skew_prod} of Brownian motion, the expectation $\E^x_{\R^d}[f(\wp(t))F((\wp(s))_{0\le s\le t})]$ is also equal to
\begin{align}
\label{E:pf_skew2}
\int_0^\infty \hspace{-5pt} \d r~p_t^{\Bes_d}(r_0,r) \hspace{-4pt} \int_{\S^{d-1}} \hspace{-5pt} m_{\S^{d-1}}(\d \omega) f(r\omega) \E_{\Bes_d}^{r_0,r;t} [ p^{\S^{d-1}}_{\tau^\rad}(\omega_0,\omega) \E_{\S^{d-1}}^{\omega_0,\omega;\tau^\rad}[F((\wp^\rad(s) \wp^\ang (U^\rad(s)))_{0\le s\le t})] ]
\end{align}
where $\wp^\rad$ is sampled from $\P_{\Bes_d}^{r_0,r;t}$, $U^\rad$ and $\tau^\rad$ are as in \eqref{E:L_skew} and $\wp^\ang$ is sampled from $\P_{\S^{d-1}}^{\omega_0,\omega;\tau^\rad}$. In particular, \eqref{E:pf_skew1} and \eqref{E:pf_skew2} agree. Since this is true for any test function $f$, we must have for almost all $r$ and $\omega$,
\begin{align*}
& r^{d-1} p_t^{\R^d}(x,r\omega) \E_{\R^d}^{x,r \omega;t}[F(\wp)] = p_t^{\Bes_d}(r_0,r) \E_{\Bes_d}^{r_0,r;t} [ p^{\S^{d-1}}_{\tau^\rad}(\omega_0,\omega) \E_{\S^{d-1}}^{\omega_0,\omega;\tau^\rad}[F(\wp^\rad \wp^\ang \circ U^\rad))] ].
\end{align*}
Since $F$ is continuous, the left hand side and right hand side terms of the above display are continuous in $r$ and $\omega$ and we actually deduce that the above equality holds for all $r$ and $\omega$. In particular, it holds for $r=r_0$ and $\omega=\omega_0$ which corresponds to \eqref{E:pf_skew3} after relabelling.

We now use the definition \eqref{E:loopmeasure} of $\loopmeasure_{\R^d}$, perform a change of variable and use \eqref{E:pf_skew3} to get that
\begin{align*}
& \int F(\wp) \loopmeasure_{\R^d}(\d \wp)
= \int_{\R^d} \d x \int_0^\infty \frac{\d t}{t} p^{\R^d}_t(x,x) \E_{\R^d}^{x,x;t}[F(\wp)]\\
& = \int_0^\infty \d r~ r^{d-1} \int_{\S^{d-1}} m_{\S^{d-1}}(\d \omega) \int_0^\infty \frac{\d t}{t} p_t^{\R^d}(r\omega, r \omega) \E_{\R^d}^{r\omega,r\omega;t}[F(\wp)]\\
& = \int_0^\infty \d r \int_0^\infty \frac{\d t}{t} p_t^{\Bes_d}(r,r) \E_{\Bes_d}^{r,r;t} [ p^{\S^{d-1}}_{\tau^\rad}(\omega,\omega) \E_{\S^{d-1}}^{\omega,\omega;\tau^\rad}[F(\wp^\rad \wp^\ang \circ U^\rad))] ].
\end{align*}
The Bessel loop measure \eqref{E:loopmeasure_Bessel} naturally appears in the last equation. Together with the fact that $p^{\S^{d-1}}_{\tau^\rad}(\omega,\omega)$ does not depend on $\omega$ (rotational invariance of Brownian motion), this concludes the proof of the lemma.
\end{proof}

We now recall a result from \cite{Lupu18} which decomposes the Bessel loop measure $\loopmeasure_{\Bes_d}$ according to the minimal point of the loop. To this end, we recall that for all $a>0$ there exists a natural infinite measure $\mu_{\Bes_d}^{\ge a}$ (denoted $\eta^{>a}$ in \cite{Lupu18} with $w(a)=a^{1-d}$) on Bessel excursions starting and ending at $a$ and staying in $[a,\infty)$. This is the analogue of bubble measures \eqref{E:def_bubmeasure} for the Bessel process.

\begin{proposition}[Corollary 3.20 of \cite{Lupu18}]\label{P:titus}
As measures on the space $[\Pf]$ of unrooted loops, we have
\begin{equation}
\loopmeasure_{\Bes_d} = \int_0^\infty \mu_{\Bes_d}^{\ge a} a^{1-d} \d a.
\end{equation}
\end{proposition}

We now state a skew-product representation of the bubble measure:

\begin{lemma}\label{L:skew_bubble}
For any $a>0$, $\omega \in \S^{d-1}$ and any bounded measurable function $F: \Pf \to \R$,
\begin{equation}\label{E:L_skew_bubble}
\int F(\wp) \mu^{\bub,\R^d \setminus B(0,a)}_{a\omega}(\d \wp) = a^{-2(d-1)} \int \mu_{\Bes_d}^{\ge a}(\d \wp^\rad) p_{\tau^\rad}^{\S^{d-1}}(\omega,\omega) \E^{\omega,\omega;\tau^\rad}_{\S^{d-1}} [F([\wp^\rad \wp^\ang \circ U^\rad])],
\end{equation}
where $\tau^\rad$ and $U^\rad$ are as in \eqref{E:L_skew}.
\end{lemma}

\begin{proof}
The proof boils down to the skew-product representation \eqref{E:skew_prod} of Brownian motion. Since it is very similar to the proof of Lemma \ref{L:skew_loopmeasure} we omit the details. To check that the multiplicative constant is correct, one can for instance compute the asymptotic behaviour as $R \to \infty$ of the measures on both sides of \eqref{E:L_skew_bubble} of the event that the loop hits $a R \S^{d-1}$. For the left hand side, using similar computations as in the proof of Proposition \ref{P:loopmeasure_cross}, one finds $(1+o(1))\frac{d-2}{|\S^{d-1}|} a^{-d} R^{2-d}$. For the right hand side, using the explicit law of the maximum of $\wp$ under $\mu_{\Bes_d}^{\ge a}$ derived in \cite[Corollary 3.20]{Lupu18}, one finds
\begin{align*}
& (1+o(1)) a^{-2(d-1)} (\lim_{t \to \infty} p_t^{\S^{d-1}}(\omega,\omega)) \mu_{\Bes_d}^{\ge a}(\max \wp > aR)\\
& = (1+o(1)) a^{-2(d-1)} \frac{1}{|\S^{d-1}|} (d-2)^2 \int_{aR}^\infty \frac{b^{1-d}}{(a^{2-d} - b^{2-d})^2} \d b
= (1+o(1)) \frac{d-2}{|\S^{d-1}|} a^{-d} R^{2-d}.
\end{align*}
This concludes the proof.
\end{proof}

We now have all the ingredients to prove Proposition \ref{P:decompo_loopmeasure}.

\begin{proof}[Proof of Proposition \ref{P:decompo_loopmeasure}]
Let $F:\Pf \to [0,\infty)$ be a measurable function.
Combining Lemma \ref{L:skew_loopmeasure} and Proposition \ref{P:titus}, we obtain that $\int F([\wp]) \loopmeasure_{\R^d}(\d \wp)$ is equal to
\begin{align*}
\int_0^\infty \d a ~a^{1-d} \int_{\S^{d-1}} m_{\S^{d-1}}(\d \omega) \int \mu_{\Bes_d}^{\ge a}(\d \wp^\rad) p_{\tau^\rad}^{\S^{d-1}}(\omega,\omega) \E^{\omega,\omega;\tau^\rad}_{\S^{d-1}} [F([\wp^\rad \wp^\ang \circ U^\rad])]
\end{align*}
where $\tau^\rad$ and $U^\rad$ are as in \eqref{E:L_skew}.
The skew-product decomposition of the bubble measure (Lemma \ref{L:skew_bubble}) then shows that
\[
\int F([\wp]) \loopmeasure_{\R^d}(\d \wp) = \int_0^\infty \d a ~a^{d-1} \int_{\S^{d-1}} m_{\S^{d-1}}(\d \omega) \int F(\wp) \mu^{\bub,\R^d \setminus B(0,a)}_{a\omega}(\d \wp)
\]
which concludes the proof of Proposition \ref{P:decompo_loopmeasure}.
\end{proof}

\subsection{Rooting the loop at a double point}\label{SS:BLM3}

In this section, we give a decomposition of the measure $\loopmeasure_{\R^d}$ by rooting the loop at a double point.
Since Brownian motion is almost surely simple in dimension at least 4, the results of this section are restricted to $d \in \{2,3\}$. Moreover, one can easily generalise our arguments to $p$-multiple points in dimension 2 for any $p$. In fact, in dimension 2, an analogous statement for the so-called thick points which are points of infinite multiplicity has been derived in \cite[Lemma 5.1]{aidekon2023multiplicative}.

For a Brownian trajectory $\wp$, let $\ell^2_\wp$ be the self intersection local time of $\wp$. It is a $\sigma$-finite measure on $(0,T(\wp))^2$ supported on pairs $(t_1,t_2)$ such that $\wp(t_1)=\wp(t_2)$ formally given by
\begin{equation}\label{E:l2}
\ell^2_\wp(\d t_1 \d t_2) = \indic{\wp(t_1)=\wp(t_2)} \d t_1 \d t_2.
\end{equation}
Since the expectation of the measure on the right hand side of \eqref{E:l2} vanishes, a renormalisation procedure is needed to properly define $\ell^2_\wp$; see \cite{MR723731, MR701921}. 
The measure $\ell^2_\wp$ is infinite because of the many self intersections occurring in each arbitrary small time intervals.

\begin{lemma}\label{L:double}
Let $d \in \{2,3\}$.
For all nonnegative measurable function $F: \R^d \times \Pf \times \Pf \to [0,\infty)$,
\begin{align}
\label{E:P_double}
& \int \loopmeasure_{\R^d}(\d \wp) \int_{(0,T(\wp))^2} \ell^2_\wp(\d t_1 \d t_2) F(\wp(t_1),\wp_{\vert [t_1,t_2]}, \wp_{\vert [t_2,t_1]}) \\
& = 2 \int_{\R^d} \d x \int \mu_{x,x}(\d \wp_1) \int \mu_{x,x}(\d \wp_2) F(x,\wp_1,\wp_2),
\notag
\end{align}
where $[t_1,t_2]$ and $[t_2,t_1]$ are the two intervals of $\S^1_{T(\wp)}$ with end points $t_1$ and $t_2$ and where the measure $\mu_{x,x}$ is defined in \eqref{E:construction_excursion_yy}.
\end{lemma}

Informally, Lemma~\ref{L:double} states that, if we denote by $L^2_\wp(\d x)$ the ``push forward of $\ell^2_\wp$ by the map $(t_1,t_2) \mapsto \wp(t_1)$'', then the following measures on unrooted loops agree:
\begin{equation}
L^2_\wp(\d x) \loopmeasure_{\R^3}(\d \wp) = 2 \mu_{x,x} \wedge \mu_{x,x} \d x,
\end{equation}
where ``$\wedge$'' means that we concatenate the two loops.
This is only informal since the measure $L^2_\wp(\d x)$ is infinite on any open set of $\R^3$.
Similarly, the measure $\mu_{x,x} \wedge \mu_{x,x}$ is infinite on any nondegenerate event since one of the two excursion measure can always produce an arbitrarily small loop.

We initially derived Lemma~\ref{L:double} in order to prove Lemma \ref{L:key} below, but ended up using a different argument. We nevertheless decided to keep Lemma~\ref{L:double} because we find it interesting in its own right and it could be useful for other purposes.

\begin{proof}
By symmetry, the left hand side of \eqref{E:P_double} is twice the same expression with the additional constraint that $t_1<t_2$ (alternatively, one can do the computation for both contribution and then realise that they agree).
By definition \eqref{E:loopmeasure} of $\loopmeasure_{\R^d}$, and recalling from \eqref{E:heat_kernel} that we denote by $(p_t^{\R^d}(x,y))_{t>0,x,y\in\R^d}$ the heat kernel in $\R^d$, the left hand side of \eqref{E:P_double} is then equal to
\begin{align*}
2 \int_{\R^d} \d z \int_0^\infty \frac{\d t}{t} p_t^{\R^d}(z,z) \int \P^{z,z;t}(\d \wp) \int_{(0,t)^2} \ell^2_\wp(\d t_1 \d t_2) \indic{t_1<t_2} F(\wp(t_1),\wp_{\vert [t_1,t_2]}, \wp_{\vert [t_2,t_1]}).
\end{align*}
By \cite[Theorem 2.1]{MR880979}, for any $z \in \R^d$ and $t>0$,
\begin{align*}
& p_t^{\R^d}(z,z) \int \P^{z,z;t}(\d \wp) \int_{(0,t)^2} \ell^2_\wp(\d t_1 \d t_2) \indic{t_1<t_2} F(\wp(t_1),\wp_{\vert [t_1,t_2]}, \wp_{\vert [t_2,t_1]})\\
& = \int_{(0,t)^2} \d t_1 \d t_2 \indic{t_1<t_2} \int_{\R^d} \d x ~p_{t_1}^{\R^d}(z,x) p_{t_2-t_1}^{\R^d}(x,x) p_{t-t_2}^{\R^d}(x,z)\\
& \hspace{20pt} \times \int \P^{z,x;t_1}(\d \wp_1) \int \P^{x,x;t_2-t_1}(\d \wp_2) \int \P^{x,z;t-t_2}(\d \wp_3) F(x,\wp_1,\wp_3 \wedge \wp_2).
\end{align*}
Since the integrand is nonnegative, we can exchange the order of integration. Integrating with respect to $z \in \R^3$ first, we end up computing
\begin{align*}
\int_{\R^d} \d z~ p_{t-t_2}^{\R^d}(x,z) p_{t_1}^{\R^d}(z,x) \P^{x,z;t-t_2} \wedge \P^{z,x;t_1} = p_{t-t_2+t_1}^{\R^d}(x,x) \P^{x,x;t-t_2+t_1}.
\end{align*}
Putting things together, we obtain that the left hand side of \eqref{E:P_double} equals
\begin{align*}
& 2 \int_{\R^d} \d x \int_0^\infty \frac{\d t}{t} \int_{(0,t)^2} \d t_1 \d t_2 \indic{t_1<t_2} p_{t_2-t_1}^{\R^d}(x,x) p_{t-t_2+t_1}^{\R^d}(x,x) \\
& \hspace{16pt} \times \int \P^{x,x;t_2-t_1}(\d \wp_1) \int \P^{x,x;t-t_2+t_1}(\d \wp_2) F(x,\wp_1,\wp_2).
\end{align*}
We now do the change of variables $s_1 = t_2-t_1$, $s_2 = t-t_2+t_1$ and $s_3=t_1$ where the domain of integration is $\{ s_1, s_2 \in (0,\infty), s_3 \in (0,s_1+s_2)\}$. Integrating first with respect to $s_3$, we obtain a factor $s_1+s_2$ which cancels out with the $t$ in the denominator of the above display. Overall, we get that the left hand side of \eqref{E:P_double} equals
\begin{align*}
2 \int_{\R^d} \d x \int_{(0,\infty)^2} \d s_1 \d s_2~ p_{s_1}^{\R^d}(x,x) p_{s_2}^{\R^d}(x,x) \int \P^{x,x;s_1}(\d \wp_1) \int \P^{x,x;s_2}(\d \wp_2) F(x,\wp_1,\wp_2).
\end{align*}
Combining this with the relation \eqref{E:excursion_bridge} between $\mu_{x,x}$ and the bridge probability measures, this concludes the proof.
\end{proof}

\subsection{Crossing \texorpdfstring{$n$}{n} times a spherical shell}\label{SS:BLM4}

For $n \geq 1$ and $r \in (0,1)$, let $\cross_{r,n}(\wp)$ be the event that the loop $\wp$ crosses $B(1) \setminus \overline{B(r)}$ exactly $n$ times.
We are going to describe precisely the loop measure $\loopmeasure_{\R^d}$ restricted to this event, up to rerooting the loop.
To this end, we will write $\mathbf{x} = (x_1,\dots, x_n)$, $\mathbf{y} = (y_1,\dots, y_n)$ and
\[
m_{\S^{d-1}}^{\otimes n}(\d \mathbf{x}) = m_{\S^{d-1}}(\d x_1) \dots m_{\S^{d-1}}(\d x_n)
\quad \text{and} \quad
m_{r\S^{d-1}}^{\otimes n}(\d \mathbf{y}) = m_{r\S^{d-1}}(\d y_1) \dots m_{r\S^{d-1}}(\d y_n).
\]

\begin{proposition}\label{P:loopmeasure_cross}
Let $r \in (0,1)$ and $n \ge 1$. For any bounded measurable function $F:\Pf \to \R$,
\begin{align}\label{E:P_loopmeasure_cross}
\int F([\wp]) \mathbf{1}_{\cross_{r,n}(\wp)} \loopmeasure_{\R^d}(\d \wp)
& = \frac{1}{n} \int_{(\S^{d-1})^n}
m_{\S^{d-1}}^{\otimes n}(\d \mathbf{x}) \int_{(r\S^{d-1})^n} m_{r\S^{d-1}}^{\otimes n}(\d \mathbf{y}) \times \\
& \times \prod_{j=1}^n \int \mu_{x_j,y_j}^{\R^d \setminus \overline{B(r)}}(\d e_{2j-1}) \int \mu_{y_j,x_{j+1}}^{B(1)}(\d e_{2j}) F([e_1 \wedge e_2 \wedge \dots \wedge e_{2n}]),
\notag
\end{align}
where the product of integrals means multiple integrals and $x_{n+1}=x_1$.
\end{proposition}

We emphasise that the excursion measures on the right hand side of \eqref{E:P_loopmeasure_cross} are not normalised and thus implicitly contain Poisson kernel terms; see in particular Section \ref{SS:excursion_domain}.
A more readable version of \eqref{E:P_loopmeasure_cross} when $n=1$ is written in \eqref{E:pf_cross2}. In Corollary \ref{C:crossing_exactly_n} below, we spell out direct consequences of Proposition \ref{P:loopmeasure_cross}, coupled with elementary properties on Poisson point processes.

\begin{proof}
To ease the notations, we prove this for $n=1$. By Proposition \ref{P:decompo_loopmeasure}, we have
\begin{align}\label{E:pf_cross1}
\int F([\wp]) \mathbf{1}_{\cross_{r,n}(\wp)} \loopmeasure_{\R^d}(\d \wp)
= \int_0^r \d a~ a^{d-1} \int_{\S^{d-1}} m_{\S^{d-1}}(\d \omega) \int F([\wp]) \mathbf{1}_{\cross_{r,n}(\wp)} \mu^{\bub,\R^d\setminus B(a)}_{a\omega}(\d \wp).
\end{align}
Let $a \in (0,r)$, $\omega \in \S^{d-1}$, $\wp$ be a loop sampled according to $\mathbf{1}_{\cross_{r,n}(\wp)} \mu^{\bub,\R^d\setminus B(a)}_{a\omega}$ and consider the following stopping times:
\[
\tau_1 = \inf \{ t>0: \wp(t) \in \S^{d-1} \}
\quad \text{and} \quad
\tau_2 = \inf \{ t> \tau_1: \wp(t) \in r \S^{d-1} \}.
\]
This decomposes $\wp$ into three pieces:
\begin{itemize}
\item
an excursion $\tilde{e}_1 = (\wp(t))_{0 \le t \le \tau_1}$ in $B(1) \setminus \overline{B(a)}$ between two boundary points;
\item
an excursion $e_1 = (\wp(t))_{\tau_1 \le t \le \tau_2}$ in $\R^d \setminus \overline{B(r)}$ between a bulk point and a boundary point;
\item
an excursion $\tilde{e}_2 = (\wp(t))_{\tau_2 \le t \le T(\wp)}$ in $B(1) \setminus \overline{B(a)}$ between a bulk point and a boundary point.
\end{itemize}
This yields
\begin{align*}
& \int F([\wp])\mathbf{1}_{\cross_{r,n}(\wp)} \mu^{\bub,\R^d\setminus B(a)}_{a\omega}(\d \wp)
= \int_{\S^{d-1}} m_{\S^{d-1}}(\d x_1) \int_{r\S^{d-1}} m_{r\S^{d-1}}(\d y_1) \times \\
& \hspace{60pt} \times
\int \mu_{a\omega,x_1}^{B(1) \setminus \overline{B(a)}}(\d \tilde{e}_1)
\int \mu_{x_1,y_1}^{\R^d \setminus \overline{B(r)}}(\d e_1)
\int \mu_{y_1,a\omega}^{B(1) \setminus \overline{B(a)}}(\d \tilde{e}_2)
F([\tilde{e}_1 \wedge e_1 \wedge \tilde{e}_2]).
\end{align*}
We now exchange the order of integration in \eqref{E:pf_cross1} and fix $x_1$ and $y_1$ and integrate with respect to $a$ and $\omega$.
Decomposing an excursion from $y_1$ to $x_1$ in $B(1)$ according to its point closest to the origin, we have for any bounded measurable function $\tilde{F}:\Pf \to \R$,
\begin{align*}
& \int \tilde{F}(e_2) \mu_{y_1,x_1}^{B(1)}(\d e_2) = \int_0^r \hspace{-1.3pt} \d a~a^{d-1} \int_{\S^{d-1}} \hspace{-1.5pt} m_{\S^{d-1}}(\d \omega) \int \mu_{y_1,a\omega}^{B(1) \setminus \overline{B(a)}}(\d \tilde{e}_2) \int \mu_{a\omega,x_1}^{B(1) \setminus \overline{B(a)}}(\d \tilde{e}_1) \tilde{F}(\tilde{e}_2 \wedge \tilde{e}_1).
\end{align*}
Since $F([\tilde{e}_1 \wedge e_1 \wedge \tilde{e}_2]) = F([e_1 \wedge \tilde{e}_2 \wedge \tilde{e}_1])$, we have obtained that
\begin{align}\label{E:pf_cross2}
& \int F([\wp])\mathbf{1}_{\cross_{r,n}(\wp)} \mu^{\bub,\R^d\setminus B(a)}_{a\omega}(\d \wp) \\
& = \int_{\S^{d-1}} m_{\S^{d-1}}(\d x_1) \int_{r\S^{d-1}} m_{r\S^{d-1}}(\d y_1) \int \mu_{x_1,y_1}^{\R^d \setminus \overline{B(r)}}(\d e_1) \int \mu_{y_1,x_1}^{B(1)}(\d e_2) F([e_1 \wedge e_2]).
\notag
\end{align}
This is \eqref{E:P_loopmeasure_cross} for $n=1$ which concludes the proof.
\end{proof}

We now describe consequences of Proposition \ref{P:loopmeasure_cross}.
Let $\Lc_r^{(n)}$ be the subset of $\Lc_{\R^d}^\alpha$ consisting of the loops which cross $B(1) \setminus \overline{B(r)}$ exactly $n$ times.
Conditionally on $\# \Lc_r^{(n)}$, the loops in $\Lc_r^{(n)}$ are i.i.d. Let us denote by $\Pds_r^{(n)}$ their common distribution. An unrooted loop sampled according to $\Pds_r^{(n)}$ can be decomposed as the (unrooted) concatenation of $e_1, e_2, \dots, e_{2n}$ with $e_{2j-1}$ (resp. $e_{2j}$) being excursions from $\S^{d-1}$ to $r\S^{d-1}$ in $\R^d \setminus \overline{B(r)}$ (resp. from $r\S^{d-1}$ to $\S^{d-1}$ in $B(1)$).
We will denote by $\wt \Pds_r$ (resp. $\wh \Pds_r$) the law of a Brownian excursion from $\S^{d-1}$ to $r\S^{d-1}$ in $\R^d \setminus \overline{B(r)}$ (resp. from $r\S^{d-1}$ to $\S^{d-1}$ in $B(1)$) with starting and ending points that are independent and uniform on both spheres.
Finally, we introduce the following measure on $(\S^{d-1})^n \times (r\S^{d-1})^n$:
\begin{equation}
\label{E:nu_r^n}
d \nu_r^{(n)}(\mathbf{x},\mathbf{y}) :=
\frac1n \Big( \prod_{j=1}^n H_{\R^d \setminus \overline{B(r)}}(x_j,y_j) H_{B(1)}(y_j,x_{j+1}) \Big) m_{\S^{d-1}}^{\otimes n}(\d \mathbf{x}) m_{r\S^{d-1}}^{\otimes n}(\d \mathbf{y}),
\end{equation}
with the convention that $x_{n+1}=x_1$.

\begin{corollary}\label{C:crossing_exactly_n}
Let $n \geq 1$ be an integer and $r \in (0,1)$.

1. \textbf{Number of loops.} $\# \Lc_r^{(n)}$ is a Poisson random variable with mean $\alpha$ times
\begin{align}
\label{E:loopmeasure_exactly_n}
& \loopmeasure_{\R^d}(\cross_{r,n}(\wp))
= \nu_r^{(n)}((\S^{d-1})^n \times (r\S^{d-1})^n).
\end{align}

2. \textbf{Description of $\Pds_r^{(n)}$.}
Let $(X_1, \dots, X_n, Y_1, \dots, Y_n)$ be a random element of $(\S^{d-1})^n \times (r\S^{d-1})^n$ sampled according to $\frac1Z \nu_r^{(n)}$ where $Z$ is the normalising constant \eqref{E:loopmeasure_exactly_n}.
Conditionally on $(X_1, \dots, Y_n)$, let $e_{2j-1}$ (resp. $e_{2j}$) be independent excursions from $X_j$ to $Y_j$ in $\R^3 \setminus \overline{B(r)}$ (resp. from $Y_j$ to $X_{j+1}$ in $B(1)$), $j=1, \dots, n$.
Then the loop obtained by concatenating $e_1, e_2, \dots, e_{2n}$ is distributed according to $\Pds_r^{(n)}$.

3. \textbf{Decoupling.} Let $r' < r$, define $E_{r,r'}^{(n)} := \{ \exists j=1, \dots, n, e_{2j}$ reaches $B(r')\}$ and recall that $\wt \Pds_r$ and $\wh \Pds_r$ are defined above \eqref{E:nu_r^n}.
The laws of $(e_{2j-1})_{j=1}^n$ under $\Pds_r^{(n)}$, $(e_{2j-1})_{j=1}^n$ under $\Pds_r^{(n)}(\cdot \vert E_{r,r'}^{(n)})$ and $(e_{2j})_{j=1}^n$ under $\Pds_r^{(n)}$ are mutually absolutely continuous with respect to $\wt \Pds_r^{\otimes n}$, $\wt \Pds_r^{\otimes n}$ and $\wh \Pds_r^{\otimes n}$ respectively.
Moreover, each Radon-Nikodym derivative
\begin{equation}
\label{E:L_Radon-N}
\frac{\d \Pds_r^{(n)}((e_{2j-1})_{j=1}^n)}{\d \wt \Pds_r^{\otimes n}},
\quad
\frac{\d \Pds_r^{(n)}((e_{2j-1})_{j=1}^n \vert E_{r,r'}^{(n)})}{\d \wt \Pds_r^{\otimes n}}
\quad \text{and} \quad
\frac{\d \Pds_r^{(n)}((e_{2j})_{j=1}^n)}{\d \wh \Pds_r^{\otimes n}}
\end{equation}
is equal to $(1+O(r))^n$ (uniformly in $r'$ for the second one).
\end{corollary}

The third point compares the law of the excursions $(e_{2j-1})_{j=1}^n)$ under $\Pds_r^{(n)}$ (which are not independent since $\nu_r^{(n)}$ is not a product measure) with the law of independent excursions. This is a decoupling estimate. Of course, we cannot consider all excursions in this procedure because the last point of an excursion coincides with the first point of the next and so we only consider every second excursion.

\begin{proof}[Proof of Corollary \ref{C:crossing_exactly_n}]
Points 1. and 2. are direct consequences of Proposition \ref{P:loopmeasure_cross}. 

We now prove the third point. Conditionally on their endpoints, the excursions $e_1, \dots, e_{2n}$ are independent. Since the law of the endpoints $(X_1, \dots, X_n, Y_1, \dots, Y_n)$ is $\frac1Z \nu_r^{(n)}$ (Point 2.), it is enough to show that the Radon--Nikodym derivative of $\frac{1}{Z} \nu_r^{(n)}$ with respect to the product measure $m_{\S^2}^{\otimes n} \otimes m_{r\S^2}^{\otimes n}$, normalised to be a probability measure, equals $(1+O(r))^n$. This follows from the following estimate.
Using the explicit expression \eqref{E:Poisson_sphere} of the Poisson kernels in $B(1)$ and $\R^d \setminus \overline{B(r)}$, we find that for all $y \in r\S^{d-1}$ and $x \in \S^{d-1}$,
\begin{equation}
\label{E:Poisson_approx}
H_{B(1)}(y,x) = (1+O(r)) \frac{\Gamma(d/2)}{2\pi^{d/2}}
\quad \text{and} \quad
H_{\R^d \setminus \overline{B(r)}}(x,y) = (1+O(r)) \frac{\Gamma(d/2)}{2\pi^{d/2}} r^{2-d}.
\end{equation}
This concludes the proof.
\end{proof}

\section{Three-dimensional estimates}

This short section records estimates on the three-dimensional Brownian loop measure that will be of use. It relies on the groundwork of Sections \ref{S:Excursions} and \ref{S:BLM}.

\subsection{Crossing a spherical shell}

\begin{lemma}\label{L:estimate_crossing}
The following estimates hold:
\begin{align}
\label{E:L_estimate_crossing1}
\loopmeasure_{\R^3}(\{ \wp : \S^2 \overset{\wp}{\longleftrightarrow} r\S^2 \})
& \le \frac{1+r}{2} ((1-r)^{-2}-1) \qquad  \text{for all } r \in (0,1); \\
\label{E:L_estimate_crossing2}
\loopmeasure_{\R^3}(\{ \wp : \S^2 \overset{\wp}{\longleftrightarrow} r\S^2 \})
& = (1+O(r)) r \qquad  \text{as } r  \to 0.
\end{align}
\end{lemma}

\begin{proof}
By Proposition \ref{P:decompo_loopmeasure},
\begin{equation}
\label{E:samedi1}
\loopmeasure_{\R^3}(\{ \wp : \S^2 \overset{\wp}{\longleftrightarrow} r\S^2 \})
= \int_0^r \d a~ a^2 \int_{\S^2} m_{\S^2}(\d \omega) \mu_{a\omega}^{\bub,\R^3 \setminus B(a)}(\wp \cap \S^2 \neq \varnothing).
\end{equation}
Let $a \in [0,r]$ and $\omega \in \S^2$. We have
\begin{align*}
\mu_{a\omega}^{\bub,\R^3 \setminus B(a)}(\wp \cap \S^2 \neq \varnothing)
= \int_{\S^2} m_{\S^2}(\d x) H_{B(1)\setminus B(a)}(a\omega,x) H_{\R^3 \setminus B(a)}(x,a\omega).
\end{align*}
Using the explicit expression \eqref{E:Poisson_sphere} of the Poisson kernel in $\R^3 \setminus B(a)$, we have for each $x \in \S^2$,
\begin{equation}
\label{E:samedi2}
H_{\R^3 \setminus B(a)}(x,a\omega') \le \frac{1+r}{4\pi a(1-a)^2}
\quad \text{and} \quad
H_{\R^3 \setminus B(a)}(x,a\omega') = \frac{(1+O(r))}{4\pi a}.
\end{equation}
Using the first bound, together with \eqref{E:L_Poisson3}, we get that
\begin{align*}
\mu_{a\omega}^{\bub,\R^3 \setminus B(a)}(\wp \cap \S^2 \neq \varnothing)
\le \frac{2}{4\pi a(1-a)^2} \int_{\S^2} m_{\S^2}(\d x) H_{B(1)\setminus B(a)}(a\omega,x)
= \frac{1+r}{4\pi a(1-a)^2} \frac{a^{-2}}{a^{-1} - 1}.
\end{align*}
Plugging this estimate in \eqref{E:samedi1}, we have obtained
\[
\loopmeasure_{\R^3}(\{ \wp : \S^2 \overset{\wp}{\longleftrightarrow} r\S^2 \})
\le (1+r) \int_0^r \frac{\d a}{(1-a)^3} = \frac{1+r}{2} ((1-r)^{-2}-1).
\]
This proves \eqref{E:L_estimate_crossing1}. To prove \eqref{E:L_estimate_crossing2}, we use the same method by using the second estimate in \eqref{E:samedi2} instead of the first one.
\end{proof}

\begin{lemma}\label{L:sphere_bound}
There exists $C>0$ such that for all $R\ge1$,
\begin{equation}
\label{E:sphere_bound}
\loopmeasure ( \{ \wp : \mathrm{diam}(\wp) \ge 1, \wp \cap \partial B(0,R) \neq \varnothing \} ) \le C R^2.
\end{equation}
\end{lemma}

\begin{proof}
By scale invariance and \eqref{E:L_estimate_crossing1},
\[
\loopmeasure(\{ \wp: (R-1)\S^2 \overset{\wp}{\longleftrightarrow} R\S^2 \}) = \loopmeasure(\{ \wp: (1-1/R)\S^2 \overset{\wp}{\longleftrightarrow} \S^2 \}) \le R^2.
\]
It then remains to bound
\[
\loopmeasure ( \{ \wp : \mathrm{diam}(\wp) \ge 1, \wp \cap \partial B(0,R) \neq \varnothing , \wp \cap (R-1)\S^2 = \varnothing\} ).
\]
By Proposition \ref{P:decompo_loopmeasure}, it is equal to
\begin{align*}
& \int_{R-1}^R \d a~a^2 \int_{\S^2} m_{\S^2}(\d \omega) \mu_{a\omega}^{\bub,\R^3 \setminus B(a)}(\mathrm{diam}(\wp) \ge 1, \wp \cap \partial B(0,R) \neq \varnothing) \\
& \le \int_{R-1}^R \d a~a^2 \int_{\S^2} m_{\S^2}(\d \omega) \mu_{a\omega}^{\bub,\R^3 \setminus B(a)}(\mathrm{diam}(\wp) \ge 1).
\end{align*}
A simple computation shows that the above bubble measure can be bounded by some constant independent of $a$, $\omega$ or $R$. It remains to integrate the volume of $B(R) \setminus B(R-1)$ which is of order $R^2$. This concludes the proof.
\end{proof}

\subsection{Welding lemma}

We finish this section by stating and proving an intermediate result (Lemma \ref{L:intersection_loop} below) which will be useful to weld successive loops.
Before doing so, we recall a result from \cite{zbMATH00918284} on the intersection of independent Brownian paths in 3D.

\begin{lemma}\label{L:intersection_BM}
Let $B_1$ and $B_2$ be two independent Brownian paths with uniform starting points on $\S^2$. Let $T_i(R) := \inf \{ t > 0: |B_i(t)| = R \}$, $i=1,2$.
There exist $\xi_3(1,1) >0$ and constants $C_1, C_2 >0$ such that for all $R \geq 1$,
\[
C_1 R^{-\xi_3(1,1)} \leq \Prob{ \{ B_1(t), t \in [0,T_1(R)] \} \cap \{ B_2(t), t \in [0,T_2(R)] \} = \varnothing } \leq C_2 R^{-\xi_3(1,1)}.
\] 
\end{lemma}

The existence of the exponent was first derived in \cite{zbMATH04123007} using a subadditivity argument. The up-to-constant estimate stated above was derived in \cite{zbMATH00918284}.
Contrary to the planar case, the Brownian intersection exponent $\xi_3(1,1)$ is not known in 3D. However, it is known that $1/2 < \xi_3(1,1) < 1$ \cite{zbMATH04186799}. Actually, crude bounds on the non-intersection probability would be enough for our purposes.

\medskip

Let $r >0$, $a>1$ and let $(B_t, t \in [0,\tau_r])$ be a Brownian motion that starts uniformly on $r\S^{2}$ and killed upon reaching $ar \S^2$. For any path $\wp$, let
\begin{equation}
\label{E:g_ra}
g_{r,a}(\wp) := \Prob{ \{ \wp_t, t \in [0,T(\wp)]\} \cap \{B_t, t \in [0,\tau_r] \} \neq \varnothing },
\end{equation}
where only the randomness of $B$ is integrated out. The quantity $g_{r,a}(\wp)$ is a notion of size of $\wp$ viewed from Brownian motion. One could certainly use the capacity of $\wp$ instead.

\begin{lemma}\label{L:intersection_loop}
Let $\eta \in (0,1)$. There exist $a=a(\eta)>$ large enough, $c=c(\eta,\alpha)>0$ such that for all $r \in (0,1)$ and any loop $\wp_0$ with $g_{1,a}(\wp_0) \geq \eta$,
\begin{equation}
\Prob{ \exists \wp \in \Lc_{B(a^2)}^\alpha: \wp_0 \overset{\wp}{\longleftrightarrow} B(r), g_{r,a}(\wp) \geq \eta } \geq c r.
\end{equation}
\end{lemma}

\begin{proof}
Let $r \in (0,1)$. We start by noticing that if $B$ is a Brownian motion that starts uniformly on $r\S^{2}$ and which is killed upon reaching $\partial B(ar)$, then Lemma \ref{L:intersection_BM} shows that
\begin{align*}
\Expect{1-g_{r,a}(B)} \leq C_2 a^{-\xi_3(1,1)}.
\end{align*}
By Markov inequality, we deduce that
\begin{equation}
\label{E:intersection1}
\Prob{g_{r,a}(B) \geq \eta} = 1-\Prob{1-g_{r,a}(B) > 1-\eta}
\geq 1 - C_2 a^{-\xi_3(1,1)}/(1-\eta).
\end{equation}
By picking $a$ large enough, we can ensure that the right hand side is at least $1-\eta/2$.

Let $\wp_0$ be a loop as in the statement of the lemma. We apply Corollary~\ref{C:crossing_exactly_n} to $n=1$.
The probability that there is a loop crossing exactly once
$B(a) \setminus \overline{B(r)}$ is of order $r$.
Consider such a loop $\wp$ whose law is described in Corollary~\ref{C:crossing_exactly_n}. Let $E_1$ be the event that $\wp$ intersects $\wp_0$ and $E_2$ the event that $g_{r,a}(\wp) \geq \eta$. Using Corollary~\ref{C:crossing_exactly_n} Point 3. and the definition of $g_{1,a}(\wp_0)$, we have that $\Prob{E_1} \geq (1+O(a^{-1})) \eta$. On the other hand, by Corollary~\ref{C:crossing_exactly_n} Point 3. and \eqref{E:intersection1}, $\Prob{E_2} \geq (1+O(a^{-1}))(1-\eta/2)$. If $a$ is large enough, the sum of these two probabilities exceeds one and by a union bound we deduce that $\Prob{E_1 \cap E_2} \ge c$ for some constant $c>0$ which may depend on $\eta$ and $a$. This concludes the proof.
%
%
%
%
%
%
\end{proof}

\part{Brownian loop soup clusters}\label{PartB}

In this entire part, we work in dimension $d=3$.

\section{First percolative properties}\label{S:first_perco}

This section gathers many rather soft properties concerning the structure of clusters, depending on the value of $\alpha$.

\subsection{Ergodicity and zero-one law}

Recall that we denote by $\Lf$ \eqref{E:space_collection_loops} the space of locally finite collections of unrooted loops.
For $x \in \R^3$, let $\tau_x : \Lf \to \Lf$ be the shift operator induced by the shift $y \in \R^3 \mapsto x+y \in \R^3$.

\begin{lemma}[Ergodicity]\label{L:ergodicity}
For all $x \in \R^3 \setminus \{0\}$, the shift operator $\tau_x$ is ergodic with respect to the law of $\Lc^\alpha$. That is, the law of $\Lc^\alpha$ is preserved by $\tau_x$ and if $A \subset \Lf$ is a Borel set which is $\tau_x$-invariant ($\tau_x^{-1}(A) = A$), then $\P(\Lc^\alpha \in A) \in \{0,1\}$.
\end{lemma}

The discrete analogue of this result can be found in \cite[Proposition 3.2]{chang2016phase}.

\begin{proof}[Proof of Lemma \ref{L:ergodicity}]
This will follow from standard arguments. We provide a proof for completeness.
Let $x \in \R^3 \setminus \{0\}$.
In this proof, we will denote by $\Mc^\alpha$ the law of $\Lc^\alpha$, which is a Borel measure on $\Lf$.
Firstly, by translation invariance of $\Lc^\alpha$, $\tau_x$ preserves $\Mc^\alpha$. It remains to show that any Borel set $A \subset \Lf$ that is $\tau_x$-invariant has $\Mc^\alpha$-measure 0 or 1.
To show this, it is enough to prove that for all Borel sets $A,B \subset \Lf$,
\begin{equation}
    \label{E:mixing}
\lim_{n \to \infty} \Mc^\alpha(A \cap t_x^{-n}(B)) = \Mc^\alpha(A) \Mc^\alpha(B).
\end{equation}
Fix $B$ first.
For any $D \subset \R^3$ and any Borel set $A \subset \Lf$, we will say that $A$ is $D^c$-independent if for all $\Lc \in \Lf$, $\Lc \in A$ if, and only if, the restriction $\{ \wp \in \Lc: \wp \subset D\}$ of $\Lc$ to $D$ belongs to $A$. Denote by $\Pc$ the collection of Borel sets $A \subset \Lf$ which are $D^c$-independent for some bounded $D \subset \R^3$. The collection $\Pc$ is a $\pi$-system which generates the $\sigma$-algebra of $\Lc$.
The set
\[
\mathcal{G}=\{ A \subset \Lf \text{ Borel}: \lim_{n \to \infty} \Mc^\alpha(A \cap t_x^{-n}(B)) = \Mc^\alpha(A) \Mc^\alpha(B) \}
\]
is a $\lambda$-system. By Dynkin's $\pi-\lambda$ theorem, to show that $\mathcal{G}$ contains all Borel subsets $A \subset \Lf$, it is enough to show that $\Pc \subset \mathcal{G}$. Fixing now $A \in \mathcal{P}$ and running the above argument with $B$ shows that, to obtain \eqref{E:mixing} for all Borel subsets $A,B$, it is enough to derive it for all $A,B \in \Pc$. But for such events this is clear since $A$ and $t_x^{-n}(B)$ are independent for $n$ large enough.
\end{proof}

\begin{lemma}[Zero-one law]\label{L:zero_one}
Let $k \in \{c, u, \b2b, \tr\}$ and let $E_k$ be the event appearing in the definition of $\Ir_k$. Then for all $\alpha \in \Ir_k$, $\P(E_k)=1$.

\noindent
For $R>1$, let $E_\tr^R$ be the event appearing in the definition of $\Ir_\tr^R$. Then for all $\alpha \in \Ir_\tr^R$, $\P(E_\tr^R)=1$.
\end{lemma}

\begin{proof}[Proof of Lemma \ref{L:zero_one}]
For $\Ir_k$, $k \in \{c,u,\tr\}$, and for $\Ir_\tr^R$, $R>1$, this follows directly from Lemma~\ref{L:ergodicity} since the related events are all invariant under translations and have thus probability 0 or 1.
Concerning $\Ir_\b2b$, we need to work a bit more since it is not clear that the event $E_\b2b$ appearing in the definition of $\Ir_\b2b$ is translation invariant.
Consider the modified event
\[
\tilde{E}_\b2b := \{ \exists r>0, \forall R>r, B(0,r) \overset{\Lc^\alpha}{\longleftrightarrow} \partial B(0,R) \}.
\]
This event is invariant under translations since for any $x \in \R^3$, we can ``sandwich'' spheres around $x$ by two spheres around the origin and vice versa. By Lemma \ref{L:ergodicity}, we deduce that $\P(\tilde{E}_\b2b) \in \{0,1\}$. Let $\alpha \in \Ir_\b2b$. By definition, for such a value of $\alpha$, $\P(\tilde{E}_\b2b)>0$ and thus $\P(\tilde{E}_\b2b)=1$. Now, consider the random variable
\[
\rho := \inf\{ r>0: \forall R>r, B(0,r) \overset{\Lc^\alpha}{\longleftrightarrow} \partial B(0,R) \}.
\]
By scale invariance of $\Lc^\alpha$, the law of $\rho$ is scale invariant and thus $\rho$ takes values in $\{0,+\infty\}$. Since $\P(\tilde{E}_\b2b)=1$, $\rho$ is finite a.s. and therefore, $\rho=0$ a.s. In particular, $\rho < 1$ a.s. and thus $\P(E_\b2b)=1$: with probability 1, for all $R>1$, there is a cluster of $\Lc^\alpha$ intersecting both $B(0,1)$ and $\partial B(0,R)$.
This concludes the proof.
\end{proof}

\subsection{First qualitative behaviours of the supercritical phases}

\begin{lemma}\label{L:I2_no_bdd}
Let $\alpha \in \Ir_\infty$. Almost surely, each cluster of $\Lc^\alpha$ is unbounded and dense in $\R^3$. Moreover, either $\Lc^\alpha$ contains a unique cluster almost surely, or infinitely many almost surely.
\end{lemma}

\begin{proof}[Proof of Lemma \ref{L:I2_no_bdd}]
Let $(W_t)_{t \ge0}$ be a standard Brownian motion starting at the origin of $\R^3$ and independent of $\Lc^\alpha$ and let
\[
T := \inf \{ t>0: W_t \text{ belongs to an unbounded cluster of } \Lc^\alpha \}.
\]
Let $\lambda >0$. Because $((\lambda^{-1/2} W_{\lambda t})_{t \ge 0}, \lambda^{-1/2} \Lc^\alpha)$ has the same law as $((W_t)_{t \ge 0}, \Lc^\alpha)$, $\lambda T \overset{\rm (d)}{=} T$, i.e. the law of $T$ is scale invariant. In particular, $\P(T \in \{0,\infty\}) = 1$. Assume by contradiction that $\P(T = \infty) >0$. Then the event
\[
E := \{ W_{[0,1]} \text{ does not intersect an unbounded cluster of } \Lc^\alpha \}
\]
has a positive probability. But, by Markov's property and because clusters of $\Lc^\alpha$ are non polar for Brownian motion,
\[
\P(W_{[1,2]} \text{ intersects an unbounded cluster of } \Lc^\alpha \vert E) > 0.
\]
This implies that $\P(T \in [1,2]) > 0$ which is absurd. We have proved that $T=0$ a.s.

The first half of a Brownian bridge being absolutely continuous with respect to an unconditioned Brownian motion, the same conclusion would remain for a Brownian bridge trajectory independent of $\Lc^\alpha$: almost surely, the trajectory hits instantaneously an unbounded cluster of $\Lc^\alpha$.
We can then conclude by Palm's formula (Lemma \ref{L:Palm}) that the cluster of each loop in $\Lc^\alpha$ is unbounded.

We now show that each cluster is dense in $\R^3$. Let $W = (W_t)_{0 \le t \le T}$ be a Brownian bridge trajectory starting and ending at some point $x_0 \in \R^3$, $x \in \R^3 \setminus \{x_0\}$ and $\eps >0$. Let $U = B(x,\eps) \cap \partial B(x_0,\|x-x_0\|)$. We decompose the sphere $\partial B(x_0,\|x-x_0\|)$ as a finite union of $U_i, i=1, \dots, N$, where each $U_i$ is some rotated version of $U$ (rotation centred at $x_0$). For $i=1, \dots, N$, let $A_i^T$ be the event that there is a cluster of $\Lc^\alpha$ intersecting both $W$ and $U_i$. Conditionally on $W$, these events are increasing and we get by the square-root trick (see Lemma~\ref{L:square-root_trick}) that
\[
\max_{i=1, \dots, N} \P(A_i^T \vert W) \ge 1 - (1-\P(A_1^T \cup \dots \cup A_N^T\vert W))^{1/N}.
\]
Because there is a.s. an unbounded cluster of $\Lc^\alpha$ intersecting $W$, $\P(A_1^T \cup \dots \cup A_N^T \vert W)=1$.
We deduce that the set
\[
I_T = \{ i \in \{1,\dots,N\}: \P(A_i^T \vert W) = 1 \}
\]
is a.s. nonempty. By monotonicity,
\[
\bigcap_{T>0} I_T = \{ i \in \{1,\dots,N\}: \lim_{T \to 0^+} \P(A_i^T \vert (W_t)_{0 \le t \le T}) = 1 \}
\]
is also nonempty a.s. For all $i=1,\dots,N$, the random variable $\lim_{T \to 0^+} \P(A_i^T \vert (W_t)_{0 \le t \le T})$ is measurable with respect to the tail sigma algebra of $(W_t)_{t\ge0}$ and is thus constant. By rotational invariance, this constant is independent of $i$ and must be equal to 1.
In particular, for any $T>0$, there is a cluster intersecting $(W_t)_{0 \le t \le T}$ and $B(x,\eps)$ a.s. By Palm's formula (Lemma~\ref{L:Palm}), we deduce that the cluster of each loop reaches $B(x,\eps)$ a.s. Since this is true for any $x$ and $\eps$, it shows that the cluster of each loop is dense in $\R^3$.

It remains to show that the number $N$ of clusters is either one a.s. or infinity a.s.
Since $N$ is invariant under translations, it is actually deterministic by ergodicity (Lemma \ref{L:ergodicity}) and there exists $k \in \{1, 2, \dots\} \cup \{\infty\}$ such that $N=k$ a.s. Assume by contradiction that $k\notin \{1,\infty\}$. Since clusters are not polar, by adding an independent Brownian trajectory, we can connect two of these clusters and reduce the total number of clusters by at least one with positive probability. By Palm's formula, we deduce that $\P(N \le k-1) >0$ which is absurd. This shows that $\P(N=1) = 1$ or $\P(N=\infty)=1$.
\end{proof}

\begin{lemma}\label{L:I1_two_open_subsets}
Let $\alpha \in \Ir_\b2b$. Then, for any non empty open sets $U, V \subset \R^3$, $\P( U \overset{\Lc^\alpha}{\longleftrightarrow} V) = 1$.
\end{lemma}

\begin{proof}[Proof of Lemma \ref{L:I1_two_open_subsets}]
By translation and scaling invariance, we can assume that $U$ contains the ball $B(0,1)$ and $V$ contains a ball $B(x,\eps)$ for some $x \in \R^3 \setminus B(0,1)$ and $\eps >0$. By Lemma \ref{L:zero_one}, the probability $\P(B(0,1) \overset{\Lc^\alpha}{\longleftrightarrow} \partial B(0,\|x\|) )$ equals 1. As in the proof of Lemma \ref{L:I2_no_bdd}, we can then use the square-root trick to show the desired result: $\P( U \overset{\Lc^\alpha}{\longleftrightarrow} V) = 1$.
\end{proof}

\begin{lemma}[Uniqueness of unbounded cluster]\label{L:I3}
For all $\alpha \in \Ir_\tr$, $\Lc^\alpha_{\ge1}$ contains a unique unbounded cluster a.s.
The supercritical phase $\Ir_u$ is an interval of the form $(\alpha_c^3, \infty)$ or $[\alpha_c^3, \infty)$ and $\Ir_u \supset \Ir_\tr$. Moreover, for all $\alpha \in \Ir_u$, $\Lc^\alpha$ contains a unique cluster almost surely (no bounded clusters).
\end{lemma}

\begin{proof}[Proof of Lemma \ref{L:I3}]
Let $\alpha \in \Ir_\tr$. We start by showing that $\Lc^\alpha_{\ge 1}$ contains a unique unbounded cluster a.s. By Lemma \ref{L:zero_one}, $\Lc^\alpha_{\ge1}$ contains at least one unbounded cluster a.s. By ergodicity, the number of unbounded cluster is deterministic and, by the same argument as in the proof of Lemma~\ref{L:I2_no_bdd}, this number is either 1 or $\infty$. To exclude the scenario of infinitely many unbounded clusters, we will use a Burton--Keane-type argument \cite{MR990777}. 
For each loop $\wp \in \Lc^\alpha_{\ge 1}$, we will say that $\wp$ is a \textit{trifurcation loop} if, after deleting $\wp$, the cluster of $\Lc^\alpha_{\ge 1}$ containing $\wp$ is made of at least three unbounded clusters.

Let $R>0$. Enumerate in any given way the trifurcation loops $\wp_1, \dots, \wp_N$ of $\Lc^\alpha_{\ge 1}$ that are included in $B(0,R)$. By a standard procedure, we can associate to each $\wp_i$ pairwise distinct loops $\tilde \wp_i$ that intersect $\partial B(0,R)$. This implies that
\begin{equation}
\label{E:pf_Burton_Keane}
\E [ \# \{\wp \in \Lc^\alpha_{\ge1}: \wp~\text{trifurcation}, ~\wp \subset B(0,R) \} ]
\le \E [ \# \{ \wp \in \Lc^\alpha_{\ge1} : \wp \cap \partial B(0,R) \neq \varnothing \} ] \le C R^2.
\end{equation}
The last bound follows by Lemma \ref{L:sphere_bound} which crucially uses the lower bound on the diameter on the loops: it states that the expected number of loops in $\Lc^\alpha_{\ge 1}$ that intersect $\partial B(0,R)$ scales at most like the area of the sphere $\partial B(0,R)$.

Assume now by contradiction that $\Lc^\alpha_{\ge1}$ has infinitely many unbounded clusters a.s. Then there exists $R_0>0$ large enough so that the probability that $\Lc^\alpha_{\ge1}$ has at least three unbounded clusters intersecting $B(0,R_0/2)$ is positive. Since these clusters are not polar, the probability that an independent Brownian bridge trajectory, of duration $t \in [1,2]$ and starting at 0, has a diameter at least 1, stays in $B(0,R_0)$, and intersects three infinite clusters of $\Lc^\alpha_{\ge 1}$ is positive. By translation invariance, the same is true for any starting point of the Brownian bridge. By Palm's formula, we deduce that, for $R>0$ large enough, the left hand side of \eqref{E:pf_Burton_Keane} is at least a positive constant times $R^3$. This creates a contradiction and proves the first item.

\medskip

Let $\alpha \in \Ir_u$ (assuming at this stage that $\Ir_u$ is not empty). Let $(W_t)_{t \ge0}$ be a Brownian motion starting at 0 independent of  $\Lc^\alpha$ and let
\[
T := \inf \{ t> 0: W_t \text{ belongs to the unique cluster of } \Lc^\alpha \}.
\]
By our Brownian-germ trick, $T=0$ a.s. Let $\alpha'>\alpha$ and couple $\Lc^\alpha$ and $\Lc^{\alpha'}$ such that $\Lc^{\alpha'} \setminus \Lc^\alpha$ is independent of $\Lc^\alpha$ and distributed as $\Lc^{\alpha'-\alpha}$. We deduce that each loop of $\Lc^{\alpha'} \setminus \Lc^\alpha$ intersects the unique cluster of $\Lc^\alpha$ almost surely and thus $\Lc^{\alpha'}$ contains a unique cluster almost surely. This shows that $\alpha' \in \Ir_u$ and concludes the proof that $\Ir_u$ is an interval, unbounded from above.

Now, let $\alpha \in \Ir_\tr$. We want to show that $\alpha \in \Ir_u$.
Let $r>0$. Generalising the first item of this lemma to any $r$, the collection $\Lc^\alpha_{\ge r}$ contains a unique unbounded cluster $\Cc_\infty(r)$ a.s.
Clearly, $\Cc_\infty(r') \subset \Cc_\infty(r)$ if $r' \ge r$. In particular, $\Cc_\infty(0) := \bigcup_{r>0} \Cc_\infty(r)$ is connected. For a Brownian motion $W$ independent of $\Lc^\alpha$, we consider
\[
T := \inf \{ t>0: W_t \in \Cc_\infty(0) \}.
\]
By our Brownian-germ trick, $T=0$ a.s. As before, using Palm's formula we deduce that each loop of $\Lc^\alpha$ belongs to $\Cc_\infty(0)$ almost surely. This concludes the proof of the inclusion $\Ir_\tr \subset \Ir_u$.

Finally, the fact that for all $\alpha \in \Ir_u$, $\Lc^\alpha$ contains a unique cluster almost surely follows from the combination of Lemmas \ref{L:zero_one} and \ref{L:I2_no_bdd}.
\end{proof}

\begin{lemma}\label{L:Jr}
Let $R > 1$. The supercritical phase $\Ir_\tr^R$ \eqref{E:Jr} can be alternatively defined as
\begin{equation}
\label{E:L_Jr1}
\Ir_\tr^R = \{ \alpha>0: \lim_{n \to \infty} \P(\S^2 \overset{\Lc_{1,R}^\alpha}{\longleftrightarrow} n\S^2) > 0 \}.
\end{equation}
Moreover, for all $\alpha \in \Ir_\tr^R$,
\begin{equation}
\label{E:L_Jr2}
\lim_{r \to \infty} \P(\forall r'>r,  ~r\S^2 \overset{\Lc_{1,R}^\alpha}{\longleftrightarrow} r'\S^2) = 1.
\end{equation}
\end{lemma}

\begin{proof}[Proof of Lemma \ref{L:Jr}]
Let $\tilde{\Ir}_\tr^R$ be the set on    the right hand side of \eqref{E:L_Jr1}. By definition \eqref{E:Jr} of $\Ir_\tr^R$, if $\alpha \in \Ir_\tr^R$, the probability that there exists an unbounded cluster of $\Lc_{1,R}^\alpha$ is positive. By translation invariance, such a cluster intersects the unit sphere with positive probability showing that $\Ir_\tr^R \subset \tilde{\Ir}_\tr^R$. 
To prove the other inclusion, let us fix $\alpha \in \tilde\Ir_\tr^R$. By definition, the event that for all $n \ge 1$, there exists a cluster $\Cc_n$ of $\Lc_{1,R}^\alpha$ which intersects both $\S^2$ and $n\S^2$, has positive probability. 
Thanks to the restriction on the diameter of the loops, the number of loops in $\Lc_{1,R}^\alpha$ intersecting $\S^2$ is finite a.s., and so is the collection $\{ \Cc_n, n \ge 1\}$ of clusters. In particular, there exists an increasing subsequence $(n_k)_{k \ge 1}$ such that for all $k \ge 1$, $\Cc_{n_k} = \Cc_{n_1}$. By construction, the cluster $\Cc_{n_1}$ is an unbounded cluster of $\Lc_{1,R}^\alpha$ showing that $\alpha \in \Ir_\tr^R$.
This concludes the proof of \eqref{E:L_Jr1}.

Let $\alpha \in \Ir_\tr^R$. The proof of \eqref{E:L_Jr2} is a small variant of the proof of Lemma \ref{L:zero_one}. Indeed, the event
\[
\{ \exists r>0: \forall r'>r, r \S^2 \overset{\Lc_{1,R}^\alpha}{\longleftrightarrow} r' \S^2 \}
\]
is invariant under translations. By ergodicity (Lemma \ref{L:ergodicity}), its probability is thus 0 or 1. Since $\alpha \in \Ir_\tr^R$, its probability is positive and therefore equal to 1. This concludes.
\end{proof}

\subsection{Existence of phase transition}

In this section, we will show that all the critical intensities we introduced before are nondegenerate, i.e. strictly positive and finite. To this end, it will be enough to show that $\alpha_\b2b > 0$ and $\alpha_\tr^R < \infty$ for all $R>1$.

\begin{lemma}\label{L:finite_critical_point}
For all $R>1$, $\alpha_\tr^R < \infty$.
\end{lemma}

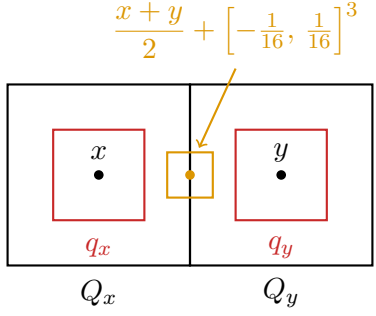
\begin{figure}
    \centering
\begin{tikzpicture}[scale=1.2]
\def\BigSize{2}      
\def\SmallSize{1}  
\def\TinySize{0.5}  

\definecolor{myblue}{RGB}{40,90,200}
\definecolor{myred}{RGB}{200,40,40}
\definecolor{myorange}{RGB}{220,150,0}

\coordinate (cx) at (-1,0);
\coordinate (cy) at (1,0);

\draw[thick] ($(cx)+(-\BigSize/2,-\BigSize/2)$)
    rectangle ($(cx)+(\BigSize/2,\BigSize/2)$);
\draw[thick] ($(cy)+(-\BigSize/2,-\BigSize/2)$)
    rectangle ($(cy)+(\BigSize/2,\BigSize/2)$);

\node at ($(cx)+(0,-\BigSize/2-0.3)$) {$Q_x$};
\node at ($(cy)+(0,-\BigSize/2-0.3)$) {$Q_y$};

\draw[myred, thick] ($(cx)+(-\SmallSize/2,-\SmallSize/2)$)
    rectangle ($(cx)+(\SmallSize/2,\SmallSize/2)$);
\draw[myred, thick] ($(cy)+(-\SmallSize/2,-\SmallSize/2)$)
    rectangle ($(cy)+(\SmallSize/2,\SmallSize/2)$);

\node[myred] at ($(cx)+(0,-\SmallSize/2-0.3)$) {$q_x$};
\node[myred] at ($(cy)+(0,-\SmallSize/2-0.3)$) {$q_y$};

\fill (cx) circle (1.5pt);
\fill (cy) circle (1.5pt);

\node at ($(cx)+(0,0.25)$) {$x$};
\node at ($(cy)+(0,0.25)$) {$y$};

\coordinate (mid) at ($(cx)!0.5!(cy)$);

\draw[myorange, thick] ($(mid)+(-\TinySize/2,-\TinySize/2)$) rectangle ($(mid)+(\TinySize/2,\TinySize/2)$);
\fill[myorange] (mid) circle (1.5pt);

\node[myorange] (midlabel) at (0.5,1.6) 
   {$\displaystyle \frac{x+y}{2} + \left[-\tfrac{1}{16},\, \tfrac{1}{16}\right]^3$};

\draw[myorange, thick, ->] (midlabel.south) -- ($(mid)+(0.1,0.3)$);
\end{tikzpicture}
\caption{Illustration of the proof of Lemma \ref{L:finite_critical_point}}\label{Figure1}
\end{figure}

\begin{proof}[Proof of Lemma \ref{L:finite_critical_point}]
Let $R>1$. We want to show that, if $\alpha$ is large enough, $\Lc^\alpha_{1,R}$ contains an unbounded cluster with positive probability (equivalently, with probability one).
To do so, we will show that our model stochastically dominates a Bernoulli percolation with finite range interaction and conclude with \cite{MR1428500}.
There are certainly many ways of deriving such a domination. We present one way which relies on paving $\R^3$ using cubes of sidelength $1/2$. Figure \ref{Figure1} contains an illustration of some of the notations we use.

In this proof, it will be convenient to view our collections of loops as collections of \emph{rooted} loops. Let $\alpha>0$ be a large intensity parameter. We can realise $\Lc^\alpha_{1,R}$ as the union of two independent collections of loops $\Lc^{\alpha/2,(1)}_{1,R}$ and $\Lc^{\alpha/2,(2)}_{1,R}$ each having the same law as $\Lc^{\alpha/2}_{1,R}$.

Let $x \in \frac12 \Z^3$. We define $Q_x = x + [-1/4,1/4]^3$ and $q_x = x + [-1/8,1/8]^3$.
For each Borel set $A \subset q_x$, let $g_x(A)$ be the probability that $A$ is hit by a Brownian trajectory starting uniformly on the boundary of $q_x$ and killed upon exiting $Q_x$.
Let $I^{(1)}_x$ be the indicator function of the event that there exists a loop $\wp \in \Lc^{\alpha/2,(1)}_{1,R}$ whose root belongs to $q_x$ and such that $g_x(\wp \cap q_x) \ge 1/2$. If $I^{(1)}_x=1$, we select a specific loop $\wp_x \in \Lc^{\alpha/2,(1)}_{1,R}$ realising the above event, uniformly at random among each possible loops. The probability $\P(I^{(1)}_x =1)$ does not depend on $x$ and goes to 1 as $\alpha \to \infty$. Moreover, the random variables $I_x^{(1)}$, $x \in \frac12 \Z^3$, are independent.

Now, let $\{x,y\}$ be an edge of $\frac12\Z^3$ such that $I^{(1)}_x=I^{(1)}_y=1$. Let $I^{(2)}_{\{x,y\}}$ be the indicator function of the event that there exists a loop $\wp \in \Lc^{\alpha/2,(2)}_{1,R}$ whose root belongs $(x+y)/2 + [-1/16,1/16]^3$ and such that $\wp$ intersects both $\wp_x$ and $\wp_y$. We declare the edge $\{x,y\}$ open if $I^{(1)}_x=I^{(1)}_y=1$ and $I^{(2)}_{\{x,y\}}=1$.
Thanks to the bounds $g_x(\wp_x \cap q_x), g_y(\wp_y \cap q_y) \ge 1/2$, there exists a deterministic function $u : (0,\infty) \to [0,1]$ with $u(\alpha) \to 1$ as $\alpha \to \infty$ such that
\[
\P( \{x,y\} \text{ open} \vert I^{(1)}_x=I^{(1)}_y=1, \Lc^{\alpha/2,(1)}_{1,R} ) \ge u(\alpha) \qquad \text{a.s.}
\]
Thus, $\P( \{x,y\} \text{ open})$ does not depend on the edge $\{x,y\}$ and goes to 1 as $\alpha \to \infty$.

The states of two different edges $\{x,y\}$ and $\{x',y'\}$ are independent as soon as they do not share any endpoint. Wrapping up, we have a one-dependent percolation model on the edges of $\frac12 \Z^3$ with $\P(\{x,y\} \text{ open})$ as close to 1 as desired. By \cite{MR1428500}, this model percolates almost surely if $\alpha$ is large enough. By construction, this builds an unbounded cluster of loops in $\Lc^\alpha_{1,R}$.
\end{proof}

\begin{lemma}\label{L:positive_critical_point}
The critical point
$\alpha_\b2b$ is positive: $\alpha_\b2b>0$.
\end{lemma}

\begin{proof}[Proof of Lemma \ref{L:positive_critical_point}]
Let $E$ be the event that there exists a ``crossing sheet'' in $[0,2]^2\times [0,1]$ that does not intersect any loop of $\Lc^\alpha$ and which separates $[0,2]^2 \times \{0\}$ and $[0,2]^2 \times \{1\}$. More precisely, the complementary event $E^c$ is defined to be the event that there is continuous path in
\[
([0,2]^2\times [0,1]) \cap \{ x \in \R^3, \exists \wp \in \Lc^\alpha, x \in \text{range}(\wp)\}
\]
joining $[0,2]^2 \times \{0\}$ and $[0,2]^2 \times \{1\}$.
We are going to show that
\begin{equation}\label{E:blocking_sheet}
\exists \hspace{1pt} \alpha_0 >0, \quad \forall \alpha \in (0,\alpha_0), \quad
\P(E) >0.
\end{equation}
Gluing several of these blocking sheets and by FKG inequality, we will immediately deduce that, if $\alpha \in (0,\alpha_0)$ and $R$ is large enough,
\[
\P(B(0,1) \overset{\Lc^\alpha}{\longleftrightarrow} \partial B(0,R)) < 1.
\]
By Lemma \ref{L:zero_one}, this will imply that $(0,\alpha_0) \subset (0,\infty) \setminus \Ir_\b2b$ and thus $\alpha_\b2b \ge \alpha_0 > 0$ as desired.

Our main task is now to prove \eqref{E:blocking_sheet}.
As in the article \cite{SheffieldWernerCLE} which considered the percolation of a 2D Brownian loop soup, we will compare our loop soup percolation to Mandelbrot's fractal percolation. This will be achieved by replacing loops by slightly larger cubes.  However, we will rely here on a 3D result derived in \cite{Chayes91} for the existence of blocking sheet. The article \cite{Chayes91} proves the existence of such a blocking sheet for an ``easy'' crossing, rather than a ``hard'' crossing as above. We will remedy this issue with a simple trick: we will compare loops to $4\times 4\times 1$ rectangular cuboids instead of cubes.

In this proof, we will use the following notations, for $n \in \N$:
\begin{gather*}
\mathcal{R}_n = \{ [0,8.2^{-n}]^2 \times [0,2.2^{-n}] + (i_1,i_2,i_3) 2^{-n} : i_1,i_2,i_3 \in \N\}, \quad \mathcal{R}=\bigcup_{n \in \N} \mathcal{R}_n,\\
\mathcal{Q}_n = \{ [0,2^{-n}]^3 + (8j_1,8j_2,2j_3) 2^{-n}: j_1,j_2,j_3 \in \N \},\quad \mathcal{Q}=\bigcup_{n \in \N} \mathcal{Q}_n,\\
\mathcal{Q}_n' = \{ [0,2^{-n}]^3 + (j_1,j_2,j_3) 2^{-n}: j_1,j_2,j_3 \in \N \},\quad \mathcal{Q}'=\bigcup_{n \in \N} \mathcal{Q}_n'.
\end{gather*}
We start by assigning to each loop $\wp\in \Lc^\alpha$ a specific rectangular cuboid $R(\wp) \in \mathcal{R}$ as follows.
Let $\wp \in \Lc^\alpha$. Let $d(\wp)$ be the diameter of $\wp$ with respect to the L$^1$-norm and let $n(\wp) \in \Z$ be the unique integer such that $d(\wp) \in [2^{-n(\wp)-1},2^{-n(\wp)})$. The loop $\wp$ can intersect at most 8 dyadic cubes of side length $2^{-n(\wp)}$. We assign to $\wp$ the minimal (with respect to some lexicographic order) dyadic vertex $v(\wp) = (i_1 2^{-n(\wp)}, i_2 2^{-n(\wp)},i_3 2^{-n(\wp)})$ such that $\wp$ intersects $[0,2^{-n(\wp)}]^3 + v(\wp)$. Finally, we define $R(\wp)$ to be the translated rectangular cuboid $[0,8.2^{-n(\wp)}]^2 \times [0,2.2^{-n(\wp)}]$ centred at $v(\wp)$. By construction, the loop $\wp$ is included in $R(\wp)$.

There is a positive probability that none of the loops $\wp \in \Lc^\alpha$ with $n(\wp) < 0$ (``big loops'') intersect $[0,2]^2 \times [0,1]$. In the following we will thus focus on the loops $\wp$ with $n(\wp) \ge 0$ (recall that in the definitions of $\mathcal{R}, \mathcal{Q}$ and $\mathcal{Q}'$, only nonnegative values of $n$ are considered).

We now define a variant of Mandelbrot's fractal percolation. 
For each $R\in \mathcal{R}$, let $X(R)$ be the indicator function of the event that there is no loop $\wp$ in $\Lc^\alpha$ such that $R(\wp) = R$.
The variables $X(R)$, $R\in \mathcal{R}$, are independent. Moreover, by invariance under scaling and translation, there exists $c>0$ such that for all $R \in \mathcal{R}$, $\P(R=1) = e^{-c \alpha}$. Consider the random compact
\[
K = ([0,2]^2 \times [0,1]) \setminus \bigcup_{R \in \mathcal{R}: X(R) = 0} R.
\]
As in \cite{SheffieldWernerCLE},
this percolation model dominates a genuine Mandelbrot's fractal percolation with a smaller parameter. Indeed, define for each cube $Q\in \mathcal{Q}_n$, $n \in \N$,
\[
\hat{X}(Q) = \min_{R \in \mathcal{R}_n: Q \subset R} X(R).
\]
Thanks to the room left between each cube of $\mathcal{Q}_n$, the variables $\hat{X}(Q), Q \in \mathcal{Q}$, are i.i.d. Bernoulli random variables with success parameter $e^{-128 c \alpha}$ (the above minimum ranges over a set of size 128).
Now, let $Q' \in \mathcal{Q}'$. The stretched set
\[
\{ (8x_1,8x_2,2x_3) : (x_1,x_2,x_3) \in Q' \} \in \mathcal{R}
\]
contains a unique cube $Q \in \mathcal{Q}$ and we set $X'(Q') = \hat{X}(Q)$. The random compact
\[
K' = ([0,1/4]^2 \times [0,1/2]) \setminus \bigcup_{Q' \in \mathcal{Q}': X'(Q)=0} Q'
\]
has the law of the open cluster of a Mandelbrot's fractal percolation with parameter $e^{-128c\alpha}$ in $[0,1/4]^2 \times [0,1/2]$. By \cite[Theorem 4]{Chayes91}, if $\alpha$ is small enough, there is a positive probability that $K'$ contains a crossing sheet separating $[0,1/4]^2 \times \{0\}$ and $[0,1/4]^2 \times \{1/2\}$. By construction, this shows, for such values of $\alpha$, the existence of a crossing sheet for $K$ with positive probability. Since each loop $\wp \in \Lc^\alpha$ is contained in $R(\wp)$, this proves \eqref{E:blocking_sheet} and concludes the proof of the lemma.
\end{proof}

We can conclude with a proof of Theorem \ref{T:phase_transition}.

\begin{proof}[Proof of Theorem~\ref{T:phase_transition}]
    It follows from Lemmas \ref{L:I3}, \ref{L:finite_critical_point} and \ref{L:positive_critical_point}.
\end{proof}

\section{Crossing exponent: Proof of Theorem \ref{T:crossing}}\label{S:one-arm}

This section is dedicated to the proof of Theorem \ref{T:crossing} concerning the asymptotic behaviour of
\begin{equation}\label{E:def_pr}
p_r := \P \Big( \partial B(1) \overset{\Lc_{\R^3}^\alpha}{\longleftrightarrow} \partial B(r) \Big), \qquad r \ge 0.
\end{equation}
We wish to show the existence of an exponent $\xi=\xi(\alpha) \ge 0$ such that $p_r = r^{\xi+o(1)}$ as $r \to 0$ and establish bounds on $\xi$.

To prove the existence of the exponent, one could be tempted to show an inequality of the form $p_{rs} \ge (rs)^{o(1)} p_r p_s$ whose proof could go along the following lines. If a cluster $\Cc$ intersects both $B(rs)$ and $\partial B(r)$ and a cluster $\Cc'$ intersects both $B(r)$ and $\partial B(1)$, it only remains to make an intermediate connection to join $\Cc$ and $\Cc'$ to have an overall crossing of $B(1) \setminus \overline{B(rs)}$. By FKG inequality, the probability of the intersection of these three events is at least the product of the three probabilities. If one can guarantee that the intermediate connection has a not-too-small probability ($(rs)^{o(1)}$ say), then we would get the desired bound $p_{rs} \ge (rs)^{o(1)} p_r p_s$.
However, this last statement requires a precise control on the shape of the clusters conditioned on the event that they realise difficult crossings.

To circumvent this difficulty, we prove a submultiplicative inequality, instead of a supermultiplicative inequality:

\begin{proposition}\label{P:submultiplicative}
Defining
\begin{equation}\label{E:def_Fs}
\Fs : s \in [0,1] \mapsto \int_s^1 \frac{1}{r^2} p_r \d r,
\end{equation}
there exists $C=C(\alpha)>0$ such that for all $s,s' \in [0,1]$,
\begin{equation}
\label{E:existence_exp0}
\Fs(ss')+1 \leq C (\Fs(s)+1) (\Fs(s')+1).
\end{equation}
\end{proposition}

We cannot use FKG inequality to prove \eqref{E:existence_exp0} since it goes in the wrong direction. Instead, we will control the deepest level reached by loops in $\{ \wp \in \Lc_{\R^3}^\alpha: \wp \cap (B(1) \setminus \overline{B(r)}) \neq \varnothing\}$ conditioned on the event that there exists a cluster crossing $B(1) \setminus \overline{B(r)}$; see Lemma \ref{L:small_loop}. This will allow us to recover a new independent set of loops.

\smallskip

The rest of this section is organised as follows.
Section \ref{SS:crossing1} contains the decoupling argument and proves Lemma \ref{L:small_loop}. Section \ref{SS:crossing2} proves Proposition \ref{P:submultiplicative} and deduce the existence of the exponent $\xi$ as well as its lower bound when $\alpha \notin \Ir_\b2b$. Finally, Section \ref{SS:crossing3} proves the upper bound on $\xi$.

\subsection{Decoupling step}\label{SS:crossing1}

The main result of this section is the following lemma.

\begin{lemma}\label{L:small_loop}
There exists $C=C(\alpha)>0$ such that for all $r \in (0,1)$ and $x \in (0,1)$,
\begin{equation}
\label{E:existence_exp5}
\P \Big( \exists \wp \in \Lc_{\R^3}^\alpha: \partial B(r) \overset{\wp}{\longleftrightarrow} \partial B(rx) \vert \partial B(1) \overset{\Lc_{\R^3}^\alpha}{\longleftrightarrow} \partial B(r)\Big) \leq Cx.
\end{equation}
\end{lemma}

By \eqref{E:L_estimate_crossing2}, the probability on the left hand side of \eqref{E:existence_exp5} without the conditioning equals $\alpha x + o(x)$ as $x \to 0$. Lemma \ref{L:small_loop} thus shows that the conditioning increases this probability by at most a multiplicative constant.
Its proof is based on an intermediate result, Lemma \ref{L:small_loop_intermediate} below.
For $r,x \in (0,1)$, let $\Lc_{r,x} := \{ \wp \in \Lc_{\R^3}^\alpha: \partial B(r) \overset{\wp}{\longleftrightarrow} \partial B(rx) \}$ and
\begin{equation}
\label{E:not_A_rx}
A_{r,x} := \overline{\bigcup_{\wp \in \Lc_{r,x}} \text{range}(\wp) \cap (\R^3 \setminus B(r))}.
\end{equation}

\begin{lemma}\label{L:small_loop_intermediate}
There exist $x_0=x_0(\alpha) \in (0,1)$ and $C=C(\alpha)>0$ such that for all $r \in (0,1), x \in (0,x_0]$,
\begin{equation}
\label{E:xx_0}
\P \Big( \partial B(1) \overset{\Lc^\alpha_{\R^3 \setminus B(r)}}{\longleftrightarrow} A_{r,x} \vert A_{r,x} \neq \varnothing \Big)
\leq C \P\Big( \partial B(1) \overset{\Lc^\alpha_{\R^3 \setminus B(r)}}{\longleftrightarrow} A_{r,x_0} \vert A_{r,x_0} \neq \varnothing\Big).
\end{equation}
\end{lemma}

\begin{proof}[Proof of Lemma \ref{L:small_loop_intermediate}]
Let $x_0 \in (0,1)$. For $n \geq 1$, let $P_n$ be the number of loops crossing $B(r) \setminus \overline{B(rx_0)}$ exactly $n$ times. By Corollary~\ref{C:crossing_exactly_n}, $P_n$ is a Poisson random variable with parameter $\lambda_n = (1+O(x_0))^n \frac{x_0^n}{n}$. Let $K = \sum_{n \geq 1} n P_n$ be the total number of crossings of $B(r) \setminus \overline{B(rx_0)}$ by all the loops in $\Lc_{\R^3}^\alpha$. For ease of future reference, we note that the moment generating function of $K$ has a positive radius of convergence $R(K) = (1+O(x_0))/x_0$ as $x_0 \to 0$. Indeed, this follows from the estimate on $\lambda_n$ and from the fact that for $t >0$,
\begin{align*}
\Expect{t^K} = \prod_{n=1}^\infty \Expect{t^{nP_n}} = \prod_{n=1}^\infty \exp \left( \lambda_n (t^n -1) \right).
\end{align*}

Conditioning on $\{A_{r,x} \neq \varnothing\}$ tends to make the random variable $K$ larger. We are first going to control this ``stochastic increase''; see \eqref{E:xx_01} below.
Let $k \geq 1$ and let $e_1, \dots, e_k$ be $k$ i.i.d. excursion from $\partial B(rx_0)$ to $\partial B(r)$ in $B(r)$ with starting points and ending points that are independent and uniform on both spheres. By Corollary~\ref{C:crossing_exactly_n} (more precisely, we use the bound on the third Radon-Nikodym derivative in \eqref{E:L_Radon-N}),
\begin{align*}
& \Prob{A_{r,x} \neq \varnothing \vert K=k}
\leq (1+O(x_0))^k \Prob{\exists j =1, \dots, k, e_j \cap B(rx) \neq \varnothing} \\
& = (1+O(x_0))^k ( 1 - \Prob{e_1 \cap B(rx) = \varnothing}^k ).
\end{align*}
The probability that $e_1$ reaches $B(rx)$ is comparable to the probability that a Brownian motion starting on $\partial B(rx_0)$ hits $\partial B(rx)$ before hitting $\partial B(r)$. This latter probability being equal to $x/x_0$, we obtain that
\[
\Prob{A_{r,x} \neq \varnothing \vert K=k} \leq (1+O(x_0))^k (1 - (1-Cx/x_0)^k) \leq C (1+O(x_0))^k k x/x_0.
\]
Rearranging and using that $\Prob{A_{r,x} \neq \varnothing} = (1+O(x)) \alpha x$ (which follows from \eqref{E:L_estimate_crossing2}),
\begin{equation}
\label{E:xx_01}
\Prob{K=k \vert A_{r,x} \neq \varnothing} \leq C x_0^{-1} (1+O(x_0))^k k \P(K=k).
\end{equation}

Let $\tilde{e}_1, \dots, \tilde{e}_k$ be i.i.d. excursions from $\partial B(r)$ to $\partial B(rx_0)$ in $\R^3 \setminus \overline{B(rx_0)}$ with starting and ending points that are uniform and independent on both spheres. Define
\[
q_r := \P\Big( \partial B(1) \overset{\Lc^\alpha_{\R^3 \setminus B(r)}}{\longleftrightarrow} \tilde{e}_1 \Big).
\]
Now, by Corollary~\ref{C:crossing_exactly_n} (more precisely, using the bound on the second Radon--Nikodym derivative in \eqref{E:L_Radon-N}),
\begin{align}
\label{E:xx_02}
\P\Big(\partial B(1) \overset{\Lc^\alpha_{\R^3 \setminus B(r)}}{\longleftrightarrow} A_{r,x} \vert A_{r,x} \neq \varnothing, K=k\Big)
& \leq (1+O(x_0))^k \P\Big(\partial B(1) \overset{\Lc^\alpha_{\R^3 \setminus B(r)}}{\longleftrightarrow} \bigcup_{i=1}^k \tilde{e}_i \Big) \\
& = (1+O(x_0))^k (1-(1-q_r)^k)
\leq (1+O(x_0))^k k q_r.
\nonumber
\end{align}
Combining \eqref{E:xx_01} and \eqref{E:xx_02}, we obtain that 
\begin{align*}
\P\Big(\partial B(1) \overset{\Lc^\alpha_{\R^3 \setminus B(r)}}{\longleftrightarrow} A_{r,x} \vert A_{r,x} \neq \varnothing\Big)
\leq C x_0^{-1} q_r \sum_{k=1}^\infty (1+O(x_0))^k k^2 \Prob{K=k}.
\end{align*}
By scaling, the sum on the right hand side of the above display depends only on $x_0$. Because the radius of convergence of the moment generating function of $K$ is equal to $(1+O(x_0))/x_0$, if $x_0$ is small enough, this sum is simply a finite constant. Wrapping things up, we have shown that the left hand side of \eqref{E:xx_0} is at most $C q_r$ for some constant $C$ that may depend on $x_0$. Using Corollary~\ref{C:crossing_exactly_n}, one can see that the right hand side of \eqref{E:xx_0} is at least $c q_r$ concluding the proof.
\end{proof}

\begin{proof}[Proof of Lemma \ref{L:small_loop}]
Let $r \in (0,1)$, $x_0 \in (0,1)$ as in Lemma \ref{L:small_loop_intermediate} and $x \in (0,x_0)$ (the result for $x \geq x_0$ is clear by bounding the left hand side of \eqref{E:existence_exp5} by 1). In this proof, we will denote by $\Lc = \Lc_{\R^3}^\alpha$, $\Lc_0 = \{\wp \in \Lc$ crossing $B(r) \setminus \overline{B(rx)} \}$, $\Lc_1 = \Lc \setminus \Lc_0$ and $\Lc_2 = \{ \wp \in \Lc: \wp \subset \R^3 \setminus B(r) \}$. The probability we are interested in is equal to
\[
\Prob{ \Lc_0 \neq \varnothing, \partial B(1) \overset{\Lc}{\longleftrightarrow} \partial B(r) } / \Prob{ \partial B(1) \overset{\Lc}{\longleftrightarrow} \partial B(r) }.
\]
The numerator is equal to
\begin{align}
\label{E:existence_exp4}
& \Prob{ \Lc_0 \neq \varnothing, \partial B(1) \overset{\Lc_1}{\longleftrightarrow} \partial B(r) } + \Prob{ \exists \wp \in \Lc_0, \partial B(1) \overset{\Lc_1}{\longleftrightarrow} \wp } \\
& + \Prob{\Lc_0 \neq \varnothing, \partial B(1) \overset{\Lc}{\longleftrightarrow} \partial B(r), \{\partial B(1) \overset{\Lc_1}{\longleftrightarrow} \partial B(r) \}^c, \{ \exists \wp \in \Lc_0, \partial B(1) \overset{\Lc_1}{\longleftrightarrow} \wp\}^c }.\nonumber
\end{align}
By independence of $\Lc_0$ and $\Lc_1$, the first probability agrees with
\[
\Prob{ \Lc_0 \neq \varnothing} \Prob{ \partial B(1) \overset{\Lc_1}{\longleftrightarrow} \partial B(r) } \leq \Prob{ \Lc_0 \neq \varnothing} \Prob{ \partial B(1) \overset{\Lc}{\longleftrightarrow} \partial B(r) }.
\]
The event appearing in the last probability in \eqref{E:existence_exp4} is contained in the event that there exists a loop in $\Lc$ crossing $B(1) \setminus B(rx)$. Concerning the second term in \eqref{E:existence_exp4}, we notice that if we need to use loops from $\Lc_1 \setminus \Lc_2$ to intersect a loop $\wp \in \Lc_0$, then there must be a cluster of $\Lc_1$ that crosses $B(1) \setminus B(r)$. In other words, the second term in \eqref{E:existence_exp4} is at most
\begin{align*}
\Prob{ \exists \wp \in \Lc_0, \partial B(1) \overset{\Lc_2}{\longleftrightarrow} \wp } + \Prob{ \Lc_0 \neq \varnothing, \partial B(1) \overset{\Lc_1}{\longleftrightarrow} \partial B(r) }.
\end{align*}
Putting things together, we have obtained the following upper bound for the left hand side of \eqref{E:existence_exp5}:
\begin{align}
\label{E:existence_exp6}
& 2 \Prob{ \Lc_0 \neq \varnothing } + \Prob{\exists \wp \in \Lc: \partial B(1) \overset{\wp}{\longleftrightarrow} \partial B(rx)} / \Prob{ \partial B(1) \overset{\Lc}{\longleftrightarrow} \partial B(r) } \\
& + \Prob{ \exists \wp \in \Lc_0, \partial B(1) \overset{\Lc_2}{\longleftrightarrow} \wp } / \Prob{ \partial B(1) \overset{\Lc}{\longleftrightarrow} \partial B(r) }.\nonumber
\end{align}
By \eqref{E:L_estimate_crossing1}, the probability in the numerator of the second term is at most $C rx$. Concerning the denominator of the same term, we simply bound it from below by the probability that the crossing has been realised by a single loop which is at least $cr$. The second term is therefore at most $Cx$. By \eqref{E:L_estimate_crossing1}, the first term is also at most $Cx$. It only remains to deal with the third term. Let
$A_{r,x}$ and $A_{r,x_0}$ be as in \eqref{E:not_A_rx}.
By Lemma \ref{L:small_loop_intermediate},
\begin{align*}
& \Prob{ \exists \wp \in \Lc_0, \partial B(1) \overset{\Lc_2}{\longleftrightarrow} \wp }
= \Prob{\Lc_0 \neq \varnothing} \Prob{ \partial B(1) \overset{\Lc_2}{\longleftrightarrow} A_{r,x} \vert A_{r,x} \neq \varnothing} \\
& \leq C \Prob{\Lc_0 \neq \varnothing} \Prob{ \partial B(1) \overset{\Lc_2}{\longleftrightarrow} A_{r,x_0} \vert A_{r,x_0} \neq \varnothing}.
\end{align*}
The probability that a cluster of $\Lc_2$ intersects $\partial B(1)$ and a loop in $A_{r,x_0}$ is at most the probability that a cluster of $\Lc$ intersects both $\partial B(1)$ and $\partial B(r)$ (i.e. the denominator of the third term in \eqref{E:existence_exp6}).
This shows that the third term of \eqref{E:existence_exp6} is also at most $C \Prob{\Lc_0 \neq \varnothing} \leq C x$. This concludes the proof.
\end{proof}

\subsection{Existence of the exponent}\label{SS:crossing2}

\begin{proof}[Proof of Proposition \ref{P:submultiplicative}]
Let $s,s', r \in (0,1)$. Let $E_r$ be the event that there is a cluster of $\Lc_{\R^3}^\alpha$ crossing $B(1) \setminus \overline{B(r)}$. We can write
\begin{equation}
\label{E:existence_exp2}
p_{rs} = p_r \Prob{ \partial B(1) \overset{\Lc_{\R^3}^\alpha}{\longleftrightarrow} \partial B(rs) \vert E_r}.
\end{equation}
We decompose $\Lc_{\R^3}^\alpha$ into two independent sets of loops: $\Lc_1$ and $\Lc_2$ respectively formed of loops that touch and that do not touch $B(1) \setminus \overline{B(r)}$.
Define the random variables
\[
R_1 := \inf \{ R < r : \exists \wp \in \Lc_{\R^3}^\alpha, \partial B(r) \overset{\wp}{\longleftrightarrow} \partial B(R) \}
\]
and
\[
R_2 := \sup \{ R> rs: \partial B(R) \overset{\Lc_2}{\longleftrightarrow} \partial B(rs) \}.
\]
The event $E_r$ and the random variable $R_1$ are measurable with respect to $\Lc_1$ whereas the random variable $R_2$ is measurable with respect to $\Lc_2$.
On the event $E_r$, $\{ \partial B(1) \overset{\Lc_{\R^3}^\alpha}{\longleftrightarrow} \partial B(rs) \} \subset \{ R_2 \geq R_1 \}$.
Hence
\begin{equation}
\label{E:existence_exp1}
\Prob{ \partial B(1) \overset{\Lc_{\R^3}^\alpha}{\longleftrightarrow} \partial B(rs) \vert E_r}
\leq \Expect{ \Prob{R_1 \leq R_2 \vert R_2, E_r} }.
\end{equation}
By Lemma \ref{L:small_loop}, $\Prob{R_1 \leq R_2 \vert R_2, E_r} \leq C (R_2/r \indic{R_2/r \leq 1} + \indic{R_2/r >1})$ a.s. Integrating by parts we thus get that
\begin{align*}
& \Expect{ \Prob{R_1 \leq R_2 \vert R_2, E_r} }
\leq C \Expect{ R_2/r \indic{R_2/r \leq 1} } + C \Prob{R_2/r >1}
= C s + C \int_s^1 \Prob{R_2/r \geq y} \d y.
\end{align*}
Adding more loops does not decrease the crossing probability, so
\[
\Prob{R_2/r \geq y} \leq \Prob{ \partial B(ry) \overset{\Lc_{\R^3}^\alpha}{\longleftrightarrow} \partial B(rs) } = p_{s/y},
\]
by scaling. Overall, we obtain that the left hand side of \eqref{E:existence_exp1} is at most
$
C s + C\int_s^1 p_{s/y} \d y.
$
With a change of variable, and recalling the definition \eqref{E:def_Fs} of $\Fs$, we observe that this is equal to $Cs (\Fs(s)+1)$. Going back to \eqref{E:existence_exp2}, we obtained that
\[
p_{rs} \leq C p_r s (\Fs(s)+1).
\]
Multiplying this inequality by $1/r^2$ and integrating from $r=s'$ to $r=1$, we get that
\begin{align*}
\int_{s'}^1 \frac{1}{r^2} p_{rs} \d r \leq C \Fs(s') s (\Fs(s)+1).
\end{align*}
Since the left hand side term is equal to
\[
s \int_{ss'}^s \frac{1}{r^2} p_r \d r = s (\Fs(ss') - \Fs(s)),
\]
this inequality becomes
\[
\Fs(ss') - \Fs(s) \leq C \Fs(s')(\Fs(s)+1)).
\]
The submultiplicative inequality \eqref{E:existence_exp0} then follows.
\end{proof}

\begin{corollary}
There exists $\xi \in [0,1]$ such that $p_r = r^{\xi + o(1)}$ as $r \to 0$. Moreover, $\xi$ is strictly positive as soon as $p_r \to 0$ as $r \to 0$ (i.e. when $\alpha \in \Ir_\b2b^c$).
\end{corollary}

\begin{proof}
By Proposition~\ref{P:submultiplicative} and Fekete's subadditive lemma, there exists $\xi_0 \in \R$ such that $\Fs(s) = s^{-\xi_0+o(1)}$ as $s \to 0$.
Because $r \mapsto p_r$ is nondecreasing, we deduce that
\[
p_s = p_s \frac{s}{1-s} \int_s^1 \frac{\d r}{r_2}
\leq \frac{s}{1-s} \int_s^1 \frac{\d r}{r_2} p_r = \frac{s}{1-s} \Fs(s) = s^{1-\xi_0+o(1)}.
\]
To get a lower bound, we first notice that (using \eqref{E:L_estimate_crossing2})
\begin{equation}
\label{E:existence_exp3}
p_r \geq \Prob{ \exists \wp \in \Lc_{\R^3}^\alpha: \partial B(1) \overset{\wp}{\longleftrightarrow} \partial B(r) } \geq c r, \quad \quad r \in (0,1).
\end{equation}
This implies in particular that
$\Fs(s) \geq c |\log s|$ so the exponent $\xi_0$ is nonnegative. Let $\eps >0$. We claim that $\Fs(s) - \Fs(s^{1-\eps}) \geq s^{-\xi_0+o(1)}$ as $s \to 0$. Indeed, this follows directly from the estimate $\Fs(s) = s^{-\xi_0+o(1)}$ if $\xi_0 >0$. If $\xi_0=0$, we use \eqref{E:existence_exp3} above to bound
\[
\Fs(s) - \Fs(s^{1-\eps}) \geq c \int_s^{s^{1-\eps}} \frac{\d r}{r} \geq c \eps \log s = s^{o(1)}.
\]
Combining the inequality $\Fs(s) - \Fs(s^{1-\eps}) \geq s^{-\xi_0+o(1)}$ with
\[
\Fs(s) - \Fs(s^{1-\eps}) \leq p_{s^{1-\eps}} \int_s^{s^{1-\eps}} \frac{\d r}{r^2} = p_{s^{1-\eps}} (s^{-1}-s^{-1+\eps}),
\]
we obtain that $p_{s^{1-\eps}} \geq s^{1-\xi_0+o(1)}$. Since this is true for every $\eps>0$, this shows that $p_s \geq s^{1-\xi_0+o(1)}$. This concludes the proof of the existence of the exponent $\xi = 1-\xi_0$.

Let us now assume that $p_r \to 0$ as $r \to 0$. We want to show that $\xi >0$. Let us write the constant $C$ appearing on the right hand side of the submultiplicative inequality \eqref{E:existence_exp0} as $C=1/\eps$, for some $\eps >0$. Because $p_r \to 0$, $\Fs(s) = o(s^{-1})$ as $s \to 0$ and there exists $s_0 >0$ small enough such that $\Fs(s_0) + 1 \le \eps^2 s_0^{-1}$. Injecting this estimate recursively in \eqref{E:existence_exp0}, we get that for all $n \ge 1$,
$\Fs(s_0^n) + 1 \le \eps^{n+1} s_0^{-n}$. This shows that the exponent
$\xi_0 = \lim_{s \to 0} \log \Fs(s) / |\log s|$ is strictly less than 1 and that $\xi=1-\xi_0$ is strictly positive.
\end{proof}

\subsection{Lower bound on the crossing probability}\label{SS:crossing3}

\begin{lemma}\label{L:lower_bound}
For all $\alpha>0$, there exist $c=c(\alpha)>0$ and $\eps=\eps(\alpha)>0$ such that $p_r \geq c r^{1-\eps}$ for all $r \in (0,1)$.
\end{lemma}

\begin{proof}
The idea of the proof is elementary. We will explore the cluster of $\partial B(1)$ step by step. First, we look at the smallest radius $\mathbf{R}_1$ that can be reached with a loop $\wp_1$ that intersects $\partial B(1)$ and that has a large capacity (more precisely, we will require $g_{\mathbf{R}_1,a}(\wp_1) \geq 1/2$, recalling \eqref{E:g_ra}). If $\mathbf{R}_1$ is not small enough, we will stop the procedure. If $\mathbf{R}_1$ is small enough, we iterate and look at the loops that intersect $\wp_1$ (instead of $\partial B(1)$). Thanks to the requirement on the capacity of $\wp_1$, connecting to $\wp_1$ will be essentially the same as connecting to the sphere $\partial B(\mathbf{R}_1)$. The loops involved in this second step will live in the domain $B(1)$ (since we know that no loop intersecting $\partial B(1)$ crosses further than $\partial B(\mathbf{R}_1)$). The condition that $\mathbf{R}_1$ is small enough in turn gives enough room for these loops.

We now define precisely this procedure.
Let $a>0$ be large enough so that Lemma \ref{L:intersection_loop} applies to $\eta=1/2$.
We are going to define a sequence of random variables $(\mathbf{R}_i)_{i \geq 0} \subset (0,1)$ and a non-decreasing sequence $(\Cc_i)_{i \geq 0}$ of subsets of $\R^3$ inductively as follows.
Let $\mathbf{R}_0 =1$, $\Cc_0 = \partial B(1)$. Let $i \geq 0$ and assume that we have defined $\mathbf{R}_i$ and $\Cc_i$. If $\mathbf{R}_i = +\infty$ (meaning that the procedure has already stopped), we set $\mathbf{R}_{i+1}=+\infty$ and $\Cc_{i+1} = \Cc_{i}$. Otherwise, let
\begin{equation}
R_{i+1} := \min \{ r \in (0,1): \exists \wp \in \Lc_{a^2 \mathbf{R}_i \D}^\alpha: \Cc_i \overset{\wp}{\longleftrightarrow} r \mathbf{R}_i \D, g_{r\mathbf{R}_i,a}(\wp) \geq 1/2 \},
\end{equation}
where we recall that $g_{r\mathbf{R}_i,a}(\wp)$ is defined in \eqref{E:g_ra}.
If $R_{i+1} > a^{-2}$, we consider that we do not have enough room to continue the procedure and we set $\mathbf{R}_{i+1}=+\infty$ and $\Cc_{i+1} = \Cc_{i}$. If $R_{i+1} \leq a^{-2}$, we instead define $\mathbf{R}_{i+1}= \mathbf{R}_i R_{i+1}$ and $\Cc_{i+1}$ to be the union of $\Cc_i$ and of all the loops that intersect $\Cc_i$.

Let $I = \inf \{ i \geq 1: \mathbf{R}_i = +\infty \}$. By construction,
$
p_r \geq \Prob{\mathbf{R}_{I-1} \leq r}.
$
By scaling and Lemma~\ref{L:intersection_loop}, there exists $\eps >0$ such that for all $i \geq 0$ and $r \in (0,1)$,
\begin{align*}
\Prob{R_i \leq r \vert (R_j,\Cc_j)_{j=1}^{i-1}, I > i-1} \geq \eps r.
\end{align*}
Iterating this estimate implies that for all $i \geq 1$ and $r_1, \dots, r_i \in (0,1)$,
\[
\Prob{R_1 \leq r_1, \dots, R_i \leq r_i, I > i} \geq \eps^i r_1 \dots r_n.
\]
Hence $(|\log R_1|, \dots, |\log R_{I-1}|, I)$ stochastically dominates $(E_1, \dots, E_{J-1}, J)$ where $J$ is a geometric random variable with success parameter $1-\eps$ and $E_i, i\geq 1$, are i.i.d. exponential random variables with parameter 1 independent of $J$.
One can compute explicitly the law of $\sum_{i=1}^{J-1} E_i$: it is equal to
\[
(1-\eps) \delta_0 + \frac{\eps}{1-\eps} e^{-(1-\eps)t} \indic{t>0} \d t.
\]
Wrapping up,
\[
p_r \ge 
\P ( \mathbf{R}_{I-1} \leq r ) = \P \Big( \sum_{i=1}^{I-1} |\log R_i| \geq \log r \Big) \geq \P \Big(\sum_{i=1}^{J-1} E_i \geq |\log r|\Big) = \eps r^{1-\eps}
\]
as desired.
\end{proof}

\section{Strict monotonicity of \texorpdfstring{$\alpha_\tr^R$}{inf JR}: Proof of Theorem~\ref{T:JR}, \protect\eqref{E:TJR}}\label{S:AG}

This section is devoted to the proof of Theorem \ref{T:JR}, \eqref{E:TJR} stating that the critical point $\alpha_\tr^R$ \eqref{E:def_alpha_R} is strictly decreasing with $R$. Using an Aizenman--Grimmett type argument \cite{MR1116036}, it boils down to the following key lemma (which will be used once more in the proof of Theorem \ref{T:JR}, \eqref{E:T_locuniq}). Crucially, the estimate \eqref{E:L_key} below is uniform over all $K_1$ and $K_2$. We state it for any dimension $d \ge 1$ since it could be useful in other contexts.

\begin{lemma}\label{L:key}
Let $d \ge 1$ and $R' > R > 1$. There exists  $c=c(d,R,R')>0$ such that for all Borel sets $K_1, K_2 \subset \R^d$,
\begin{equation}\label{E:L_key}
\loopmeasure_{\R^d}( \diam(\wp) \in [R,R'], K_1 \overset{\wp}{\longleftrightarrow} K_2 )
\ge c \loopmeasure_{\R^d}( \diam(\wp) \in [1,R], K_1 \overset{\wp}{\longleftrightarrow} K_2 ).
\end{equation}
\end{lemma}

A similar statement is proven in \cite{chang2017supercritical} for discrete loop soups; see Lemma~3.1 therein. Its proof is rather direct and uses in an essential way the discrete nature of the problem (for instance the bound (6) in \cite{chang2017supercritical} cannot hold in the continuum). 

Before proving Lemma \ref{L:key}, let us show how Theorem \ref{T:JR} follows. This step could be considered routine for percolation experts since it follows the standard strategy developed by Aizenman and Grimmett \cite{MR1116036}.

\begin{proof}[Proof of Theorem \ref{T:JR}, assuming Lemma \ref{L:key}]
Let $R'>R>1$ and $c>0$ be the same constant as in Lemma \ref{L:key}. Let $\alpha >0$, $n \ge 1$.
Consider the following family of measures on loops
\[
\mu(t)(\d \wp) = \alpha ((1+c t) \indic{\diam(\wp) \in [1,R)} + (1-t) \indic{\diam(\wp) \in [R,R']}) \loopmeasure(\d \wp), \quad t \in [0,1],
\]
interpolating between
\[
\mu(0)(\d\wp) = \alpha \indic{\diam(\wp) \in [1,R']} \loopmeasure(\d \wp)
\quad \text{and} \quad
\mu(1)(\d\wp) = \alpha (1+c) \indic{\diam(\wp) \in [1,R)} \loopmeasure(\d \wp).
\]
Denote by $\Lc(t)$ a random collection of loops sampled according to a Poisson point process of intensity $\mu(t)$ and let
\[
\Theta_n (t) := \P(\S^2 \overset{\Lc(t)}{\longleftrightarrow} n \S^2).
\]
Denoting by $\Cc_1(t)$ and $\Cc_n(t)$ the union of all clusters of $\Lc(t)$ intersecting $\S^2$ and $n\S^2$ respectively, a simple Poisson point process computation shows that
\begin{align*}
\frac{\partial}{\partial t} \Theta_n (t)
= \E[ \indic{\Cc_1(t) \neq \Cc_n(t)} \frac{\partial}{\partial t} \mu(t)(\Cc_1(t) \overset{\wp}{\longleftrightarrow} \Cc_n(t) ) ],
\end{align*}
where $\mu(t)$ only acts on $\wp$.
By definition of $\mu(t)$,
\begin{align*}
\frac{\partial}{\partial t} \mu(t)(\Cc_1(t) \overset{\wp}{\longleftrightarrow} \Cc_n(t) )
& = c \loopmeasure( \diam(\wp) \in [1,R], \Cc_1(t) \overset{\wp}{\longleftrightarrow} \Cc_n(t)) \\
& - \loopmeasure( \diam(\wp) \in [R,R'], \Cc_1(t) \overset{\wp}{\longleftrightarrow} \Cc_n(t)).
\end{align*}
By Lemma \ref{L:key} (and our choice of constant $c$), the right hand side is nonpositive. We deduce that $\frac{\partial}{\partial t} \Theta_n(t) \le 0$ and thus $\Theta_n(1) \le \Theta_n(0)$. In particular,
\[
\lim_{n \to \infty} \P ( \S^2 \overset{\Lc_{1,R}^{\alpha(1+c)}}{\longleftrightarrow} n \S^2 ) \le \lim_{n \to \infty} \P ( \S^2 \overset{\Lc_{1,R'}^{\alpha}}{\longleftrightarrow} n \S^2 ).
\]
Recall from \eqref{E:L_Jr1} that the right hand side vanishes if, and only if, $\alpha \notin \Ir_\tr^{R'}$ and similarly for the left hand side.
The above inequality therefore implies that $\alpha_\tr^{R'} (1+c) \le \alpha_\tr^R$. Because $\alpha_\tr^{R'} >0$ (by Lemma~\ref{L:positive_critical_point}), this implies that $\alpha_\tr^{R'} < \alpha_\tr^R$ as desired.
\end{proof}

The rest of this section is then dedicated to the proof of Lemma~\ref{L:key}.

\begin{proof}[Proof of Lemma \ref{L:key}]
The main idea of the proof is to cut the loop into two parts where one part intersects $K_1$ and $K_2$ and the other part achieves the appropriated diameter. Because $K_1$ and $K_2$ are arbitrary, the first part can have a complicated law. But the second part has an explicit distribution so that we can change the overall diameter at will. As already alluded to, Lemma~\ref{L:double} was actually first derived with this goal in mind, but turned out to be non tractable (mainly because the resampling of the root of the loop induces the appearance of new terms, such as the self-intersection local time, that have to be dealt with). We will use another cutting procedure that is more suited to our specific situation.

Consider a loop $\wp$. We are going to define three events $E_1$, $E_2$ and $E_3$ for the loop $\wp$ in such a way that
\begin{equation}\label{E:events3}
\{ K_1 \overset{\wp}{\longleftrightarrow} K_2 \} \cap \{\diam(\wp) \ge 1\}
\subset E_1 \cup E_2 \cup E_3.
\end{equation}
We first define stopping times using the convention that $\inf \varnothing = +\infty$ and $\sup \varnothing = - \infty$. Let $\tau_1$ be the first hitting time of $K_1 \cup K_2$ and $i \in \{1,2\}$ be such that $\wp(\tau_1) \in K_i$. Let $\tau_2$ be the first hitting time of $K_{2-i}$ after $\tau_1$. Finally, let $\tau_3$ be the first hitting time of $\partial B(\wp(0),1/4)$ after $\tau_2$.
We will denote by $E_1$ the event that $\tau_3 < T(\wp)$.
The event $E_2$ is then the event that $E_1$ holds for the time reversal $(\wp(T(\wp)-t))_{0 \le t \le T(\wp)}$.
Let
\[
\sigma_2 = \inf \{ t\ge 0: \wp(t) \in \partial B(\wp(0),1/4) \},
\quad \sigma_1 = \sup \{ t \in [0,T(\wp)]: \wp(t) \in \partial B(\wp(0),1/4) \}.
\]
Denoting by $[\sigma_1,\sigma_2]\subset T(\wp) \S^1$ the counter clockwise segment with endpoints $\sigma_1$ and $\sigma_2$,
let $E_3$ be the event that $\sigma_1 > -\infty$, $\sigma_2 < \infty$ and $\wp_{\vert [\sigma_1,\sigma_2]}$ intersects both $K_1$ and $K_2$.

Let us now show that \eqref{E:events3} holds. To this end, it is enough to show that
\[
E_3^c \cap \{ K_1 \overset{\wp}{\longleftrightarrow} K_2 \} \cap \{\diam(\wp) \ge 1\} \subset E_1 \cup E_2.
\]
We work on the event on the left hand side of the above display.
Because $\diam(\wp) \ge 1$, $\wp$ cannot stay in the ball $B(\wp(0),1/4)$ and we must have $\sigma_1 > -\infty$, $\sigma_2 < \infty$. Hence, $\wp_{\vert [\sigma_1,\sigma_2]}$ does not intersect at least one of the $K_i$. Without loss of generality, assume that it does not intersect $K_1$.
Because $\wp$ intersects both $K_1$ and $K_2$, $\tau_1$ and $\tau_2$ are finite.
To finish the proof of \eqref{E:events3}, we consider the following cases:
\begin{itemize}
\item
Case 1: $\wp(\tau_2) \in K_1$. Because $\wp_{\vert [\sigma_1,\sigma_2]}$ does not intersect $K_1$, we must have $\tau_2 < \sigma_1$ and thus $\tau_3 \le \sigma_1$. This shows that $E_1$ holds.
\item
Case 2: $\wp(\tau_2) \in K_2$. By definition of $\tau_1$ and $\tau_2$, we must have $\wp(\tau_1) \in K_1$. The same reasoning as in Case 1 for the reverse loop then applies, showing that $E_2$ holds.
\end{itemize}

As a consequence of \eqref{E:events3}, we have
\begin{align*}
\loopmeasure_{\R^d}( \diam(\wp) \in [1,R], K_1 \overset{\wp}{\longleftrightarrow} K_2 )
\le \sum_{i=1}^3 \loopmeasure_{\R^d}( \diam(\wp) \in [1,R], K_1 \overset{\wp}{\longleftrightarrow} K_2 ,E_i).
\end{align*}
By symmetry, the terms for $i=1$ and $i=2$ are equal. We will thus only consider $E_1$ and $E_3$.

\noindent
\textbf{Contribution of $E_1$.}
Let $x \in \R^d$ and let $\nu_x$ be the joint law of $(\wp(\tau_3),\tau_3)$ where $\wp \sim \P^x$ is a Brownian trajectory starting from $x$. We view $\nu_x$ as a probability measure on $(\R^d \cup \{\infty\}) \times ( \R \cup \{\infty\} )$. For $y \in \R^d$ and $s>0$, denote by $\tilde{\P}^{x,y;s}$ the law of $\wp_{\vert[0,s]}$ under $\P^x$, conditionally on $(\wp(\tau_3),\tau_3) = (y,s)$. The law $\tilde{\P}^{x,y;s}$ is well defined for $\nu_x$-almost every $(y,s)$. It is potentially complicated and we have no control over it.
By Markov's property, for any $x \in \R^d$ and $t>0$,
\[
p_t^{\R^d}(x,x) \P^{x,x;t} \mathbf{1}_{E_1}
= \int_{\R^d \times (0,t)} \nu_x(\d y \d s) p_{t-s}^{\R^d}(y,x) \tilde{\P}^{x,y;s} \wedge \P^{y,x;t-s},
\]
where ``$\wedge$'' means that the two paths are concatenated.
Importantly, the second part of the trajectory is simply a Brownian bridge. Integrating over $x$ and $t$ and then exchanging integrals and doing the change of variables $u = t-s$, we find that
\begin{align}\label{E:pf_key1}
\mathbf{1}_{E_1} \loopmeasure(\d \wp) & = \int_{\R^d} \d x \int_0^\infty \frac{\d t}{t} \int_{\R^d \times (0,t)} \nu_x(\d y \d s) p_{t-s}^{\R^d}(y,x) \tilde{\P}^{x,y;s} \wedge \P^{y,x;t-s} \\
& = \int_{\R^d} \d x \int_{\R^d \times (0,\infty)} \nu_x(\d y \d s) \tilde{\P}^{x,y;s} \wedge \Big( \int_0^\infty \frac{\d u}{s+u}p_u^{\R^d}(y,x) \P^{y,x;u} \Big).\notag
\end{align}
In particular,
\begin{align*}
& \loopmeasure_{\R^d}( \diam(\wp) \in [1,R], K_1 \overset{\wp}{\longleftrightarrow} K_2 ,E_1)\\
& = \int_{\R^d} \d x \int_{\R^d \times (0,\infty)} \nu_x(\d y \d s) \int \tilde{\P}^{x,y;s}(\d \wp_1) \indic{\diam(\wp_1) \in [1,R]} \\
& \hspace{130pt} \times \Big( \int_0^\infty \frac{\d u}{s+u}p_u^{\R^d}(y,x) \int \P^{y,x;u}(\d \wp_2) \indic{\diam(\wp_1 \wedge \wp_2) \in [1,R]} \Big).
\end{align*}
We are going to show that the term in parenthesis is at most a constant times the same expression where the diameter of $\wp_1 \wedge \wp_2$ is required to belong to $[R,R']$ instead.
Recalling that $\|x-y\|=1/4$, we have $p_u^{\R^d}(y,x) = (2\pi u)^{-d/2} e^{-\frac{1}{32u}}$ and
\begin{align*}
& \int_0^\infty \frac{\d u}{s+u}p_u^{\R^d}(y,x) \int \P^{y,x;u}(\d \wp_2) \indic{\diam(\wp_1 \wedge \wp_2) \in [1,R]}
\le \int_0^\infty \frac{\d u}{s+u}p_u^{\R^d}(y,x) \P^{y,x;u}(\diam(\wp_2) \le R)\\
& \le C \int_0^\infty \frac{\d u}{(s+u)u^{d/2}} e^{-\frac{1}{32u}} e^{-cu} \le C \min(1,s^{-1}).
\end{align*}
Similarly, for any $\wp_1$ with $\diam(\wp_1) \in [1,R]$ and endpoints $x$ and $y$ at distance $1/4$, for any $u>0$, one has
\[
\int \P^{y,x;u}(\d \wp_2) \indic{\diam(\wp_1 \wedge \wp_2) \in [R,R']}
\ge c e^{-Cu - C/u}
\]
and thus
\[
\int_0^\infty \frac{\d u}{s+u}p_u^{\R^d}(y,x) \int \P^{y,x;u}(\d \wp_2) \indic{\diam(\wp_1 \wedge \wp_2) \in [R,R']} \ge c \min(1,s^{-1}).
\]
We deduce that 
\begin{align*}
& \loopmeasure_{\R^d}( \diam(\wp) \in [1,R], K_1 \overset{\wp}{\longleftrightarrow} K_2 ,E_1)\\
& \le C \int_{\R^d} \d x \int_{\R^d \times (0,\infty)} \hspace{-4pt} \nu_x(\d y \d s) \int \tilde{\P}^{x,y;s}(\d \wp_1) \indic{\diam(\wp_1) \in [1,R]} \\
& \hspace{130pt} \times \Big( \int_0^\infty \frac{\d u}{s+u}p_u^{\R^d}(y,x) \int \P^{y,x;u}(\d \wp_2) \indic{\diam(\wp_1 \wedge \wp_2) \in [R,R']} \Big)\\
& = C \int_{\R^d} \d x \int_{\R^d \times (0,\infty)} \hspace{-4.1pt}\nu_x(\d y \d s) \int \tilde{\P}^{x,y;s}(\d \wp_1) \int_0^\infty \hspace{-4pt} \frac{\d u}{s+u}p_u^{\R^d}(y,x) \int \P^{y,x;u}(\d \wp_2) \indic{\diam(\wp_1 \wedge \wp_2) \in [R,R']}.
\end{align*}
Unwrapping the above procedure, we have obtained that
\[
\loopmeasure_{\R^d}( \diam(\wp) \in [1,R], K_1 \overset{\wp}{\longleftrightarrow} K_2 ,E_1)
\le C \loopmeasure_{\R^d}( \diam(\wp) \in [R,R'], K_1 \overset{\wp}{\longleftrightarrow} K_2 ,E_1).
\]

\noindent
\textbf{Contribution of $E_3$.}
The contribution of the event $E_3$ can be treated in a very similar way. The portion $\wp_{\vert [\sigma_1,\sigma_2]}$ intersecting both $K_1$ and $K_2$ can have a complicated law, but the remaining of the trajectory is simply a Brownian excursion measure. 
The biggest difference comes from the fact that the two endpoints $\wp(\sigma_1)$ and $\wp(\sigma_2)$ can be arbitrarily close to each other (they are arbitrary points of $\partial B(x,1/4)$ for some $x \in \R^d$). However, $\wp_{\vert [\sigma_1,\sigma_2]}$ stays in a ball of diameter $1/2$. Thus, the remaining of the trajectory is required to go at a macroscopic distance from its starting point in order to achieve an overall diameter in $[1,R]$ or $[R,R']$. In both cases, the probability is of the same order and we can conclude as before that
\[
\loopmeasure_{\R^d}( \diam(\wp) \in [1,R], K_1 \overset{\wp}{\longleftrightarrow} K_2 ,E_3)
\le C \loopmeasure_{\R^d}( \diam(\wp) \in [R,R'], K_1 \overset{\wp}{\longleftrightarrow} K_2 ,E_3).
\]
This concludes the proof of Lemma \ref{L:key}.
\end{proof}

\section{Local uniqueness: Proof of Theorem~\ref{T:JR}, \eqref{E:T_locuniq}}\label{S:GM}

This section is dedicated to the two equalities stated in \eqref{E:T_locuniq}. We will prove the second one and explain at the end of the section how to modify the arguments to prove the first one; see Section~\ref{SS:pf_1st_equality}.

\begin{proof}[Proof of Theorem \ref{T:JR}, \eqref{E:T_locuniq}, second equality, assuming Proposition \ref{P:local_main}]
Let 
\[ 
\loc = \{ \alpha >0: \forall~D\subset\R^3 \text{ connected and open}, \P(\Lc^\alpha_D \text{ contains a unique cluster}) =1 \}.
\]
We start by showing the easy inclusion $\bigcup_{R>1} \Ir_\tr^R \supset \loc$ which follows from a ``static'' renormalisation procedure. Let $\alpha \in \loc$ and $\delta >0$. For any domain $D \subset \R^3$, let $\Cc_{D,\delta}$ be the cluster of $\{\wp \in \Lc^\alpha_D: \diam(\wp)>\delta \}$ whose diameter is the biggest, say.
Consider the domains $D = [0,2] \times [0,1]^2$ and $D' = [0,1] \times [0,2] \times [0,1]$. The diameter of the loops staying in these domains does not exceed $\sqrt{6}$.
By uniqueness of the cluster in $\Lc^\alpha_D$, $\Cc_{D,\delta}$ contains some loops that stay entirely in $D \cap D'$ when $\delta$ is small enough. The same is true for $D'$. But since there is a unique cluster in $\Lc_{D \cap D'}^\alpha$, this shows that 
\[
\P(\Cc_{D,\delta} \cap \Cc_{D',\delta} \neq \varnothing) \to 1, \qquad \text{as} \quad \delta \to 0.
\]
This is the crucial step to be able to compare the percolation of $\Lc_{\delta,\sqrt{6}}^\alpha$ with a percolation on the sites of $\Z^3$ with a finite range of dependence and such that the probability $p(\delta)$ for a site to be open can be made arbitrarily close to 1. Since this is similar to the proof of Lemma~\ref{L:finite_critical_point}, we omit the details. We deduce that, if $\delta$ is small enough, $\Lc_{\delta,\sqrt{6}}^\alpha$ contains an unbounded cluster a.s. By scaling, this shows that $\alpha \in \Ir_\tr^{\sqrt{6}/\delta}$ which shows the desired inclusion: $\loc \subset \bigcup_{R>1} \Ir_\tr^R$.

We now initiate the proof of the reverse inclusion with some preliminary reduction steps.
Let $\alpha \in \Ir_\tr^R$ for some $R>1$. Let $D \subset \R^3$ be a connected open set. We want to show that for any $\wp, \wp' \in \Lc_D^\alpha$, $\wp$ and $\wp'$ belong to the same cluster a.s. By Palm's formula (Lemma~\ref{E:Palm}), it is enough to show that for any $\wp, \wp'$ independent loops sampled according to $\alpha \loopmeasure_D$ and for $\Lc_D^\alpha$ independent of $\wp$ and $\wp'$, a.s. there exists a cluster of $\Lc_D^\alpha$ which intersects both $\wp$ and $\wp'$.
This claim is a consequence of Proposition \ref{P:local_main} below. Our above line of arguments proves the second equality of \eqref{E:T_locuniq}, assuming this proposition.
\end{proof}

\begin{proposition}\label{P:local_main}
Let $R>1$, $\alpha \in \Ir_\tr^R$, $D \subset \R^3$ be a connected open set and $R''>2R$. For $\delta >0$, let $\Lc_\delta = \{\wp \in \Lc_{\delta,\delta R''}^\alpha: \wp \subset D\}$ and let $\wp$ and $\wp'$ be independent Brownian loops sampled according to $\alpha \loopmeasure_{D}$, independent of $\Lc_\delta$. Then, for almost all sample of $\wp$ and $\wp'$,
\[
\P( \wp \overset{\Lc_\delta}{\longleftrightarrow} \wp' \vert \wp, \wp') \to 1
\qquad \text{as} \quad \delta \to 0.
\]
\end{proposition}

The rest of this section is devoted to the proof of Proposition \ref{P:local_main} and we fix $R>1$, $\alpha \in \Ir_\tr^R$ and $R''>R'>2R$ (a slight variation of our argument below allows to take $R''>R' > R > 1$ instead).

\begin{definition}[Seed]\label{D:seed}
For $r>0$ and $x \in \Z^3$, let
\begin{equation}
Q(x,r) = x + [-r,r]^3.
\end{equation}
We will simply write $Q(r)$ instead of $Q(0,r)$.

For $x \in \R^3$ and $r>R''/2$, let $\mathrm{Seed}_{x,r}$ be the event that $\{ \wp \in \Lc_{R',R''}^\alpha, \wp \cap Q(x,r-R''/2) \neq \varnothing\}$ contains a unique cluster $\Cc(\mathrm{Seed}_{x,r})$ which satisfies the property that
for all $y\in \R^3$ at distance at most $R$ to $Q(x,r)$, $z \in B(y,1/2)$ and $u \in [1,2]$,
\begin{equation}
\label{E:def_seed}
\P^{y,z;u}(\Cc(\mathrm{Seed}_{x,r}) \cap \wp \neq \varnothing \vert \Lc_{R',R''}^\alpha) \ge 1/2.
\end{equation}
If $\Seed_{x,r}$ does not hold, we will use the convention that $\Cc(\Seed_{x,r}) = \varnothing$.
\end{definition}

Using a Grimmett--Marstrand-type argument \cite{zbMATH04169828}, we will be able to prove:

\begin{proposition}\label{P:GM}
    For all $\eps >0$, there exist $r'>r>0$ large enough and a coupling between $\Lc_{1,R''}^\alpha$, a site Bernoulli percolation $(\eta_x)_{x \in \Z^3}$ on $\Z^3$ where for all $x \in \Z^3$,
    \[ 
    \P(\eta_x=1) = 1-\P(\eta_x = 0) = 1-\eps,
    \]
    and an embedding $F:\Z^3 \to \R^3$ such that the following holds. Almost surely, for all $x,y \in \Z^3$, $\|F(x)-F(y) - r'(x-y)\| \le r'$ and for all $x\sim y$, $F(y) \in \partial Q(F(x),r')$. Moreover, if $\eta_x=1$, then $\Seed_{F(x),r}$ holds and if $x\sim y$ are such that $\eta_x=\eta_y=1$, then there exists a cluster of $\{ \wp \in \Lc_{1,R'}^\alpha : \wp \cap Q(F(x),r'-r)\neq \varnothing\}$ which intersects both $\Cc(\Seed_{F(x),r})$ and $\Cc(\Seed_{F(y),r})$.
\end{proposition}

The proof of Proposition \ref{P:local_main} then follows:

\begin{proof}[Proof of Proposition \ref{P:local_main}, assuming Proposition \ref{P:GM}]
    Let $\hat D \Subset \widetilde D \Subset D$ be open subsets compactly included in each other. By considering an increasing limit $\hat D \uparrow D$, it is enough to show the result for paths $\wp$ and $\wp'$ sampled from $\loopmeasure_{\hat D}$.
    Let $\eps >0$ and consider $r'>r>0$ as in Proposition \ref{P:GM} and the resulting coupling between $\Lc_{1,R''}^\alpha$, a Bernoulli site percolation with parameter $1-\eps$ and an embedding $F$. The collections of loops $\Lc_\delta$ in the statement of Proposition \ref{P:local_main} are coupled together by
    \[
    \Lc_\delta = \{ \delta \wp: \wp \in \Lc_{1,R''}^\alpha, \wp \subset \delta^{-1} D \}.
    \]
    For almost all sample $\wp$ from $\loopmeasure_{\hat D}$,
    \[
    \# \{ x \in \Z^3: Q(\delta F(x),\delta r) \cap \wp \neq \varnothing \} \to \infty, \quad \text{as~} \delta\to 0.
    \]
    Moreover, each time $\wp$ hits such a cube $Q(\delta F(x),\delta r)$, it has a positive probability to hit the seed cluster $\delta \Cc(\Seed_{F(x),r})$. So, for almost all sample $\wp$ from $\loopmeasure_{\hat D}$,
    \[
    \P(\exists x \in \Z^3: \wp \cap \delta\Cc(\Seed_{F(x),r}) \neq \varnothing \vert \wp) \to 1, \quad \text{as~} \delta\to 0.
    \]
    Now, consider two independent samples $\wp$ and $\wp'$ from $\loopmeasure_{\hat D}$. With probability tending to 1 as $\delta \to 0$, there exist $x, x' \in \Z^3$, $\wp \cap \delta\Cc(\Seed_{F(x),r})\neq \varnothing$ and $\wp' \cap \delta\Cc(\Seed_{F(x'),r})\neq \varnothing$.
    Now, by Lemma~\ref{L:Bernoulli_site} below, with probability at least $p(\eps)$ which goes to 1 as $\eps \to 0$, $x$ and $x'$ belong to the same Bernoulli site percolation cluster and moreover, there exists a path $\gamma$ joining $x$ to $x'$ such that $\delta F\circ \gamma$ never exits $\widetilde D$.
    This constructs a cluster $\Cc_\delta$ of $\Lc_{1,R''}^\alpha$ of loops remaining in $\delta^{-1} D$ with $\Cc_\delta \cap \wp \neq \varnothing$ and $\Cc_\delta \cap \wp' \neq \varnothing$ and shows that
    \[
    \liminf_{\delta \to 0} \P(\delta\Cc_\delta \cap \wp \neq \varnothing, \delta\Cc_\delta \cap \wp' \neq \varnothing \vert \wp, \wp') \ge p(\eps).
    \]
    Since $p(\eps) \to 1$ as $\eps \to 0$, this concludes the proof.
\end{proof}

\begin{lemma}\label{L:Bernoulli_site}
    Let $\hat D \Subset \widetilde D$.
    For $\eps>0$, let $\P_{1-\eps}$ be the law of a Bernoulli site percolation on $\Z^3$ where the probability for a given site to be open is $1-\eps$. Then
    \[
    \lim_{\eps \to 0} \inf_{\delta \in (0,1)} \inf_{x,y \in \Z^3 \cap \delta^{-1} \hat D} \P_{1-\eps}( x \longleftrightarrow y \text{ within } \delta^{-1} \widetilde D ) = 1.
    \]
\end{lemma}

\begin{proof}[Proof of Lemma \ref{L:Bernoulli_site}]
    This is a standard result. A spatially-controlled connection between any two pair of points can be made by gluing successive paths. Local uniqueness of macroscopic clusters (see e.g. \cite[Section 7.4]{zbMATH01301486}) guarantees the success of the gluing with high probability.
\end{proof}

\tdplotsetmaincoords{80}{20}

\begin{figure}
    \centering
\begin{tikzpicture}[tdplot_main_coords, line join=round, scale=.5]

\def\L{10}   
\def\l{2}    
\def\H{5}    
\def\h{1}    

\def\X{5}  
\def\Y{2}
\def\Z{2}

\coordinate (A1) at (-\H, -\H, -\H);
\coordinate (B1) at (\H,  -\H, -\H);
\coordinate (C1) at (\H,   \H, -\H);
\coordinate (D1) at (-\H,  \H, -\H);

\coordinate (A2) at (-\H, -\H,  \H);
\coordinate (B2) at (\H,  -\H,  \H);
\coordinate (C2) at (\H,   \H,  \H);
\coordinate (D2) at (-\H,  \H,  \H);

\draw[thick] (A1)--(B1)--(C1);
\draw[thick] (A2)--(B2)--(C2)--(D2)--cycle;
\draw[thick] (A1)--(A2);
\draw[thick] (B1)--(B2);
\draw[thick] (C1)--(C2);
\draw[thick, dotted] (D1)--(D2);
\draw[thick, dotted] (D1)--(C1);
\draw[thick, dotted] (D1)--(A1);

\coordinate (M1) at (\H, 0, -\H);
\coordinate (M2) at (\H, 0,  \H);
\coordinate (M3) at (\H, -\H, 0);
\coordinate (M4) at (\H,  \H, 0);

\draw[dashed] (M1)--(M2);
\draw[dashed] (M3)--(M4);

\drawcube[fill=gray!40,fill opacity=0.6,draw=black]{0}{0}{0};
\drawcube[fill opacity=0,draw=black]{0}{0}{0};

\drawcube[fill=gray!40,fill opacity=0.6,draw=black]{\X}{\Y}{\Z};
\drawcube[fill opacity=0,draw=black]{\X}{\Y}{\Z};

\begin{scope}[shift={(-6,-6,-6)}]   
    \draw[->,thick,black] (0,0,0) -- (1.5,0,0) node[anchor=west] {$x$};
    \draw[->,thick,black] (0,0,0) -- (0,2,0) node[anchor=south] {$y$};
    \draw[->,thick,black] (0,0,0) -- (0,0,1.5) node[anchor=south] {$z$};
\end{scope}

\draw (0,0,-1.2) node {$Q(r)$};
\draw (-8,0,0) node {$Q(r')$};
\draw (\X+2.5,\Y,\Z) node {$Q(x,r)$};
\end{tikzpicture}
\caption{Illustration of some notations used in Section \ref{SS:seed}. The large cube represents $Q(r')$. The $x=r'$ face is divided into four squares isometric to $\{r'\} \times [0,r']^2$. In Proposition~\ref{P:seed}, we start with a seed at the origin, meaning that the cube $Q(r)$ contains a unique very well connected cluster of $\{\wp \in \Lc_{R',R''}^\alpha: \wp \cap Q(r-R''/2)\neq\varnothing\}$; see Definition \ref{D:seed}. Proposition~\ref{P:seed} then asserts that with high probability, we can find a seed within some $Q(x,r)$ with $x \in \{r'\} \times [0,r']^2$ which is connected to the seed at the origin in $\Lc_{1,R}^\alpha\cap Q(r'-r)$.}
    \label{fig:GM}
\end{figure}

\subsection{Proof of Proposition \ref{P:GM}, assuming Proposition \ref{P:seed}}\label{SS:seed}

\begin{notation}[Cubes, squares, etc.]
Let $r'>r>0$ be such that $r'/r \in 2\N$.
We will need to consider the six faces of $\partial Q(r')$ individually and divide each such face into four quadrants. Consider for instance $\{r'\} \times [0,r']^2$. We will pave this quadrant into $r'^2/(4r^2)$ squares of sidelength $2r$. To this end, let
\[
\mathfrak{X}_{r,r'} = \{ (r',(2n+1)r,(2m+1)r): (n,m) \in \{0,1, \dots, \frac{r'}{2r}-1\}^2 \}.
\]
The quadrant $\{r'\} \times [0,r']^2$ is then paved as follows:
\[
\{r'\} \times [0,r']^2 = \bigcup_{x \in \mathfrak{X}_{r,r'}} (x + \{0\} \times [-r,r]^2)
\]
where the interior of the squares on the right hand side are pairwise disjoint.
\end{notation}


\begin{proposition}\label{P:seed}
Let $\eps >0$. There exist $r'>2r>0$ large enough with $r'/r \in 2\N$ such that
\begin{equation}
\P(\exists x \in \mathfrak{X}_{r,r'}, ~\Cc(\Seed_{0,r}) \overset{\Lc_{1,R}^\alpha \cap Q(r'-r)}{\longleftrightarrow} \Cc(\Seed_{x,r}) \vert \Seed_{0,r} ) \ge 1 - \eps,
\end{equation}
where $\Lc_{1,R}^\alpha \cap Q(r'-r)$ stands for the subset of loops $\wp \in \Lc_{1,R}^\alpha$ which intersect the cube $Q(r'-r)$.
\end{proposition}

We now explain how to iterate. Let $\Cc$ be the union of all clusters of $\{\wp \in \Lc_{1,R}^\alpha : \wp \cap Q(r'-r) \neq \varnothing\}$ which intersect $\Cc(\Seed_{0,r})$.
With probability at least $1-\eps$, there exists $x \in \mathfrak{X}_{r,r'}$ such that $\Seed_{x,r}$ holds and such that $\Cc \cap \Cc(\Seed_{x,r}) \neq \varnothing$.
We want to lower bound
\begin{align}\label{E:P_negative1}
    \P(\exists y \in x+\mathfrak{X}_{r,r'}: \Cc(\Seed_{x,r}) \overset{\Lc_{1,R'}^\alpha\cap Q(x,r'-r)}{\longleftrightarrow} \Cc(\Seed_{y,r}) \vert \Cc, \Cc(\Seed_{x,r})).
\end{align}
The exploration which reveals $\Cc$ contains negative information: we know that there is no loop of $\Lc_{1,R}^\alpha$ which intersects both $Q(r'-r)$ and $\Cc$. This negative information will be dealt with by sprinkling some loops. Instead of increasing the intensity $\alpha$, we consider diameters in $[R,R']$.
Let $\Cc'$ be the union of all the clusters in $\{ \wp \in \Lc_{1,R}^\alpha : \wp \cap Q(x,r'-r) \neq \varnothing, \wp \cap (\Cc \cup \Cc(\Seed_{x,r})) = \varnothing \}$ which intersect some $\Cc(\Seed_{y,r})$ for some $y \in x + \mathfrak{X}_{r,r'}$. The probability in \eqref{E:P_negative1} equals
\begin{align*}
    & 1- \P(\nexists \wp \in \Lc_{1,R'}^\alpha\cap Q(x,r'-r) : \Cc\cup\Cc(\Seed_{x,r}) \overset{\wp}{\longleftrightarrow} \Cc' \vert \Cc, \Cc(\Seed_{x,r})) \\
    & \ge 1 - \P(\nexists \wp \in \Lc_{R,R'}^\alpha\cap Q(x,r'-r) : \Cc\cup\Cc(\Seed_{x,r}) \overset{\wp}{\longleftrightarrow} \Cc' \vert \Cc, \Cc(\Seed_{x,r})).
\end{align*}
By Lemma \ref{L:key}, there exists $c = c(R,R') \in (0,1)$ such that
\begin{align*}
    & \P(\nexists \wp \in \Lc_{R,R'}^\alpha\cap Q(x,r'-r) : \Cc\cup\Cc(\Seed_{x,r}) \overset{\wp}{\longleftrightarrow} \Cc' \vert \Cc, \Cc(\Seed_{x,r}), \Cc') \\
    & \le \P(\nexists \wp \in \tilde\Lc_{1,R}^\alpha\cap Q(x,r'-r) : \Cc\cup\Cc(\Seed_{x,r}) \overset{\wp}{\longleftrightarrow} \Cc' \vert \Cc, \Cc(\Seed_{x,r}), \Cc')^c,
\end{align*}
where $\tilde\Lc_{1,R}^\alpha$ is an independent copy of $\Lc_{1,R}^\alpha$. By Hölder's inequality, we deduce that the probability \eqref{E:P_negative1} is at least
\begin{align*}
    & 1 - \E[ \P(\nexists \wp \in \tilde\Lc_{1,R}^\alpha\cap Q(x,r'-r) : \Cc\cup\Cc(\Seed_{x,r}) \overset{\wp}{\longleftrightarrow} \Cc' \vert \Cc, \Cc(\Seed_{x,r}), \Cc')^c \vert \Cc, \Cc(\Seed_{x,r})] \\
    & \ge 1 - \P(\nexists \wp \in \tilde\Lc_{1,R}^\alpha\cap Q(x,r'-r) : \Cc\cup\Cc(\Seed_{x,r}) \overset{\wp}{\longleftrightarrow} \Cc' \vert \Cc, \Cc(\Seed_{x,r}))^c.
\end{align*}
All the negative information has been replaced by an independent copy of $\Lc_{1,R}^\alpha$. It only remains positive information so, by FKG inequality, the probability \eqref{E:P_negative1} is at least
\begin{align*}
    & 1 - \P(\nexists y \in x +\mathfrak{X}_{r,r'}: \Cc(\Seed_{x,r}) \overset{\Lc_{1,R}^\alpha\cap Q(x,r'-r)}{\longleftrightarrow} \Cc(\Seed_{y,r}) \vert \Cc(\Seed_{x,r}))^c \ge 1-\eps^c,
\end{align*}
using in the last inequality Proposition \ref{P:seed}.

The iteration of this procedure is then standard. In particular, the subdivision of $\{r'\} \times [-r',r']^2$ into four quadrants permits to control the lateral deviation of the procedure and thus the position of the seeds: this is known as ``steering''. See \cite[Section 7.2]{zbMATH01301486} for a detailed exposition.
This yields Proposition \ref{P:GM}.

\subsection{Proof of Proposition \ref{P:seed}}\label{SS:pf_seed}

Let $\Cc_{r,r'}$ be the union of all clusters of $\{ \wp \in \Lc^\alpha_{1,R}: \wp \subset Q(r'-r) \}$ intersecting $Q(r)$.
Let $K(r,r')$ (resp. $K^\Seed(r,r')$) be the maximal number $k\ge 0$ such that there exists $\mathfrak{X} \subset \mathfrak{X}_{r,r'}$ with $\# \mathfrak{X} = k$, such that for all $x \in \mathfrak{X}$, there exists $\wp_x \in \Lc^\alpha_{1,R}$ intersecting $Q(x,r)$ and $\Cc_{r,r'}$ (resp. intersecting $\Cc(\Seed_{x,r})$ and $\Cc_{r,r'}$). 

\begin{lemma}\label{L:localK}
For all $\eps>0$, there exists $r>0$ large enough such that for all $k \ge 1$, there exists $r' > r$ large enough such that
\[
\P(K(r,r') \ge k) \ge 1-\eps.
\]
\end{lemma}

\begin{proof}
For $r',r>0$, let $\tilde K(r,r')$ be the random variable defined in the same way as $K(r,r')$ with the difference that we consider a paving of the whole surface $\partial Q(r')$ by squares of sidelength $2r$. Since $\partial Q(r')$ is the union of 24 squares obtained by rotating $\{r'\} \times [0,r']^2$, $\{\tilde{K}(r,r') \ge 24k\}$ is included in the union of 24 rotated versions of the event $\{K(r,r')\ge k\}$.
By the square-root trick (Lemma \ref{L:square-root_trick}) and invariance under rotation, we deduce that
\[
\P(K(r,r')\ge k) \ge 1 - (1-\P(\tilde K(r,r') \ge 24k))^{1/24}.
\]
To prove the lemma, it is thus enough to show that for all $\eps >0$ and $k \ge 1$, there exists $r' > r> 0$ large enough such that
$
\P(\tilde K(r,r')\ge k) \ge 1-\eps.
$

Let $\eps >0$ be fixed. By \eqref{E:L_Jr2} (and because one can fit a big cube in a big ball and vice versa), we can first pick $r>0$ large enough such that
\begin{equation}
\label{E:pf_local1}
\P(\forall r'>r, ~\tilde K(r,r') \ge 1) =
\P(\forall r'>r,  ~Q(r) \overset{\Lc_{1,R}^\alpha}{\longleftrightarrow} \partial Q(r')) \ge 1 - \eps/2.
\end{equation}
Fix now $k \ge 1$, let $r'>r$ and recall that $\tilde K(r,r')$ only depends on the loops which hit $Q(r'-r)$.
Consider the event that none of the loops in $\{ \wp \in \Lc_{1,R}^\alpha: \wp \cap Q(r'-r) = \varnothing\}$ intersects a cluster of $Q(r)$ in $\{\wp \in \Lc_{1,R}^\alpha: \wp \cap Q(r'-r) \neq \varnothing\}$. Conditionally on $\tilde{K}(r,r')$, this occurs with probability at least $p^{\tilde K(r,r')}$ for some $p>0$ which only depends on $\alpha$ and $R$.
Since this event guarantees that $Q(r)$ is not connected to $\partial Q(r'-r+R)$, this shows that
\begin{align*}
\P(1 \le \tilde K(r,r') \le k) \le p^{-k} \P( Q(r) \overset{\Lc_{1,R}^\alpha}{\longleftrightarrow} \partial Q(r'-r), Q(r) \overset{\Lc_{1,R}^\alpha}{\centernot\longleftrightarrow} \partial Q(r'-r+R)).
\end{align*}
The right hand side tends to zero as $r' \to \infty$ which shows that $\P(1 \le \tilde K(r,r') \le k) \to 0$ as $r' \to 0$. We can now pick $r'$ large enough so that $\P(1 \le\tilde K(r,r') \le k) \le \eps/2$. Together with \eqref{E:pf_local1}, this shows that $\P(\tilde K(r,r') > k ) \ge 1-\eps$ which concludes the proof.
\end{proof}

\begin{lemma}[Seed]\label{L:seed_proba}
For all $r> R'' \vee 4R$ and $x \in \R^3$, $\P(\Seed_{x,r}) > 0$.
\end{lemma}

\begin{proof}
We say that a cluster $\Cc$ of $\{ \wp \in \Lc_{R',R''}^\alpha, \wp \cap Q(x,r-R''/2) \neq \varnothing\}$ is ``admissible'' if for all $y\in \R^3$ at distance at most $R$ to $Q(x,r)$, $z \in B(y,1/2)$ and $u \in [1,2]$,
\[
\P^{y,z;u}(\Cc \cap \wp \neq \varnothing \vert \Lc_{R',R''}^\alpha) \ge 1/2.
\]
The event $\Seed_{x,r}$ corresponds to the existence of a unique such admissible cluster. For $k \ge 1$, denote by $\Seed_{x,r}^k$ the event that $\{ \wp \in \Lc_{R',R''}^\alpha, \wp \cap Q(x,r-R''/2) \neq \varnothing\}$ contains exactly $k$ different admissible clusters $\Cc_1, \dots, \Cc_k$.
Clearly, $\P(\bigcup_{k\ge1} \Seed_{x,r}^k) >0$. So there exists $k \ge 1$ such that $\P(\Seed_{x,r}^k) >0$. By definition and on the event $\Seed_{x,r}^k$, if one adds on top of $\{ \wp \in \Lc_{R',R''}^\alpha, \wp \cap Q(x,r-R''/2) \neq \varnothing\}$ an independent loop sampled according to $\alpha \mathbf{1}_{\diam(\wp)\in[R',R''],\wp \cap Q(x,r-R''/2) \neq \varnothing}\loopmeasure(\d \wp)$, then there is a positive probability that this loop intersects all of the admissible clusters $\Cc_1, \dots, \Cc_k$. This reduces the number of admissible clusters to 1. By Palm's formula, we deduce that
$\P(\Seed_{x,r}) \ge c(k) \P(\Seed_{x,r}^k) >0.$ This concludes the proof.
\end{proof}

We now have all the ingredients to prove Proposition \ref{P:seed}.

\begin{proof}[Proof of Proposition \ref{P:seed}]
Let $\eps>0$. Let $r>0$ be large enough so that Lemma \ref{L:localK} holds: for all $k \ge 1$, there exists $r' > r$ large enough so that $\P(K(r,r') \ge k) \ge 1-\eps/2$. We fix these values of $k$ and $r'$.
By Lemma \ref{L:seed_proba}, $p_1 := \P(\Seed_{0,r})$ is positive. The value of $p_1$ does not depend on $k$ and $r'$.
By definition of the seed events, there exists $p_2 >0$ independent of $r'$ and $k$ such that for all $x \in \mathfrak{X}_{r,r'}$, on the event that $\Seed_{x,r}$ holds,
\begin{align*}
& \P(\exists \wp_x \in \Lc_{1,R}^\alpha, \wp_x \nsubseteq Q(r'-r): \Cc_{r,r'} \overset{\wp_x}{\leftrightarrow} \Cc(\Seed_{x,r}) \vert \Cc(\Seed_{x,r}), \Cc_{r,r'} ) \\
& \ge p_2 \P(\exists \wp_x \in \Lc_{1,R}^\alpha, \wp_x \nsubseteq Q(r'-r): \Cc_{r,r'} \overset{\wp_x}{\leftrightarrow} Q(x,r) \vert \Cc_{r,r'} ).
\end{align*}
When $r$ is large enough compared to $R''$ and for $x \ne y \in \mathfrak{X}_{r,r'}$, the collections $\{ \wp \in \Lc_{1,R}^\alpha: \wp \cap Q(x,r+R''/2) \neq \varnothing \}$ and $\{ \wp \in \Lc_{1,R}^\alpha: \wp \cap Q(y,r+R''/2) \neq \varnothing \}$ are independent, except when $Q(x,r) \cap Q(y,r) \neq \varnothing$, i.e. for $8$ values of $y$.
Overall, this shows that $K^\Seed(r,r')$ stochastically dominates
\[
\sum_{j=1}^{\floor{K(r,r')/8}} B_j, \qquad \text{where} \quad B_1, B_2, \dots \overset{\text{i.i.d.}}{\sim} \text{Bernoulli}(p_1p_2).
\]
Hence,
\begin{align*}
    \P(K^\Seed(r,r') \ge 1) \ge 1 - \eps/2 - (1-p_1p_2)^{\floor{k/8}}.
\end{align*}
If $k$ is large enough, $(1-p_1p_2)^{\floor{k/8}}$ is smaller than $\eps/2$, concluding that $\P(K^\Seed(r,r') \ge 1) \ge 1-\eps$. This shows that
\[ 
\P(\exists x \in \mathfrak{X}_{r,r'}, ~Q(x,r) \overset{\Lc_{1,R'}^\alpha}{\longleftrightarrow} \Cc(\Seed_{x,r}) ) \ge 1 - \eps.
\]
A variant of this argument shows the desired bound where $Q(x,r)$ is replaced by $\Cc(\Seed_{0,r})$, conditionally on $\Cc(\Seed_{0,r})$.
\end{proof}

\subsection{Proof of Theorem \ref{T:JR}, \eqref{E:T_locuniq}, first equality}\label{SS:pf_1st_equality}

The inclusion $\Ir_\tr \supset \bigcup_{R>1} \Ir_\tr^R$ is clear and we only need to prove the reverse inclusion. Let $\alpha \in \Ir_\tr$.
Let $r'>2r>0$ with $r'/r \in 2\N$ and $A>0$.
We adapt slightly definitions from Section \ref{SS:pf_seed}. Let $\Cc_{r,r'}$ be the union of all clusters of $\{ \wp \in \Lc^\alpha_{\ge 1}: \wp \subset Q(r'-r) \}$ intersecting $Q(r)$.
Let $K(A,r,r')$ be the maximal number $k\ge 0$ such that there exists $\mathfrak{X} \subset \mathfrak{X}_{r,r'}$ with $\# \mathfrak{X} = k$, such that for all $x \in \mathfrak{X}$, there exists $\wp_x \in \Lc^\alpha_{\ge 1}$ intersecting $\Cc_{r,r'}$, whose root belongs to $Q(x,r)$ and with $\diam(\wp) \le Ar'$.

Let $\eps >0$ and $k \ge 1$. Because $\alpha \in \Ir_\tr$, if $r$ is large enough,
\[
\P( \forall r'>r, \partial Q(r) \overset{\Lc_{\ge 1}^\alpha}{\longleftrightarrow} \partial Q(r') ) \ge 1-\eps/4.
\]
Moreover, by scale invariance
\begin{align*}
    \loopmeasure_{\R^3}(\{ \wp : \wp \cap Q(r') \neq \varnothing, \diam(\wp) \ge A r' \} ) = \loopmeasure_{\R^3}(\{ \wp : \wp \cap Q(1) \neq \varnothing, \diam(\wp) \ge A \} ),
\end{align*}
which goes to 0 as $A \to \infty$ (see e.g. Lemma \ref{L:estimate_crossing}).
By taking $A$ large we can thus guarantee that for all $r' > r$,
\[
\P(\P( \partial Q(r) \overset{\Lc_{1,Ar'}^\alpha}{\longleftrightarrow} \partial Q(r') ) \ge 1-\eps/2.
\]
In particular, $\P(K(A,r,r') \ge 1) \ge 1-\eps/2$ for all $r' > r$. 
As in Lemma \ref{L:localK}, we can then infer that if $r'$ is large enough, $\P(K(A,r,r') \ge k) \ge 1-\eps$.

We can then repeat the dynamical construction as before which shows that $\Lc_{1,Ar'}^\alpha$ percolates with positive probability and thus $\alpha \in \Ir_\tr^{Ar'}$. This proves that $\Ir_\tr \subset \bigcup_{R>1}\Ir_\tr^R.$

Now that we know that $\Ir_\tr = \bigcup_{R>1}\Ir_\tr^R$ and that $\alpha_\tr^{R}>\alpha_\tr^{R'}$ for all $R'>R>1$, the fact that $\Ir_\tr$ is open ($\alpha_\tr \notin \Ir_\tr$) follows.
\qed

\paragraph*{Acknowledgements}
We are grateful to Vincent Tassion for many inspiring discussions.
Thanks also to Nathanaël Berestycki, Tom Hutchcroft and Xin Sun for questions/discussions related to the invariance of Brownian motion and the Brownian loop measure under inversions.

AJ and TL would like to thank the Isaac Newton Institute for Mathematical Sciences, Cambridge, for support and hospitality during the programme \textit{Stochastic systems for anomalous diffusion}, where work on this paper was undertaken. This work was supported by EPSRC grant EP/Z000580/1.

\addcontentsline{toc}{section}{References}
{\small
\bibliographystyle{alpha}
\bibliography{bibliography}

@article{Lawler04,
 author = {Lawler, Gregory F. and Werner, Wendelin},
 title = {The {Brownian} loop soup},
 fjournal = {Probability Theory and Related Fields},
 journal = {Probab. Theory Relat. Fields},
 issn = {0178-8051},
 volume = {128},
 number = {4},
 pages = {565--588},
 year = {2004},
 language = {English},
 zbMATH = {2077979},
 Zbl = {1049.60072}
}

@article {Janson84,
    AUTHOR = {Janson, Svante},
     TITLE = {Bounds on the distributions of extremal values of a scanning process},
   JOURNAL = {Stochastic Process. Appl.},
  FJOURNAL = {Stochastic Processes and their Applications},
    VOLUME = {18},
      YEAR = {1984},
    NUMBER = {2},
     PAGES = {313--328},
      ISSN = {0304-4149},
   MRCLASS = {60G55},
  MRNUMBER = {770197},
MRREVIEWER = {Noel Cressie},
       DOI = {10.1016/0304-4149(84)90303-X},
       URL = {https://doi-org.ezp.lib.cam.ac.uk/10.1016/0304-4149(84)90303-X},
}

@article {SheffieldWernerCLE,
    AUTHOR = {Sheffield, Scott and Werner, Wendelin},
     TITLE = {Conformal loop ensembles: the {M}arkovian characterization and
              the loop-soup construction},
   JOURNAL = {Ann. Math. (2)},
  FJOURNAL = {Annals of Mathematics. Second Series},
    VOLUME = {176},
      YEAR = {2012},
    NUMBER = {3},
     PAGES = {1827--1917},
      ISSN = {0003-486X},
     CODEN = {ANMAAH},
   MRCLASS = {60J67},
  MRNUMBER = {2979861},
       DOI = {10.4007/annals.2012.176.3.8},
       URL = {http://dx.doi.org/10.4007/annals.2012.176.3.8},
}

@Article{LawlerFerreras07RWLoopSoup,
author = {Lawler, Gregory F. and Trujillo-Ferreras, José A.},
title = {Random walk loop soup},
journal = {Trans. Am. Math. Soc.},
year = {2007},
volume = {359},
number = {2},
pages = {767-787}
}

@article{aidekon2023multiplicative,
 author = {A{\"{\i}}d{\'e}kon, {\'E}lie and Berestycki, Nathana{\"e}l and Jego, Antoine and Lupu, Titus},
 title = {Multiplicative chaos of the {Brownian} loop soup},
 fjournal = {Proceedings of the London Mathematical Society. Third Series},
 journal = {Proc. Lond. Math. Soc. (3)},
 issn = {0024-6115},
 volume = {126},
 number = {4},
 pages = {1254--1393},
 year = {2023},
 language = {English},
 doi = {10.1112/plms.12511},
 zbMATH = {7740451},
 Zbl = {1540.60060}
}

@unpublished{LawlerBessel,
	title = {Notes on the {B}essel process},
	author = {Lawler, Gregory F.},
	year = "2018",
	note = {Lecture notes. Available on the webpage of the author}
}

@article {Lupu18,
    AUTHOR = {Lupu, Titus},
     TITLE = {Poisson ensembles of loops of one-dimensional diffusions},
   JOURNAL = {M\'{e}m. Soc. Math. Fr. (N.S.)},
  FJOURNAL = {M\'{e}moires de la Soci\'{e}t\'{e} Math\'{e}matique de France. Nouvelle S\'{e}rie},
    NUMBER = {158},
      YEAR = {2018},
     PAGES = {158},
      ISSN = {0249-633X},
      ISBN = {978-2-85629-891-6},
   MRCLASS = {60-02 (60G15 60G17 60G55 60G60 60J55 60J60 60J80)},
  MRNUMBER = {3865570},
       DOI = {10.24033/msmf.466},
       URL = {https://doi.org/10.24033/msmf.466},
}

@book{morters2010brownian,
  title={Brownian motion},
  author={M{\"o}rters, Peter and Peres, Yuval},
  volume={30},
  year={2010},
  publisher={Cambridge University Press}
}

@article{chang2017supercritical,
  title={Supercritical loop percolation on $\mathbb{Z}^d$ for $d \geq 3$},
  author={Chang, Yinshan},
  JOURNAL = {Stochastic Process. Appl.},
  FJOURNAL = {Stochastic Processes and their Applications},
  volume={127},
  number={10},
  pages={3159--3186},
  year={2017},
  publisher={Elsevier}
}

@article{chang2016phase,
  title={Phase transition in loop percolation},
  author={Chang, Yinshan and Sapozhnikov, Art{\"e}m},
journal={Probab. Theory and Relat. Fields},
FJOURNAL = {Probability Theory and Related Fields},
  volume={164},
  number={3-4},
  pages={979--1025},
  year={2016},
  publisher={Springer}
}

@article{werner2021clusters,
  title={{On clusters of Brownian loops in $d$ dimensions}},
  author={Werner, Wendelin},
  journal={In and Out of Equilibrium 3: Celebrating Vladas Sidoravicius},
  pages={797--817},
  year={2021},
  publisher={Springer}
}

@article {MR1428500,
    AUTHOR = {Liggett, T. M. and Schonmann, R. H. and Stacey, A. M.},
     TITLE = {Domination by product measures},
   JOURNAL = {Ann. Probab.},
  FJOURNAL = {The Annals of Probability},
    VOLUME = {25},
      YEAR = {1997},
    NUMBER = {1},
     PAGES = {71--95},
      ISSN = {0091-1798,2168-894X},
   MRCLASS = {60G60 (60G10 60K35)},
  MRNUMBER = {1428500},
MRREVIEWER = {Vincent\ De Valk},
       DOI = {10.1214/aop/1024404279},
       URL = {https://doi.org/10.1214/aop/1024404279},
}

@article{Chayes91,
 author = {Chayes, J. T. and Chayes, L. and Grannan, E. and Swindle, G.},
 title = {Phase transitions in {Mandelbrot}'s percolation process in three dimensions},
 fjournal = {Probability Theory and Related Fields},
 journal = {Probab. Theory Relat. Fields},
 issn = {0178-8051},
 volume = {90},
 number = {3},
 pages = {291--300},
 year = {1991},
 language = {English},
 doi = {10.1007/BF01193747},
 zbMATH = {4215123},
 Zbl = {0734.60100}
}

@book {MR1780932,
    AUTHOR = {Rogers, L. C. G. and Williams, David},
     TITLE = {Diffusions, {M}arkov processes, and martingales. {V}ol. 2},
    SERIES = {Cambridge Mathematical Library},
      NOTE = {It\^o{} calculus,
              Reprint of the second (1994) edition},
 PUBLISHER = {Cambridge University Press, Cambridge},
      YEAR = {2000},
     PAGES = {xiv+480},
      ISBN = {0-521-77593-0},
   MRCLASS = {60J60 (60G07 60H05 60J25)},
  MRNUMBER = {1780932},
       DOI = {10.1017/CBO9781107590120},
       URL = {https://doi.org/10.1017/CBO9781107590120},
}

@book {MR3410783,
    AUTHOR = {Simon, Barry},
     TITLE = {Harmonic analysis},
    SERIES = {A Comprehensive Course in Analysis},
    VOLUME = {Part 3},
 PUBLISHER = {American Mathematical Society, Providence, RI},
      YEAR = {2015},
     PAGES = {xviii+759},
      ISBN = {978-1-4704-1102-2},
   MRCLASS = {42-01 (26-01 30-01 31-01 43-01 46-01)},
  MRNUMBER = {3410783},
MRREVIEWER = {Fritz\ Gesztesy},
       DOI = {10.1090/simon/003},
       URL = {https://doi.org/10.1090/simon/003},
}

@article {MR880979,
    AUTHOR = {Le Gall, Jean-Fran\c{c}ois},
     TITLE = {Le comportement du mouvement brownien entre les deux instants o\`u{} il passe par un point double},
   JOURNAL = {J. Funct. Anal.},
  FJOURNAL = {Journal of Functional Analysis},
    VOLUME = {71},
      YEAR = {1987},
    NUMBER = {2},
     PAGES = {246--262},
      ISSN = {0022-1236},
   MRCLASS = {60J65 (60J55)},
  MRNUMBER = {880979},
MRREVIEWER = {S.\ J.\ Taylor},
       DOI = {10.1016/0022-1236(87)90003-6},
       URL = {https://doi.org/10.1016/0022-1236(87)90003-6},
}

@article {MR701921,
    AUTHOR = {Rosen, Jay},
     TITLE = {A local time approach to the self-intersections of {B}rownian
              paths in space},
   JOURNAL = {Comm. Math. Phys.},
  FJOURNAL = {Communications in Mathematical Physics},
    VOLUME = {88},
      YEAR = {1983},
    NUMBER = {3},
     PAGES = {327--338},
      ISSN = {0010-3616,1432-0916},
   MRCLASS = {60J65 (60G17)},
  MRNUMBER = {701921},
MRREVIEWER = {Robert\ J.\ Adler},
       URL = {http://projecteuclid.org/euclid.cmp/1103922380},
}

@article {MR723731,
    AUTHOR = {Geman, Donald and Horowitz, Joseph and Rosen, Jay},
     TITLE = {A local time analysis of intersections of {B}rownian paths in
              the plane},
   JOURNAL = {Ann. Probab.},
  FJOURNAL = {The Annals of Probability},
    VOLUME = {12},
      YEAR = {1984},
    NUMBER = {1},
     PAGES = {86--107},
      ISSN = {0091-1798,2168-894X},
   MRCLASS = {60G15 (60G17 60G60 60J65)},
  MRNUMBER = {723731},
MRREVIEWER = {Robert\ J.\ Adler},
       URL =
              {http://links.jstor.org/sici?sici=0091-1798(198402)12:1<86:ALTAOI>2.0.CO;2-M&origin=MSN},
}

@article {MR1116036,
    AUTHOR = {Aizenman, Michael and Grimmett, Geoffrey},
     TITLE = {Strict monotonicity for critical points in percolation and
              ferromagnetic models},
   JOURNAL = {J. Statist. Phys.},
  FJOURNAL = {Journal of Statistical Physics},
    VOLUME = {63},
      YEAR = {1991},
    NUMBER = {5-6},
     PAGES = {817--835},
      ISSN = {0022-4715,1572-9613},
   MRCLASS = {82B43 (82B27)},
  MRNUMBER = {1116036},
MRREVIEWER = {J.\ Theodore\ Cox},
       DOI = {10.1007/BF01029985},
       URL = {https://doi.org/10.1007/BF01029985},
}

@article {MR990777,
    AUTHOR = {Burton, R. M. and Keane, M.},
     TITLE = {Density and uniqueness in percolation},
   JOURNAL = {Comm. Math. Phys.},
  FJOURNAL = {Communications in Mathematical Physics},
    VOLUME = {121},
      YEAR = {1989},
    NUMBER = {3},
     PAGES = {501--505},
      ISSN = {0010-3616,1432-0916},
   MRCLASS = {60K35 (82A43)},
  MRNUMBER = {990777},
MRREVIEWER = {G.\ R.\ Grimmett},
       URL = {http://projecteuclid.org/euclid.cmp/1104178143},
}

@article{zbMATH03886845,
 author = {Yor, M.},
 title = {A propos de l'inverse du mouvement brownien},
 fjournal = {Annales de l'Institut Henri Poincar{\'e}. Probabilit{\'e}s et Statistiques},
 journal = {Ann. Inst. Henri Poincar{\'e}, Probab. Stat.},
 issn = {0246-0203},
 volume = {21},
 pages = {27--38},
 year = {1985},
 language = {French},
 url = {https://eudml.org/doc/77246},
 zbMATH = {3886845},
 Zbl = {0556.60032}
}

@book{zbMATH01301486,
 author = {Grimmett, Geoffrey},
 title = {Percolation.},
 edition = {2nd},
 fseries = {Grundlehren der Mathematischen Wissenschaften},
 series = {Grundlehren Math. Wiss.},
 issn = {0072-7830},
 volume = {321},
 isbn = {978-3-540-64902-1; 978-3-642-08442-3; 978-3-662-03981-6},
 year = {1999},
 publisher = {Berlin: Springer},
 language = {English},
 doi = {10.1007/978-3-662-03981-6},
 zbMATH = {1301486},
 Zbl = {0926.60004}
}

@article{zbMATH04169828,
 author = {Grimmett, G. R. and Marstrand, J. M.},
 title = {The supercritical phase of percolation is well behaved},
 fjournal = {Proceedings of the Royal Society of London. Series A. Mathematical and Physical Sciences},
 journal = {Proc. R. Soc. Lond., Ser. A},
 issn = {0080-4630},
 volume = {430},
 number = {1879},
 pages = {439--457},
 year = {1990},
 language = {English},
 doi = {10.1098/rspa.1990.0100},
 zbMATH = {4169828},
 Zbl = {0711.60100}
}

@misc{arXiv:2510.20526,
 author = {Cai, Zhenhao and Ding, Jian},
 title = {On the gap between cluster dimensions of loop soups on $\mathbb{R}^3$ and the metric graph of $\mathbb{Z}^3$},
 year = {2025},
 howpublished = {Preprint, {arXiv}:2510.20526},
 url = {https://arxiv.org/abs/2510.20526},
 arXiv = {arXiv:2510.20526}
}

@article{zbMATH07662556,
 author = {Drewitz, Alexander and Pr{\'e}vost, Alexis and Rodriguez, Pierre-Fran{\c{c}}ois},
 title = {Critical exponents for a percolation model on transient graphs},
 fjournal = {Inventiones Mathematicae},
 journal = {Invent. Math.},
 issn = {0020-9910},
 volume = {232},
 number = {1},
 pages = {229--299},
 year = {2023},
 language = {English},
 doi = {10.1007/s00222-022-01168-z},
 zbMATH = {7662556},
 Zbl = {1522.60076}
}

@misc{arXiv:2405.17417,
 author = {Drewitz, Alexander and Pr{\'e}vost, Alexis and Rodriguez, Pierre-Fran{\c{c}}ois},
 title = {Critical one-arm probability for the metric {Gaussian} free field in low dimensions},
 year = {2024},
 howpublished = {Preprint, {arXiv}:2405.17417},
 url = {https://arxiv.org/abs/2405.17417},
 arXiv = {arXiv:2405.17417}
}

@misc{arXiv:2406.02397,
 author = {Cai, Zhenhao and Ding, Jian},
 title = {One-arm {Probabilities} for {Metric} {Graph} {Gaussian} {Free} {Fields} below and at the {Critical} {Dimension}},
 year = {2024},
 howpublished = {Preprint, {arXiv}:2406.02397},
 url = {https://arxiv.org/abs/2406.02397},
 arXiv = {arXiv:2406.02397}
}

@article{zbMATH04123007,
 author = {Burdzy, Krzysztof and Lawler, Gregory F.},
 title = {Non-intersection exponents for {Brownian} paths. {Part} {I}: {Existence} and an invariance principle},
 fjournal = {Probability Theory and Related Fields},
 journal = {Probab. Theory Relat. Fields},
 issn = {0178-8051},
 volume = {84},
 number = {3},
 pages = {393--410},
 year = {1990},
 language = {English},
 doi = {10.1007/BF01197892},
 zbMATH = {4123007},
 Zbl = {0685.60080}
}

@article{zbMATH04186799,
 author = {Burdzy, Krzysztof and Lawler, Gregory F.},
 title = {Nonintersection exponents for {Brownian} paths. {II}: {Estimates} and applications to a random fractal},
 fjournal = {The Annals of Probability},
 journal = {Ann. Probab.},
 issn = {0091-1798},
 volume = {18},
 number = {3},
 pages = {981--1009},
 year = {1990},
 language = {English},
 doi = {10.1214/aop/1176990733},
 zbMATH = {4186799},
 Zbl = {0719.60085}
}

@article{zbMATH00918284,
 author = {Lawler, Gregory F.},
 title = {Hausdorff dimension of cut points for {Brownian} motion},
 fjournal = {Electronic Journal of Probability},
 journal = {Electron. J. Probab.},
 issn = {1083-6489},
 volume = {1},
 number = {2},
 pages = {1--20},
 year = {1996},
 language = {English},
 url = {https://eudml.org/doc/119501},
 zbMATH = {918284},
 Zbl = {0891.60078}
}

@misc{duminil,
 author = {Duminil-Copin, Hugo},
 title = {Introduction to {B}ernoulli percolation},
 year = {2022},
 howpublished = {Available on the webpage of the author}
}

@article{zbMATH06603570,
 author = {Lupu, Titus},
 title = {From loop clusters and random interlacements to the free field},
 fjournal = {The Annals of Probability},
 journal = {Ann. Probab.},
 issn = {0091-1798},
 volume = {44},
 number = {3},
 pages = {2117--2146},
 year = {2016},
 language = {English},
 doi = {10.1214/15-AOP1019},
 zbMATH = {6603570},
 Zbl = {1348.60141}
}

@book{zbMATH06796875,
 author = {Last, G{\"u}nter and Penrose, Mathew},
 title = {Lectures on the {Poisson} process},
 fseries = {Institute of Mathematical Statistics Textbooks},
 series = {IMS Textb.},
 volume = {7},
 isbn = {978-1-107-45843-7; 978-1-107-08801-6; 978-1-316-10447-7},
 year = {2018},
 publisher = {Cambridge: Cambridge University Press},
 language = {English},
 doi = {10.1017/9781316104477},
 url = {semanticscholar.org/paper/f397e902cdd93643ca46c82d0ba7214af2e1975d},
 zbMATH = {6796875},
 Zbl = {1392.60004}
}

@misc{JLQb,
 author = {Jego, Antoine and Lupu, Titus and Qian, Wei},
 title = {Conformally invariant fields out of {Brownian} loop soups},
 year = {2023},
 howpublished = {Preprint, {arXiv}:2307.10740},
 url = {https://arxiv.org/abs/2307.10740},
 arXiv = {arXiv:2307.10740}
}

@misc{arXiv:2304.03150,
 author = {Aru, Juhan and Lupu, Titus and Sep{\'u}lveda, Avelio},
 title = {Excursion decomposition of the 2D continuum {GFF}},
 year = {2023},
 howpublished = {Preprint, {arXiv}:2304.03150},
 url = {https://arxiv.org/abs/2304.03150},
 arXiv = {arXiv:2304.03150}
}

@incollection{zbMATH05204617,
 author = {Schramm, Oded},
 title = {Conformally invariant scaling limits: an overview and a collection of problems},
 booktitle = {Proceedings of the international congress of mathematicians (ICM), Madrid, Spain, August 22--30, 2006. Volume I: Plenary lectures and ceremonies},
 isbn = {978-3-03719-022-7},
 pages = {513--543},
 year = {2007},
 publisher = {Z{\"u}rich: European Mathematical Society (EMS)},
 language = {English},
 zbMATH = {5204617},
 Zbl = {1131.60088}
}

@book{monge,
    author = {Monge, Gaspard},
    title = {Application de l'analyse à la géometrie},
    publisher = {Bachelier},
    year = {1850},
    note = {Note VI: Extension au cas des trois dimensions de la question du tracé géographique, Joseph Liouville}
}

@article{POLCHINSKI1988226,
title = {Scale and conformal invariance in quantum field theory},
fjournal = {Nuclear Physics B},
journal = {Nucl. Phys. B},
volume = {303},
number = {2},
pages = {226-236},
year = {1988},
issn = {0550-3213},
doi = {https://doi.org/10.1016/0550-3213(88)90179-4},
url = {https://www.sciencedirect.com/science/article/pii/0550321388901794},
author = {Joseph Polchinski}
}

@article{zbMATH06398901,
 author = {El-Showk, Sheer and Paulos, Miguel F. and Poland, David and Rychkov, Slava and Simmons-Duffin, David and Vichi, Alessandro},
 title = {Solving the 3d {Ising} model with the conformal bootstrap. {II}: {{\(c\)}}-minimization and precise critical exponents},
 fjournal = {Journal of Statistical Physics},
 journal = {J. Stat. Phys.},
 issn = {0022-4715},
 volume = {157},
 number = {4-5},
 pages = {869--914},
 year = {2014},
 language = {English},
 doi = {10.1007/s10955-014-1042-7},
 url = {www.osti.gov/biblio/1599339},
 zbMATH = {6398901},
 Zbl = {1310.82013}
}

@article{zbMATH07898720,
 author = {Chang, Yinshan and Du, Hang and Li, Xinyi},
 title = {Percolation threshold for metric graph loop soup},
 fjournal = {Bernoulli},
 journal = {Bernoulli},
 issn = {1350-7265},
 volume = {30},
 number = {4},
 pages = {3324--3333},
 year = {2024},
 language = {English},
 doi = {10.3150/23-BEJ1716},
 url = {projecteuclid.org/journals/bernoulli/volume-30/issue-4/Percolation-threshold-for-metric-graph-loop-soup/10.3150/23-BEJ1716.full},
 zbMATH = {7898720},
 Zbl = {1550.60054}
}

@article{zbMATH08035669,
 author = {Cai, Zhenhao and Ding, Jian},
 title = {One-arm exponent of critical level-set for metric graph {Gaussian} free field in high dimensions},
 fjournal = {Probability Theory and Related Fields},
 journal = {Probab. Theory Relat. Fields},
 issn = {0178-8051},
 volume = {191},
 number = {3-4},
 pages = {1035--1120},
 year = {2025},
 language = {English},
 doi = {10.1007/s00440-024-01295-z},
 zbMATH = {8035669},
 Zbl = {1564.60037}
}

@misc{arXiv:2508.09291,
 author = {Vogel, Quirin},
 title = {$1+o(1)$ asymptotics for loop percolation in five and higher dimensions},
 year = {2025},
 howpublished = {Preprint, {arXiv}:2508.09291},
 url = {https://arxiv.org/abs/2508.09291},
 arXiv = {arXiv:2508.09291}
}
}
\end{document}